\documentclass[10pt,leqno]{myart}
\usepackage{amsmath}
\usepackage{amsfonts}
\usepackage{amsthm}
\usepackage{indentfirst}
\usepackage{enumitem} %Anpassbare Enumerates/Itemizes
\usepackage{caption}
\usepackage{mathrsfs}

\usepackage{verbatim} %environment for "comment"
\usepackage{graphicx}

\usepackage{tikz}
\usetikzlibrary{arrows} % LATEX and plain TEX when using Tik Z
\usetikzlibrary{calc}
\usepackage[absolute,overlay]{textpos}
\usepgfmodule{plot}

\newtheorem{theorem}{Theorem}[section]
\newtheorem{lemma}[theorem]{Lemma}
\newtheorem{proposition}[theorem]{Proposition}
\newtheorem{corollary}[theorem]{Corollary}
\theoremstyle{definition}
\newtheorem{definition}[theorem]{Definition}
\newtheorem{remark}[theorem]{Remark}
\newtheorem{example}{Example}[section]

\numberwithin{equation}{section} %FUER ENDVERSION !!
\renewcommand{\d}[1][x]{\,\operatorname{d}\!#1}
\newcommand{\ddt}{\frac{\d[]}{\d[t]}}

\newcommand{\B}{\mathbf{B}}
\newcommand{\E}{\mathbf{E}}
\newcommand{\CC}{\mathbf{C}}
\newcommand{\I}{\mathbf{I}}
\newcommand{\K}{\mathbf{K}}
\renewcommand{\P}{\mathbf{P}}
\newcommand{\Q}{\mathbf{Q}}
\newcommand{\RR}{\mathbf{R}}
\newcommand{\U}{\mathbf{U}}
\newcommand{\X}{\mathbf{X}}
\newcommand{\Y}{\mathbf{Y}}
\newcommand{\0}{\mathbf{0}}
\newcommand{\intRd}{\int_{\R^d}}
\newcommand{\finf}{f_\infty}
\newcommand{\fu}{\frac f{f_\infty}}
\renewcommand{\div}{\operatorname{div}}
\newcommand{\curl}{\operatorname{curl}}
\newcommand{\eps}{\varepsilon}
\newcommand{\rank}{\operatorname{rank}}

\newcommand{\Rho}{\mathrm{P}\!}
\newcommand{\x}{\mathbf{x}}
\newcommand{\kk}{\mathbf{k}}
\newcommand{\bfxi}{\boldsymbol{\xi}}
\newcommand{\ran}{\operatorname{ran}}
\newcommand{\la}{\langle}
\newcommand{\ra}{\rangle}
\newcommand{\md}{\,\mathrm{d}}
\newcommand{\e}{\mathrm{e}}
\newcommand{\ii}{{\mathrm{i}}}
\newcommand{\C}{\mathbb{C}}
\newcommand{\R}{\mathbb{R}}
\newcommand{\D}{{\mathbf{D}}}
\newcommand{\maD}{{\mathbf{C}}}
\newcommand{\cl}{\operatorname{cl}}
\newcommand{\LL}{{\mathcal L}}
\newcommand{\N}{\mathbb{N}}
\newcommand{\HH}{\mathcal H}
\newcommand{\F}{\mathcal F}

\newcommand{\spn}{\operatorname{span}}
\newcommand{\esssup}{\operatorname*{ess\,sup}}
\newcommand{\Imag}{\operatorname{Im}}
\newcommand{\Real}{\operatorname{Re}}

\newcommand{\dd}{{\mathbf{c}}}
\newcommand{\nn}{|{}\hspace{-0.5mm}{}|{}\hspace{-0.5mm}{}|}

\newcommand{\ild}{\int\limits_{\R^d}}

\newcommand{\abs}[1]{|#1|}
\newcommand{\firstThreeConditions}{\ref{C:detP}--\ref{C:Q11} }
\newcommand{\allConditions}{\ref{C:detP}--\ref{C:Q22} }

\newcommand{\detR}{\delta(\kappa,\gamma)}
\newcommand{\detRo}{\delta(\kappa,\gamma_0)}
\newcommand{\detRko}{\delta(\kappa_0,\gamma)}
\newcommand{\detRoo}{\delta(\kappa_0,\gamma_0)}
\newcommand{\detRR}[2]{\delta(#1,#2)}

\DeclareMathOperator{\diag}{diag}
\newcommand{\diff}[2]{\frac{\partial {#1}}{\partial {#2}}}
\newcommand{\Diff}[2]{\frac{\d[#1]}{\d[#2]}}

\newcommand{\norm}[1]{\| {#1} \|}

\newcommand{\Rno}{\R^-_0}
\newcommand{\Rp}{\R^+}
\newcommand{\Rpo}{\R^+_0}

\DeclareMathOperator{\tr}{Tr}

\newcommand{\vektor}[1]{{#1}}

\newcommand{\xx}[1]{\,\text{ #1 }\,}
\newcommand{\Xx}[1]{\quad\text{ #1 }\,}
\newcommand{\XX}[1]{\quad\text{ #1 }\quad}

\newcommand{\Id}{\mathbf{I}} %{\mathrm{Id}}

\begin{document}

\author[F. Achleitner]{Franz Achleitner} \address{Institute for Analysis and
 Scientific Computing, Technical University Vienna, Wiedner
 Hauptstra\ss{}e 8, A-1040 Vienna}
\email{franz.achleitner@tuwien.ac.at}

\author[A. Arnold]{Anton Arnold} \address{Institute for Analysis and
 Scientific Computing, Technical University Vienna, Wiedner
 Hauptstra\ss{}e 8, A-1040 Vienna}
\email{anton.arnold@tuwien.ac.at}

\author[D. St\"urzer]{Dominik St\"urzer} \address{Institute for Analysis and
 Scientific Computing, Technical University Vienna, Wiedner
 Hauptstra\ss{}e 8, A-1040 Vienna}
\email{dominik.stuerzer@tuwien.ac.at}

\title{ Large-time behavior in non-symmetric Fokker-Planck equations }

%----- first page ---------- footnote:
\renewcommand{\thefootnote}{\fnsymbol{footnote}}
\footnotetext{
%% if necessary :
This research was partially supported by the FWF-doctoral school ``Dissipation and dispersion in nonlinear partial differential equations'' and INDAM - GNFM from Italy. One author (AA) is grateful to J.~Sch\"oberl for very helpful discussions.}
\renewcommand{\thefootnote}{\arabic{footnote}}
\setcounter{footnote}{0}

\begin{abstract}
We consider three classes of linear non-symmetric Fokker-Planck equations
having a unique steady state % stationary solution
 and establish exponential convergence of solutions towards the steady state with explicit (estimates of) decay rates.
First, ``hypocoercive'' Fokker-Planck equations are degenerate parabolic equations
 such that the entropy method to study large-time behavior of solutions has to be modified.
We review a recent modified entropy method (for non-symmetric Fokker-Planck equations with drift terms that are linear in the position variable).
Second, kinetic Fokker-Planck equations with non-quadratic potentials 
 are another example of non-symmetric Fokker-Planck equations. Their drift term is \emph{nonlinear} in the position variable.
In case of potentials with bounded second-order derivatives,
 the modified entropy method allows to prove exponential convergence of solutions to the steady state.
In this application of the modified entropy method
 symmetric positive definite matrices solving a matrix inequality are needed.
We determine all such matrices achieving the optimal decay rate in the modified entropy method.
In this way we prove the optimality of previous results.
%This proves that the previous construction of a symmetric positive definite matrix already achieved
% the best decay rates via the modified entropy method. 
Third, we discuss the spectral properties of Fokker-Planck operators perturbed with convolution operators.
For the corresponding Fokker-Planck equation we show existence and uniqueness of a stationary solution. 
Then, exponential convergence of all solutions towards the stationary solution is proven with an uniform rate.
\end{abstract}

\maketitle

% version 20.3.2015, submitted to Rev Parma (1+2. Version)
% version 7.5.2015, Anton
% version 13.5.2015, Franz
% version 18.5.2015, Franz, submitted to Rev Parma
% version 21.7.2015, Anton, 2 typos found in proofs
% version 23.7.2015, Anton, 2 further modifications

\section{Introduction}\label{S1}

Fokker-Planck equations (FPEs) describe the deterministic evolution of the probability density associated to many stochastic processes \cite{RiFP89}. Hence, they constitute an important class of models in applied mathematics and an interesting object of study in the analysis of PDEs. This paper is concerned with the large time analysis of FPEs. In particular we shall analyze non-symmetric equations (corresponding to irreversible stochastic processes). We shall analyze the existence of unique (normalized) steady states and, in particular, the convergence of the time dependent solutions towards it. Here, the main emphasis will be put on the derivation of explicit exponential decay rates (or, at least, estimates of it). Apart from an intrinsic mathematical interest in such decay rates, they are even relevant for the modeling of industrial processes, like the fiber lay-down processes in technical textile production (cf. \cite{KST14}).

For linear, symmetric FPEs the sharp exponential decay rate equals the spectral gap of the generator of the evolution. But, apart from simple examples, exact values or good estimates of this spectral gap are rarely available. Based on the work of Bakry and \'Emery on diffusion processes \cite{BaEmH84, BaEmD85}, the \emph{entropy method} for PDEs has become an important tool to study the large time behavior of wide classes of parabolic equations \cite{ArMaToUn01, ArCaJuL08, ACDDJLMTV04}. The success of this approach is mainly due to its robustness to nonlinear perturbations and extensions \cite{DeVi01, CJMTU01}. More recently it was even generalized to degenerate parabolic equations \cite{ViH06,DoMoScH10, ArEr14}. 

In this paper we shall focus on the large-time behavior of three classes of linear, non-symmetric FPEs: In \S\ref{S2} we shall consider non-symmetric FPEs with drift terms that are linear in the position variable. The recent interest in these equations originated actually in developing entropy methods for the subclass of degenerate diffusions equations, or more precisely ``hypocoercive'' equations. But it turned out that this method can be viewed more naturally for non-symmetric FPEs.
The material of this chapter will be based on the recently developed entropy method from \cite{ArEr14}. We shall present a review from an updated point of view and include several typical examples to illustrate this new method.

In \S\ref{S3} we shall analyze kinetic FPEs with non-quadratic potentials. Again, they are non-symmetric FPEs, but with a drift term that is nonlinear in the position variable. This will illustrate that the entropy method from \S\ref{S2} can be applied also beyond equations with linear drift, at least in perturbative settings. The material of this chapter is an improvement of \S7 in \cite{ArEr14}.

\S\ref{S4} will be concerned with FPEs with non-local perturbations. These perturbations will again render the evolution generator non-symmetric in an appropriately weighted $L^2$--space. But, surprisingly, a wide class of non-local perturbations does not modify the spectrum of the underlying (standard symmetric) FPE. Hence, we shall use spectral methods for the large-time analysis of such models. The material of this chapter is an extension of \cite{StAr14} to FPEs with diffusion and drift matrices that are not the identity.

%%%%%%%%%%%%%%%%%%%%%%%%%%%%%%%%%%%%%%%%%%%%%%%%%%%%%%%%%%%%%%%%%%%%%%%%%%%%%%%%%%%%%%%%%%%%%%%%%
%%%%%%%%%%%%%%%%%%%%%%%%%%%%%%%%%%%%%%%%%%%%%%%%%%%%%%%%%%%%%%%%%%%%%%%%%%%%%%%%%%%%%%%%%%%%%%%%%
\section{Hypocoercive and non-symmetric Fokker-Planck equations}\label{S2}

In this chapter we shall study the evolution of a function $f(t,x);$ $t\ge0,\,x\in\R^d$, under the
linear FPE of the form  
\begin{align}
 \label{linmasterequ}  \partial_t f &= Lf := \div (\D \nabla f + \CC x\ f)\,,%=\div(\D\cdot\nabla f)+x^T\cdot\CC^T\cdot\nabla f + \tr(\CC)f\,,
\end{align}
and subject to the initial condition $f(t=0)=f_0$. Without restriction of generality we assume that 
$$
f_0\ge0\,,\quad \intRd f_0 \d =1\,.
$$
We stipulate that solutions satisfy $f(t,\cdot)\in L^1(\R^d)$. Hence, the divergence form of \eqref{linmasterequ} implies $\int_{\R^d} f(t,x) \d =1$ for all $t>0$.
In \eqref{linmasterequ}, the diffusion matrix $\D\in\R^{d\times d}$ is symmetric and positive semi-definite, and $\CC\in\R^{d\times d}$ is the drift matrix. Both are constant in space and time.

An important model of this class is the kinetic FPE from plasma physics \cite{RiFP89,Vi02}. The time evolution of the phase space probability density $f(t,x,v)$ is governed by:
\begin{align}
 \label{kinFP}  
 \partial_t f +v\cdot\nabla_x f-\nabla_x V\cdot\nabla_v f&= \nu \div_v (vf)+\sigma\Delta_v f\,;
 \quad x,\,v\in\R^n;\,t>0\,.
\end{align}
Here, the position-velocity vector $(x,v)$ plays the role of $x\in\R^d,\,d=2n$, in \eqref{linmasterequ}. $\nu,\sigma$ denote (positive) friction and diffusion parameters, respectively. $V=V(x)$ is a given confinement potential for the system. Next we rewrite \eqref{kinFP} as 
\begin{eqnarray}\label{kinFP:nonsymmFP} 
      f_t & = & \div_{x,v}\left[\left(
\begin{array}{cc}
 \0 & \0 \\
 \0 & \sigma\,\I
\end{array}
 \right) \nabla_{x,v} f + \left(
\begin{array}{c}
 -v \\
 \nabla_x V+\nu\,v
\end{array}
 \right) \,\right]\,.
\end{eqnarray}
Here, the first matrix is a singular diffusion matrix, with the identity matrix $\I\in\R^{n\times n}$; the second term is the drift. For a quadratic potential $V$, the kinetic FPE \eqref{kinFP} takes exactly the form of \eqref{linmasterequ} and its analysis will be covered in \S\ref{S2}. The case of non-quadratic potentials is the subject of \S\ref{S3}.\\

The goal of this chapter is first to identify (under appropriate assumptions on $\D$ and $\CC$) the unique normalized steady state $\finf(x)$ of \eqref{linmasterequ}. Most of all, we shall then study the convergence of $f(t)$ to $\finf$ as $t\to\infty$ with (possibly sharp) exponential rates. 
In view of space limitations we shall mostly present only formal computations, which hold rigorously for regular enough solutions. But, anyhow, parabolic and hypocoercive FPEs regularize instantaneously to $C^\infty$ (cf.\ Proposition \ref{prop1} below). So,
regularity is actually not an issue, with the possible exemption at the initial time.

%%%%%%%%%%%%%%%%%%%%%%%%%%%%%%%%%%%%%%%%%%%%%%%%%%%%%%%%%%%%%%%%%%%%%%%%%%%%%%%%%%%%%%%%%%%%%%%%
\subsection{Non-symmetric Fokker-Planck equations}

In this section we introduce the notion of (non)symmetric FPEs and the relative entropy, which will be our main tool to analyze the large-time behavior below. For these definitions we first consider FPEs with $x$--dependent coefficients. A \emph{symmetric Fokker-Planck equation} is defined to be of the form
\begin{equation}\label{symmFP}
  \partial_t f = L_1f := \div \big(\D(x) [\nabla f + f\nabla A(x)]\big),
\end{equation}
with a  diffusion matrix $\D$ that is locally uniformly positive definite on $\R^d$ and symmetric. 
We assume that the sufficiently regular confinement potential $A$ satisfies $\e^{-A}\in L^1(\R^d)$. Then $\finf:=\e^{-A}$ is the unique normalized steady state of \eqref{symmFP}. The normalization $\intRd \finf(x) \d=1$ is imposed here by changing the additive constant of $A$, which is not prescribed by \eqref{symmFP}. 
The non-degeneracy of the ground state of $L_1$ can easily be seen from the following computation:
\begin{equation}\label{quadratic-form}
  \langle L_1f,g\rangle_H 
  = -\intRd\nabla^T\big(\frac{f}{\finf}\big)  \D(x) \nabla\big(\frac{g}{\finf}\big)\ \finf\d\,,
\end{equation}
with $\langle \cdot,\cdot \rangle_H$ denoting the inner product in the weighted $L^2$--space $H := \linebreak L^2(\R^d,\finf^{-1})$.
And the right hand side of \eqref{quadratic-form}, with $f=g$, vanishes iff $f/\finf=const$.
The quadratic form \eqref{quadratic-form} also shows that the operator $L_1$ is symmetric in $H$ (cf.\ \S2 of \cite{ArMaToUn01} for more details).

The key feature of a symmetric FPE is the gradient form of its drift vector field. 
For $d\ge2$ we shall now consider more general drift fields, which will make the evolution generator non-symmetric in $H$. For regular diffusion matrices $\D(x)$, the following equation is called a \emph{non-symmetric Fokker-Planck equation}:
\begin{equation}\label{nonsymmFP}
  \partial_t f = L_2 f := \div \big(\D(x) [\nabla f + f\{\nabla A(x)+F(x)\}]\big)\,.
\end{equation}
Here we assume that the additional vector field $F$ satisfies
\begin{equation}\label{div-free}
  \div(\D(x)  F(x)\ \finf(x))=0\,,\quad \forall\,x\in\R^d\,,
\end{equation}
such that $\finf=\e^{-A}$ is still the unique steady state of \eqref{nonsymmFP}. 
%The uniqueness of the steady state is actually the only aspect in this section, for which we requested $\D(x)$ to be positive definite. Otherwise, its positive semi-definiteness would suffice.

In typical applications, however, the FPE is given with just one drift vector field that is not yet split into two summands (in contrast to \eqref{nonsymmFP}). In order to retrieve the steady state, this field then needs to be decomposed into a gradient part and a divergence free part (in the above sense). This task is a generalization of the Helmholtz-Hodge decomposition (see \S2 in \cite{ArCaMa10} for a typical example).
Such a decomposition of the vector field readily yields the following decomposition of the operator $L_2$ into its symmetric part $L_s$ in $H$ and its anti-symmetric part $L_{as}$: $L_2=L_s+L_{as}$ with
\begin{align*}\label{L-decomp1} 
      L_s f & =  \div \big(\D(x) [\nabla f + f\nabla A(x)]\big)\,,\\
      L_{as} f & =  \div \big(\D(x)  F(x)\ f\big)\,.
\end{align*}
Due to \eqref{div-free} we have $L_s\finf=L_{as}\finf=0$.

Next we give a more compact form of $L_s$ and $L_{as}$, which of course also holds for $L_1$ with $F\equiv0$.
\begin{lemma}\label{le-L-decomp}
Let $\D(x)>\0$, assume condition \eqref{div-free}, and let $f_\infty=\e^{-A}$ denote the steady state of \eqref{nonsymmFP}.
The symmetric/anti-symmetric decomposition of $L_2$ then satisfies:
\begin{eqnarray}
      L_s f \!\!\!\! & = & \!\!\!\! \div\Big(\finf \ \D(x)  \nabla\frac{f}{\finf}\Big)\,,\label{L-decomp2a} \\
      L_{as} f \!\!\!\! & = & \!\!\!\! \div\Big(\finf \ \RR(x)  \nabla\frac{f}{\finf}\Big)\,,\label{L-decomp2b} 
\end{eqnarray}
where the matrices $\RR(x)\in\R^{d\times d}$ are skew-symmetric and %=(r_{ij})
satisfy on $\R^d$:
\begin{equation}\label{R-cond}
  \nabla^T (\RR \finf) = G^T := (\D F\ \finf)^T \,.
\end{equation}
\end{lemma}
\begin{proof}
\eqref{L-decomp2a} is trivial, so we only discuss \eqref{L-decomp2b}.
The divergence-free-condition \eqref{div-free} on the vector field $G$ implies that there exists a  matrix function $\B(x)\in\R^{d\times d}$, with $\B(x)$ skew-symmetric and
\begin{equation}\label{div-free-rep}
  G^T (x)=\nabla^T \B(x)\,.
\end{equation} 
Let us briefly illustrate this statement:
For $d=2,\,3$ \eqref{div-free-rep} simplifies to the well known representations of divergence free  vector fields:
$$
  \B(x)=\left(
\begin{array}{cc}
 0 & -b(x) \\
 b(x) & 0
\end{array}
 \right)\,,\qquad G=\nabla^\perp b\,,\qquad 
 \nabla^\perp:=\left(
\begin{array}{c}
 \partial_{x_2} \\
 -\partial_{x_1}
\end{array}
 \right)\,,
$$
and, respectively,
$$
  \B(x)=\left(
\begin{array}{ccc}
 0 & -b_3(x) & b_2(x)\\
 b_3(x) & 0 & -b_1(x)\\
 -b_2(x) & b_1(x) & 0
\end{array}
 \right)\,,\qquad G=\curl \left(\begin{array}{c}
 b_1(x)\\
 b_2(x)\\
 b_3(x)
\end{array} \right)\,.
$$
In higher dimensions, \eqref{div-free-rep} can be verified either with differential forms (cf.\ \S6 in \cite{AFW06}, \cite{CMcI10}) or by Fourier transformation.

Next we compute
$$
  L_{as}f:=\div\Big( G\,\frac{f}{\finf}\Big) = \div\Big( (\nabla^T \B)^T \frac{f}{\finf}\Big) =
  (\nabla^T  \B)  \nabla \frac{f}{\finf} = \div\Big( \B \nabla\frac{f}{\finf}\Big)\,,
$$
where we have used the skewness of $\B$ in the last two steps. Setting $\RR:=\B\finf^{-1}$ yields 
\eqref{L-decomp2b}.
\end{proof}
Note that $L_{as}$ in \eqref{L-decomp2b} is only a first order operator -- due to the skew-symmetry of $\RR(x)$.

\bigskip
As mentioned above, the main goal of this chapter is to study the convergence to the equilibrium for solutions to non-symmetric FPEs. To this end, our main tool will be the relative entropy. We define (see \S2.2 of \cite{ArMaToUn01} for more details):
\begin{definition}
\begin{enumerate}
  \item [(a)] Let $J$ be either $\R^+$ or $\R$. A scalar function $\psi\in C(\bar J)\cap C^4(J)$ satisfying the conditions
  \begin{equation}\label{entropy-generator}
    \psi(1)=0\,,\quad\psi\ge0\,,\quad \psi''>0\,,\quad (\psi''')^2\le\frac12\psi''\psi^{IV}\;\mbox{ on } J
  \end{equation}
  is called \emph{entropy generator}.
  \item [(b)]
  Let $f_1\in L^1(\R^d)$, $f_2\in L^1_+(\R^d)$ with $\intRd f_1\d=\intRd f_2\d=1$ and $\frac{f_1}{f_2}(x)\in \bar J$ a.e.\ (w.r.t.\ the measure $f_2(\d)$). Then
  \begin{equation}\label{rel-entropy}
    e_\psi(f_1|f_2):=\intRd \psi\Big(\frac{f_1}{f_2}\Big) f_2 \d\ge0
  \end{equation}
  is called an \emph{admissible relative entropy} of $f_1$ with respect to $f_2$ with generating function $\psi$.
\end{enumerate}
\end{definition}

In this definition, the term ``admissible'' refers to the applicability of the \emph{entropy method} under the assumptions \eqref{entropy-generator}.
The most important examples are the logarithmic entropy $e_1(f_1|f_2)$, generated by
$$
  \psi_1(\sigma)=\sigma\ln\sigma-\sigma+1\,,
$$
and the power law entropies $e_p(f_1|f_2)$ with $1<p\le2$,  generated by
$$
  \psi_p(\sigma)=\sigma^p-1-p(\sigma-1)\,.
$$
Except for quadratic entropies $e_{\psi_2}$ we shall always use $J=\R^+$.

The above definition clearly shows that $e_\psi(f_1|f_2)=0$ iff $f_1=f_2$. In the subsequent sections we shall hence try to prove that solutions $f(t)$ to FPEs satisfy $e_\psi(f(t)|f_\infty)\to0$ as $t\to\infty$. Such a convergence in relative entropy then also implies $L^1$--convergence, due to the \emph{Csisz\'ar-Kullback inequality}:
$$
  \|f_1-f_2\|_{L^1(\R^d)}^2 \le \frac{2}{\psi''(1)}\,e_\psi(f_1|f_2)\,.
$$

This relative entropy (w.r.t.\ the steady state) is a Lyapunov functional for the evolution. As proved in \S2.4 of \cite{ArMaToUn01} we have:
\begin{lemma}\label{le-e-decay}
Let $f(t)$ be a solution to the non-symmetric FPE \eqref{nonsymmFP} with the divergence-free-condition \eqref{div-free}. Then,
%\begin{eqnarray}\label{Fisher-info}
\begin{align}\label{Fisher-info}
  \ddt e_\psi(f(t)|f_\infty) &=
  -\intRd \psi''\Big(\frac{f(t)}{\finf}\Big)\;
  \Big(\nabla \frac{f(t)}{\finf}\Big)^T \D(x) \Big(\nabla \frac{f(t)}{\finf}\Big)\;\finf\d\nonumber\\
  &=:-I_\psi(f(t)|f_\infty)\le0\,,
\end{align}
%\end{eqnarray}
where $I_\psi(f(t)|f_\infty)$ denotes the \emph{Fisher information} (of $f(t)$ w.r.t.\ $\finf$).
\end{lemma}
We remark that the right hand side of \eqref{Fisher-info} is independent of the vector field $F$, i.e.\ independent of $L_{as}$. 
So, for a fixed time $t$,
 the relative entropy and its entropy dissipation coincide for a non-symmetric FPE and its corresponding symmetric FPE.

%%%%%%%%%%%%%%%%%%%%%%%%%%%%%%%%%%%%%%%%%%%%%%%%%%%%%%%%%%%%%%%%%%%%%%%%%%%%%%%%%%%%%%%%%%%%%%%%
\subsection{Hypocoercive Fokker-Planck equations}

In this section we shall define hypocoercivity and give some typical examples. We start with the standard FPE on $\R^d$:
\begin{equation}\label{standardFP}
  \partial_t f=\div(\nabla f+xf)=:L_3 f
\end{equation}
with the unique normalized steady state
\begin{equation}\label{standardGaussian}
  \finf(x)=(2\pi)^{-\frac{d}{2}}\e^{-\frac{|x|^2}{2}}\,.
\end{equation}
As seen from \eqref{quadratic-form}, the operator $L_3$ is symmetric on $H:=L^2(\finf^{-1})$ and dissipative, i.e. $\langle L_3 f,f\rangle_H\le0\;\,\forall\,f\in \mathcal D(L_3)$. Also, $-L_3$ is \emph{coercive} in the sense that
$$
    \langle -L_3f,f\rangle_H\ge\|f\|^2_H\,, %{L^2(f_\infty^{-1})}\,,
    \quad\forall \,f \in\{f_\infty\}^\perp\,.
$$
In other words, $-L_3$ has a spectral gap of size 1 (since 0 and 1 are the lowest eigenvalues of $-L_3$) and this spectral gap determines the sharp exponential decay of solutions towards $\finf$:
\begin{equation}\label{FP-decay}
    \|\e^{L_3t}f_0-\finf\|_{H} \le \e^{- t }\|f_0-\finf\|_{ H}\,,\qquad \forall\,f_0\in H \;
    \mbox{ with } \intRd f_0\d=1;
    \quad t\ge0\,.
\end{equation}
Equilibration occurs here as a balance between diffusion and drift in \eqref{standardFP}.

Next we consider the FPEs from \eqref{linmasterequ}: 
$$
  \partial_t f=\div(\D \nabla f+\CC  x\ f)=Lf\,.
$$
For a singular diffusion matrix $\D$ this equation is degenerate parabolic, and the operator $L$ is not coercive in $L^2(f_\infty^{-1})$, where the steady state $\finf$ will be specified in \S\ref{S23} below. This non-coercivity can be seen easily from \eqref{L-decomp2a}, when choosing $f(x)=c\cdot x\,\finf(x)$ with a vector $c\in\ker \D$. 

In spite of this lack of coercivity, such degenerate FPEs will frequently still exhibit an exponential convergence to equilibrium. This motivated C.\ Villani to coin the term \emph{hypocoercivity} in \cite{ViH06}. The following definition is very general, but afterwards we shall only be concerned with FPEs of type \eqref{linmasterequ}.
\begin{definition}\label{hypocoercive}
Let $H$ be a Hilbert space.
Consider a linear operator $L$ on $H$ generating a $C_0$-semigroup $(\e^{Lt})_{t\geq 0}$. 
Also, consider a (smaller) Hilbert space $\tilde H$
 that is continuously and densely embedded in the orthogonal complement of $\mathcal K:=\ker \,L \subset H$
 (i.e. $\tilde H\hookrightarrow \mathcal K^\perp$). 
Then, $-L$ is called \emph{hypocoercive} on $\tilde H$ if there exist constants $c\ge1$ and $\lambda>0$ such that
\begin{equation}\label{hypocoercive-decay}
    \|\e^{Lt}f\|_{\tilde H} \le c \,\e^{-\lambda t }\|f\|_{\tilde H}\,,\qquad \forall\,f\in \tilde H;
    \quad t\ge0\,.
\end{equation}
\end{definition}

In many applications to FPEs, $H$ is a weighted $L^2$--space, and $\tilde H$ a weighted $H^1$--space. In \eqref{hypocoercive-decay} we shall typically have a leading multiplicative constant $c>1$, while this constant is 1 in the symmetric, non-degenerate case of \eqref{FP-decay}.

\bigskip
Next we shall give some typical examples of such hypocoercive equations in order to explain the convergence mechanism.
\begin{example}\label{Ex0}
The kinetic FPE \eqref{kinFP} is non-symmetric. With a sufficiently growing confinement potential $V(x)$ it is hypocoercive, and its steady state is
$$
  \finf(x,v)=c\,\e^{-\frac{\nu}{\sigma}\big[\frac{|v|^2}{2}+V(x)\big]}\,,
$$
with some normalization constant $c>0$. \hfill $\square$ \\
\end{example}
%\medskip

\begin{example}\label{Ex1}
Next we consider the following degenerate 2D equation of form \eqref{linmasterequ}:
\begin{align}\label{FP-Ex1}
      \partial_t f & =  \div\Big[\left(
\begin{array}{cc}
 1 & 0 \\
 0 & 0
\end{array}
 \right)  \nabla f + \left(
\begin{array}{cc}
 1 & -\omega\\
 \omega & 0
\end{array}
 \right)   x\,f\Big] =: L_4f\,.
\end{align}
\begin{comment}
First we consider \eqref{linmasterequ} in $\R^2$ with the diffusion matrix
$$
  \D=\left(
\begin{array}{cc}
 1 & 0\\
 0 & 0
\end{array}
 \right)\,,\quad \mbox{ and the drift matrix }
 \CC=\left(
\begin{array}{cc}
 1 & -\omega\\
 \omega & 0
\end{array}
 \right)\,.
$$
\end{comment}
For any parameter $\omega\in\R$, one easily verifies that the standard Gaussian \eqref{standardGaussian} is still a steady state of \eqref{FP-Ex1}, and for $\omega\ne0$ it is the unique normalized steady state $\finf$. For $\omega=0$ we have drift and diffusion in the $x_1$--direction (as in the standard FPE). But in the $x_2$--direction there is no equilibration.

The term with $\omega$ in \eqref{FP-Ex1} constitutes the rotational part of the drift matrix $\CC$ and the anti-symmetric part of the operator $L_4$. Heuristically speaking, it mixes the diffusive $x_1$--direction with the non-diffusive $x_2$--direction. Hence, for fast enough rotations, the sharp decay rate of solutions to \eqref{FP-Ex1} is the average of the decay rates in the $x_1$-- and $x_2$--directions. More precisely, the drift matrix $\CC$ has the following lower bound on the real parts of its eigenvalues:
\begin{equation}\label{spectrumC}
  \mu:=\min\{\Real(\lambda)\,|\,\lambda\in\sigma(\CC)\} = \frac12\,\qquad \mbox{ for }|\omega|>\frac12\,.
\end{equation}
As we shall show in Section \ref{S24} below, this lower bound determines the sharp decay rate $\frac12$ towards $\finf$. For slower rotations, however, the decay rate approaches zero since 
$\min\{\Real(\lambda)\,|\,\lambda\in\sigma(\CC)\} = \frac12-\sqrt{\frac14-\omega^2}$.

As we shall see in the decay analysis below, the decay behavior can be understood quite well by considering the drift characteristics corresponding to \eqref{FP-Ex1}. They satisfy the ODE--system $x_t=-\CC x$. Along a characteristic, $|x(t)|^2$ is monotonically non-increasing, and at points with $x_1\ne0$ it is even strictly monotonically decreasing. However, when crossing the $x_2$--axis, the characteristic is tangent to the level curves of $|x|^2$ (cf.\ Figure \ref{fig1}). As we shall see below, this implies that the relative entropy (e.g. $e_2(f(t)|\finf)$ ) may have a vanishing time derivative at certain points in time.

%\begin{comment} %ENTFERNEN !!
\begin{figure}[ht!]
\begin{center}
 \includegraphics[scale=.7]{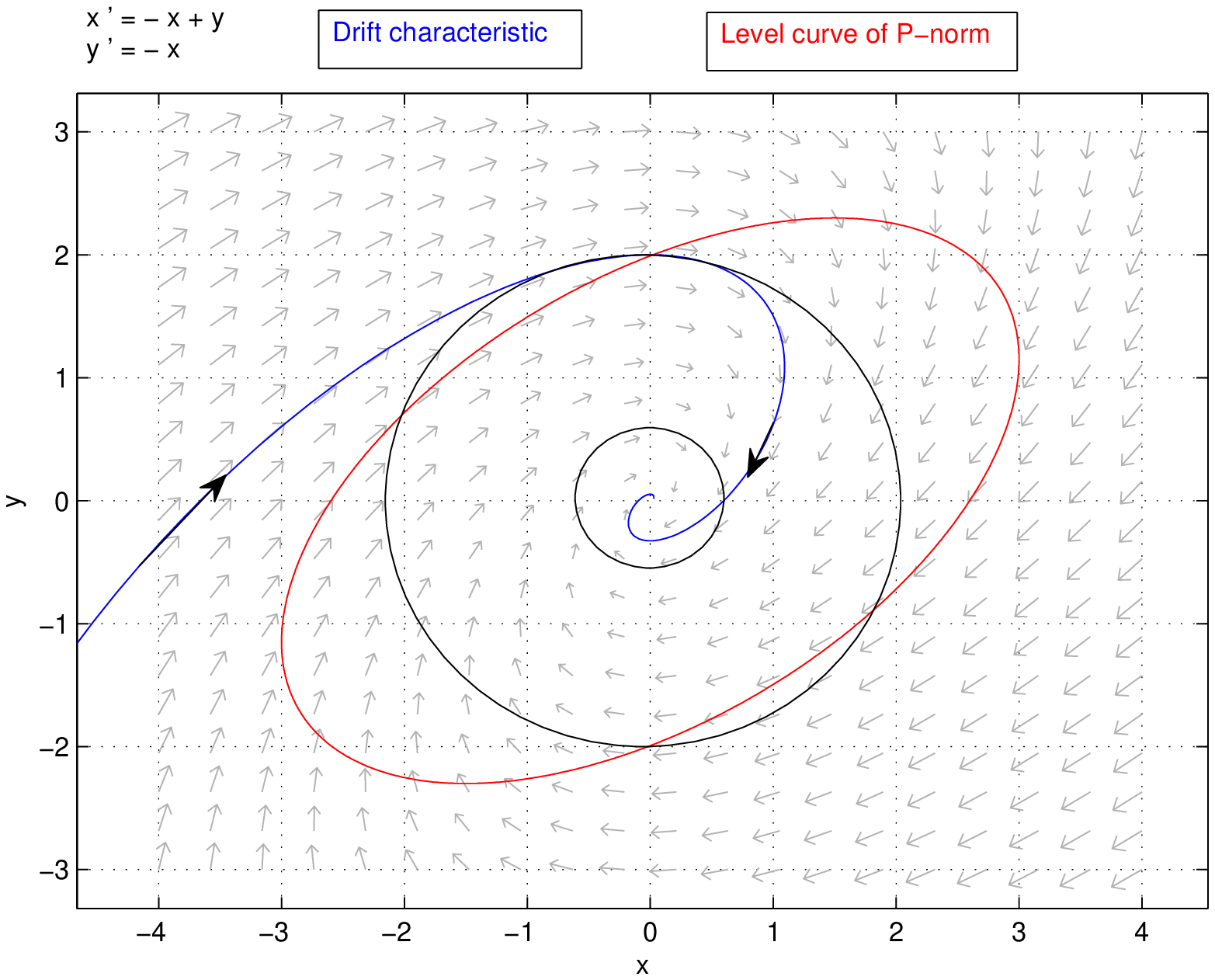}
 \caption{Drift characteristic for the 2D Fokker-Planck equation \eqref{FP-Ex1} 
 with $\D=\diag(1,\;0)$, $\CC=[1 \;\, -1 \; ;\; 1 \;\, 0]$: 
 %$C=\left(\begin{array}{cc} 1/4 & -8 \\ 2 & 1 \end{array}\right)$
 The blue spiral is tangent to the level curves of $|x|$ (black circles) when crossing the $x_2$--axis.
 The red ellipse is a level curve for the ``distorted'' vector norm $\sqrt{\langle x(t),\P x(t)\rangle}$. For this example the 
 optimal metric is given by $\P=[2 \;\, -1 \; ;\; -1 \;\, 2]$, cf.\ Lemma \ref{Pdefinition} for the algorithm how to compute $\P$. (colors only online)\\
 In the labeling of the two axes $x$ means $x_1$ and $y$ means $x_2$.}
 \label{fig1}
\end{center}
\end{figure}
%\end{comment} %ENTFERNEN !!

We now indicate a possibility to obtain a strictly (and uniformly in time) decaying Lyapunov functional for the evolution of \eqref{FP-Ex1}. At the level of drift characteristics it is advantageous to consider  (instead of $|x(t)|$) the ``distorted'' vector norm %$\|x(t)\|_P^2:=
$\sqrt{\langle x(t),\P x(t)\rangle}$ with some appropriate symmetric, positive definite matrix $\P$. This $\P$-norm will allow to realize the optimal decay of $x(t)$ with the rate $\mu$ defined in \eqref{spectrumC} -- uniformly in time (for details, see \eqref{Pnorm-decay} below). This idea is the essence of the strategy followed in \cite{DoMoScH10} for hypocoercive equations. 

To sum up, the essence of this example is a degenerate diffusion and a rotation that mixes the directions.
\hfill $\square$ \\
\end{example}

%\smallskip
\begin{example}\label{FP-Ex2}
Here we consider \eqref{linmasterequ} again on $\R^2$, with the diffusion matrix $\D=\diag(1,\;0)$ and the drift matrix $\CC=[1 \;\, 0 \; ;\; 1 \;\, 1]$. Note that $\CC$ is a (transposed) Jordan block. Hence, the drift characteristics (solving $x_t=-\CC x$) are here degenerate spirals  (cf.\ Figure \ref{fig2}a). The crucial aspect of this example is that the asymptotic direction of these characteristics (close to $x=0$) is not aligned with the diffusive $x_1$--direction. This again allows for equilibration as $t\to\infty$.

%\begin{comment} %ENTFERNEN !!
\begin{figure}[htbp]
\begin{center}
\includegraphics[width=5.5cm]{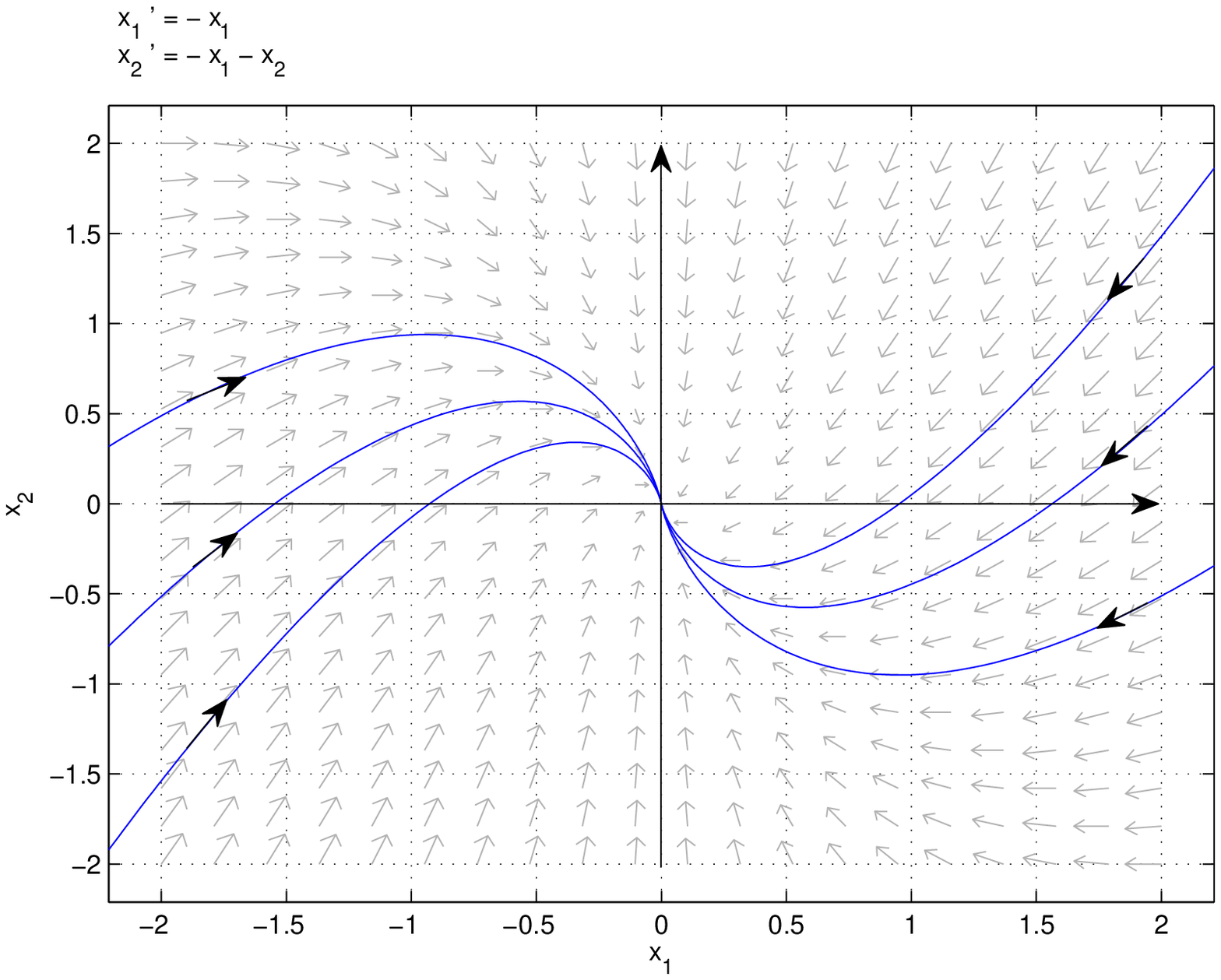}\hfill
\vspace{-0.3cm}
\includegraphics[width=6.5cm]{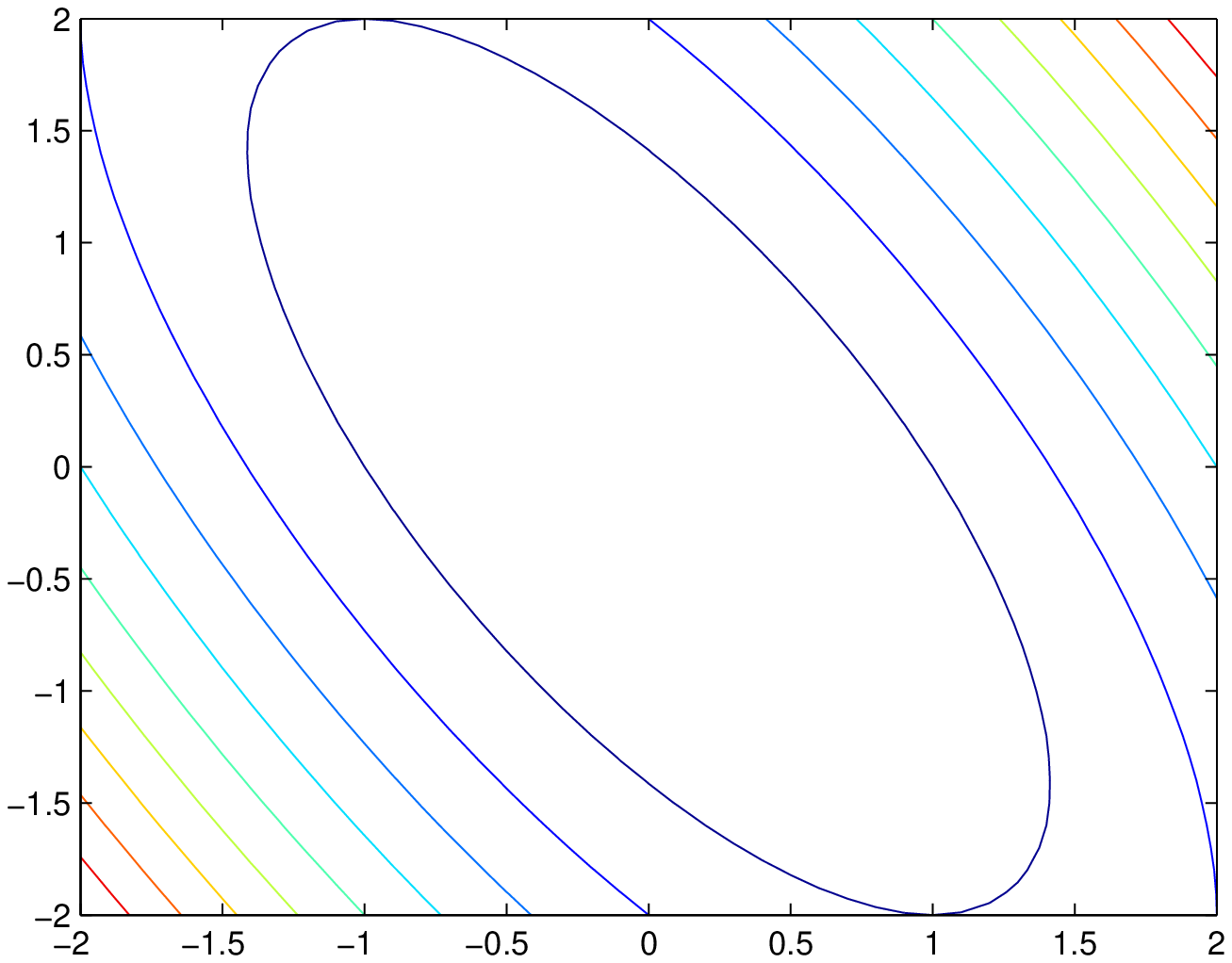}
\end{center}
%\vspace{-0.3cm}
\caption{(a) Left: Drift characteristics and flow vector field for the 2D FPE \eqref{linmasterequ} 
 with $\D=\diag(1,\;0)$, $\CC=[1 \;\, 0 \; ;\; 1 \;\, 1]$. (b) Right: Level curves of the quadratic potential appearing in the steady state, i.e. $A(x)=-\ln \finf(x)$. 
\label{fig2} 
}
\end{figure}
%\end{comment} %ENTFERNEN !!

One easily verifies that the steady state is given by the non-isotropic Gaussian
$$
  \finf(x)=c\,\e^{-(x_1^2+2x_1x_2+2x_2^2)}\,,
$$
with a normalization constant $c$.
The contour lines of the steady state potential are graphed in Figure \ref{fig2}b. Here, the ``sharp'' decay rate is given by $1-\eps$ (where $\min\{\Real(\lambda)\,|\,\lambda\in\sigma(\CC)\} =1$).
\hfill $\square$ \\
\end{example}

%%%%%%%%%%%%%%%%%%%%%%%%%%%%%%%%%%%%%%%%%%%%%%%%%%%%%%%%%%%%%%%%%%%%%%%%%%%%%%%%%%%%%%%%%%%%%%%%
\subsection{Steady states and normalized Fokker-Planck equations}\label{S23}

In the above examples we saw that, to enable convergence to an equilibrium, the drift matrix $\CC$ has to mix the diffusive and non-diffusive directions of the linear FPE \eqref{linmasterequ} (provided $\D$ is singular). Now we give conditions on $\D$ and $\CC$ such that \eqref{linmasterequ} has a unique steady state:

\begin{definition}\label{cond-A}
The operator $L$ from \eqref{linmasterequ} is said to fulfill \emph{condition (A)} if:
\begin{enumerate}
\item[(A1)] No (nontrivial) subspace of $\ker\D$ is invariant under $\CC^T$.
\item[(A2)] The matrix $\CC$ is positively stable (i.e. all eigenvalues have real part greater than zero).
\end{enumerate}
\end{definition}

Condition (A1) is equivalent to the hypoellipticity of $\partial_t-L$ (cf.\ \S1 of \cite{Ho67}). Moreover, it implies regularization and strict positivity of the solution to \eqref{linmasterequ}:
\begin{proposition}\label{prop1}
Let condition (A1) from Definition \ref{cond-A} hold, and let $f_0\in L^1(\R^d)$. 
\begin{enumerate}
\item[a)] Then the unique solution of (\ref{linmasterequ}) satisfies $f\in C^\infty(\R^+\times\R^d)$.
\item[b)] If $f_0\ge0$, we have $f(t,x)>0\;\;\forall t>0,\,x\in\R^d$.
\end{enumerate}
\end{proposition}
\begin{proof} 
For part (a) see page 148 of \cite{Ho67}. Part (b) follows from the strict positivity of the Green's function pertaining to (\ref{linmasterequ}) (see Lemma 2.5 and Theorem 2.7 in \cite{ArEr14}).
%\hfill $\square$ \\
\end{proof}

Condition (A2) means that there is a confinement potential that prevents the solution to run off to $\infty$. Without it, there would be no steady state. 
Indeed, Theorems \ref{ssexistence} and \ref{convergencerate} will show that condition (A) is both sufficient and necessary for the existence of a unique normalized steady state and exponential convergence of solutions towards the steady state. So, for equations of type \eqref{linmasterequ}, hypoellipticity and confinement are equivalent to hypocoercivity.\\

In the following lemma we give three equivalent characterizations of the hypoellipticity of $L$
 that will be used for the regularization of the propagators~$\e^{Lt}$, $t>0$
 (see Theorem \ref{entropyregularisation} below, and \S2 of \cite{MPP02}). 

\begin{lemma}
\label{Definiteness}
 The following three statements are equivalent, where we use $k:=\rank \D\in\{1,\dots,d\}$:
 \begin{itemize}
  \item[(i)] No non-trivial subspace of $\ker \D$ is invariant under $\CC^T$.
  \item[(ii)] There exist constants $\tau\in\{0,\dots,d-k\}$ and $\kappa>0$ such that
\begin{align}
\label{sumTdefinite} \sum\limits_{j=0}^{\tau} \CC^j\D(\CC^T)^j \geq \kappa\I\,.
\end{align}
  \item[(iii)] There exists a constant $\tau\in\{0,\dots,d-k\}$ such that
\begin{align}
\label{rank-cond} \rank[\D^\frac12,\,\CC\D^\frac12,...,\,\CC^\tau\D^\frac12]=d\,.
\end{align}
 \end{itemize}
\end{lemma}
\begin{proof} 
For the equivalence of (i) and (ii) we refer to Lemma 2.3 of \cite{ArEr14}.\\
For (iii)$\Rightarrow$(ii) let 
$$
  \E:=[\D^\frac12,\,\CC\D^\frac12,...,\,\CC^\tau\D^\frac12] \in\R^{d\times(\tau+1)d}\,.
$$
Then, 
$$
  \R^{d\times d}\ni \E\ \E^T = \sum\limits_{j=0}^{\tau} \CC^j\D(\CC^T)^j \ge\0 
$$
has rank $d$ and \eqref{sumTdefinite} follows.\\
For (ii)$\Rightarrow$(iii) assume we had $\rank \E^T<d$. Then, $\exists \,0\ne v\in\R^d$ with 
$\E^T v=0$. Hence, $\E\ \E^T v=0$ would contradict \eqref{sumTdefinite}.
\end{proof}
If $\tau$ is the minimal constant for which \eqref{sumTdefinite} (or, equivalently, \eqref{rank-cond}) holds, then $L$ fulfills \emph{H\"ormander's finite rank bracket condition} of order $\tau$ (see \cite{Ho67}, Theorem 1.1). For the explicit decomposition of the generator from \eqref{linmasterequ} in the H\"ormander form $-L=A^*A+B$ we refer to Proposition 5 in \cite{ViH06}. \\

As an illustration we consider the following two hypocoercive examples of \eqref{linmasterequ}, where $d=4,\,k=2$:
\begin{example}\label{FP-Ex3}
Let
\begin{align*}
 \D_1:=\left(\begin{array}{cccc} 1 & 0 & 0 & 0 \\ 0 & 1 & 0 & 0 \\ 0 & 0 & 0 & 0 \\ 0 & 0 & 0 & 0 \end{array}\right); &\quad \CC_1:=\left(\begin{array}{cccc} 1 & 0 & 1 & 0 \\ 0 & 1 & 0 & 1 \\ -1 & 0 & 0 & 0 \\ 0 & -1 & 0 & 0 \end{array}\right)\,.
\end{align*}
Here, $\rank[\D_1^\frac12,\,\CC_1\D_1^\frac12]=4$ and hence $\tau=1$.
\hfill $\square$ \\
\end{example}

\begin{example}\label{FP-Ex4}
Let
\begin{align*}
 \D_2:=\left(\begin{array}{cccc} 1 & 0 & 0 & 0 \\ 0 & 1 & 0 & 0 \\ 0 & 0 & 0 & 0 \\ 0 & 0 & 0 & 0 \end{array}\right); &\quad \CC_2:=\left(\begin{array}{cccc} 1 & 0 & 0 & 0 \\ 0 & 1 & 1 & 0 \\ 0 & -1 & 0 & 1 \\ 0 & 0 & -1 & 0 \end{array}\right)\,.
\end{align*}
Here, $\rank[\D_2^\frac12,\,\CC_2\D_2^\frac12]=3$, 
but $\rank[\D_2^\frac12,\,\CC_2\D_2^\frac12,\,\CC_2^2\D_2^\frac12]=4$. Hence $\tau=2$.
\hfill $\square$ \\
\end{example}

In many works on hypocoercive equations \cite{DoMoScH10, BaB13}, % \cite{DuH09} 
a more restrictive assumption (than condition (A)) is made, namely: ``No subspace of $\ker \D$ should be mapped into $\ker\D$ by $\CC^T$'', which corresponds to the requirement $\tau=1$. Let as reconsider the two previous examples under this aspect. In Example \ref{FP-Ex3}, $\CC^T_1$ maps the non-diffusive directions from $\ker \D_1$ into the diffusive directions from $(\ker \D_1)^\perp$. 
But in Example \ref{FP-Ex4}, $\CC^T_2$ maps the non-diffusive direction $(0,\,0,\,0,\,1)^T\in\ker \D_2$
still onto the non-diffusive direction $(0,\,0,\,1,\,0)^T\in\ker \D_2$. 
But in a second step, we have $\CC^T_2 (0,\,0,\,1,\,0)^T=(0,\,1,\,0,\,-1)^T$,
 which has a nontrivial component in the diffusive subspace $(\ker \D_2)^\perp$. \\

Next we discuss the existence of a steady state to \eqref{linmasterequ} (for the proof cf. \S1 of \cite{MPP02} or Th.\ 3.1 in \cite{ArEr14}):

\begin{theorem}\label{Theorem21}
\label{ssexistence}
There exists a unique steady state $f_\infty\in L^1(\R^d)$ of (\ref{linmasterequ}) fulfilling $\intRd\finf\d=1$ iff condition (A) holds.\\
Moreover, this steady state is of the (non-isotropic) Gaussian form
\begin{align}\label{steady-state}
 f_\infty(x)&=c_K\exp\Big(-\frac{x^T\K^{-1}x}2\Big),
\end{align}
where $\K$ is the unique, symmetric, and positive definite solution to the \emph{continuous Lyapunov equation} 
(cf.\ \cite{HoJoT91})
\begin{align}
\label{Dequ} 2\D=\CC\K+\K\CC^T,
\end{align}
and $c_K=(2\pi)^{-\frac{d}{2}} (\det \K)^{-\frac{1}{2}}$ is the normalization constant.
\end{theorem}

With the steady state at hand, we now give the decomposition of the operator $L$ from \eqref{linmasterequ}:
\begin{lemma}\label{le-L-decomp2}
Let $L$ satisfy condition (A). Then, its symmetric/anti-\linebreak symmetric decomposition satisfies:
%\begin{eqnarray}
\begin{align}
      L_s f & = \div\Big(\finf \D  \nabla\frac{f}{\finf}\Big)\,,\label{L-decomp3a} \\
      L_{as} f & = \div\Big(\finf \RR  \nabla\frac{f}{\finf}\Big)\,,\label{L-decomp3b} 
\end{align}
%\end{eqnarray}
with $\RR:=\frac12(\CC\K-\K\CC^T)\ne\0$.
\end{lemma}
\begin{proof}
To reduce this result to Lemma \ref{le-L-decomp}, we first compare \eqref{linmasterequ} to \eqref{nonsymmFP}:
The drift vector field $\CC  x$ of \eqref{linmasterequ} corresponds to $\D  \{\nabla A+F\}$. Hence, we have with \eqref{Dequ} and $\nabla A=\K^{-1}  x$:
\begin{align}\label{drift-field}
  \D  F(x)&=\CC  x-\D \nabla A(x)=\big[\CC-\frac12(\CC\K+\K\CC^T)\K^{-1}\big]   x \nonumber\\
  &=\frac12 (\CC\K-\K\CC^T) \K^{-1}  x\,.
\end{align}
To verify the divergence-free-condition \eqref{div-free} we compute
\begin{align*}
  &\!\!\!\!\div(\D  F(x)\ \finf(x)) \\
  &= \frac12 \tr([\CC\K-\K\CC^T]\K^{-1})\,\finf
  -\frac12 (\K^{-1}  x)^T \,[\CC\K-\K\CC^T]\,\K^{-1}  x\,\finf =0\,,
\end{align*}
due to the skew-symmetry of $\CC\K-\K\CC^T$.

Next we verify the condition \eqref{R-cond}: 
$$
  (\nabla^T  (\RR\finf))^T =-\RR \nabla\finf=\frac12(\CC\K-\K\CC^T) \K^{-1}  x\,\finf =\D  F(x)\finf(x)\,,
$$
where we used \eqref{drift-field} in the last step.
The claims \eqref{L-decomp3a}, \eqref{L-decomp3b} then follow from Lemma \ref{le-L-decomp}.

Finally we prove that $\RR\ne\0$. Otherwise \eqref{Dequ} would imply $\D=\K\CC^T$, and hence $\ker \D=\ker \CC^T$, which would contradict condition (A).
\end{proof}

The result $\RR\ne\0$ shows that hypocoercive FPEs of form \eqref{linmasterequ} are always non-symmetric.\\

Next we shall bring the hypocoercive FPEs \eqref{linmasterequ} to a normalized form, which will simplify our computations below. With its steady state given in \eqref{steady-state} we introduce, as a first step, the coordinate transformation $y:=\sqrt\K^{\,-1}x\in\R^d$. 
With $g(y):=f(\sqrt\K y)$, this transforms \eqref{linmasterequ} to 
$$
  \partial_t g=\div_y(\widetilde\D   \nabla_y g+\widetilde\CC  y\,g)\,,
$$
with $\widetilde\D=\sqrt\K^{\,-1}\D \sqrt\K^{\,-1}$ and $\widetilde\CC=\sqrt\K^{\,-1} \CC \sqrt\K$. A simple computation, using \eqref{Dequ} shows that
$$
  \widetilde\D=\widetilde\CC_s\,,
$$
where $\widetilde\CC_s:=\frac12(\widetilde\CC+\widetilde\CC^T)$ denotes the symmetric part of $\widetilde\CC$.
Clearly, the transformed steady state reads $g_\infty(y)=c\,\e^{-|y|^2/2}$, with some normalization constant $c>0$.
As a second step we rotate the coordinate system to diagonalize the diffusion matrix: For an orthogonal matrix $\U\in\R^{d\times d}$, let $\widehat \D:=\U^T\widetilde\D\U=\mbox{diag}(d_1,\,...,\,d_k,\,0\,,...,\,0)$, where $k=\rank\D$. We set $z:=\U^Ty$ and $h(z):=g(\U z)$, which satisfies
\begin{equation}\label{normalizedFP}
  \partial_t h=\div_z(\widehat{\D}   \nabla_z h+\widehat{\CC}  z\,h)\,.
\end{equation}
The symmetric part of the new drift matrix, $\widehat{\CC}=\U^T \widetilde\CC \U$, again satisfies $\widehat{\D}=\widehat{\CC}_s$. Since the matrices $\CC$ and $\widehat{\CC}$ are similar, we have $\sigma(\CC)=\sigma(\widehat{\CC})$, which will be the quantity that determines the decay rate of a hypocoercive FPE.
We also note that $h_\infty(z)=c\,\e^{-|z|^2/2}$ with some normalization constant $c$.

We remark that the above Examples \ref{Ex1}, \ref{FP-Ex3}, and \ref{FP-Ex4} are already of this normalized form, but Example \ref{FP-Ex2} is not. The above normalization brings Example \ref{FP-Ex2} to the form
$$
   \partial_t h  =  \div_z\Big[\left(
\begin{array}{cc}
 2 & 0 \\
 0 & 0
\end{array}
 \right)  \nabla_z h + \left(
\begin{array}{cc}
 2 & 1\\
 -1 & 0
\end{array}
 \right)   z\,h\Big] \,.
$$
Scaling time by a factor $\frac12$ shows that this equals the FPE in Example \ref{Ex1} with the rotation parameter $\omega=-\frac12$, which is a limiting case in \eqref{spectrumC}.

For the rest of this chapter we shall always assume that the FPEs are normalized as in \eqref{normalizedFP}.
So, the matrices in \eqref{linmasterequ} will satisfy $\D=\CC_s$ with $\D$ being diagonal, which implies $\K=\I$.

%%%%%%%%%%%%%%%%%%%%%%%%%%%%%%%%%%%%%%%%%%%%%%%%%%%%%%%%%%%%%%%%%%%%%%%%%%%%%%%%%%%%%%%%%%%%%%%%
\subsection{Modified entropy method}\label{S24}

To start with, let us very briefly review the standard entropy method for FPEs (cf.\ \cite{BaEmD85, ArMaToUn01} for symmetric FPEs and \cite{ArCaJuL08,BoGe10} for non-symmetric FPEs): In a first step one establishes a differential inequality between the Fisher information \eqref{Fisher-info} of a solution $f(t)$ and its time derivative, which yields exponential decay of the Fisher information. We give the result for symmetric FPEs:
\begin{lemma}\label{Le-I-decay}
Let $f(t)$ be the solution to \eqref{symmFP} with a constant diffusion matrix $\D$. Let the coefficients of this FPE satisfy the following \emph{Bakry-\'Emery condition} for some $\lambda_1>0$:
\begin{equation}\label{BEC}
  \frac{\partial^2 A}{\partial x^2}(x)\ge \lambda_1 \D^{-1}\,,\quad \forall\,x\in\R^d\,.
\end{equation}
Also, let the initial condition satisfy $I_\psi(f_0|\finf)<\infty$. Then, the Fisher information decays exponentially:
\begin{equation}\label{I-decay}
  I_\psi(f(t)|\finf) \le \e^{-2\lambda_1 t}I_\psi(f_0|\finf)\,,\quad t\ge0\,.
\end{equation}
\end{lemma}
\begin{proof}
After a lengthy computation the time derivative of the Fisher information can be written as follows (for scalar diffusions $D(x)$ cf.\ Lemma 2.13 in \cite{ArMaToUn01}, and for the generalization to non-symmetric FPEs \eqref{nonsymmFP} cf.\ Lemma 2.3 in \cite{ArCaJuL08}). Using the notation $u:=\nabla\frac{f}{f_{\infty}}$ we have:
\begin{align}\label{I-diff-ineq}
 \ddt I_\psi(f(t)) &= - 2\intRd{\psi '' \Big( \frac{f}{\finf}\Big)\ u^T 
  \D \frac{\partial^2A}{\partial x^2}\D u \,f_{\infty}\d}-  2 \intRd \tr{(\X \Y)}\, \finf \d \nonumber\\
 &\le - 2\lambda_1\intRd{\psi '' \Big( \frac{f}{\finf}\Big)\ u^T  \D   u \,f_{\infty}\d}
 =-2\lambda_1\,I_\psi(f(t))\,.
\end{align}
In the last estimate we used the Bakry-\'Emery condition \eqref{BEC} and $\tr{(\X \Y)}\ge0$. Here, the two matrices $\X,\,\Y\in\R^{2\times2}$ are defined as follows:
\begin{align}\label{Xdef}
	\X(x)&:= \left( \begin{array}{cc} 
	\psi '' & \psi ''' \\ \psi ''' & \frac12 \psi^{IV}
	\end{array}\right)
	\Big( \frac{f(x)}{f_{\infty}(x)} \Big) \ge \0 \,,\quad \forall\,x\in\R^d\,,
\end{align}
since $\det\,\X = \frac12 \psi ''\psi^{IV}-(\psi '')^2\ge0$ for admissible relative entropies (cf.\ \eqref{entropy-generator}).
\begin{align}\label{Ydef}
		\Y(x)&:= \left( \begin{array}{cc} 
		\tr\big[\big(\D\frac{\partial u}{\partial x}\big)^2\big] & u^T  \D\frac{\partial u}{\partial x}\D   u \\[2mm]
		u^T   \D\frac{\partial u}{\partial x}\D   u & (u^T \D  u)^2
		\end{array}\right) \ge\0\,,\quad \forall\,x\in\R^d\,,
\end{align}
due to the Cauchy-Schwarz inequality.
The differential inequality \eqref{I-diff-ineq} for $I_\psi(f(t))$ implies \eqref{I-decay}, and it can be written 
%$$
%  \frac{\d[]}{\d[t]} I_\psi(f(t)|\finf)\le -2\lambda_1 \,I_\psi(f(t)|\finf)\,,\quad t\ge0\,,
%$$
equivalently as $e''(t)\ge-2\lambda_1 e'(t)$ (with $e(t):=e_\psi(f(t)|\finf)$).
\end{proof}

In the second step of the entropy method one proves the exponential decay of the relative entropy \eqref{rel-entropy} of $f(t)$ w.r.t.\ $\finf$. To this end one integrates \eqref{I-diff-ineq} from $t$ to $\infty$, which yields the entropy inequality
\begin{equation}\label{entropy-ineq}
  \frac{\d[]}{\d[t]} e_\psi(f(t)|\finf)\le -2\lambda_1 \,e_\psi(f(t)|\finf)\,,\quad \forall\,t\ge0\,.
\end{equation}
Hence, the relative entropy decays exponentially:
\begin{equation}\label{e-decay}
  e_\psi(f(t)|\finf) \le \e^{-2\lambda_1 t}e_\psi(f_0|\finf)\,,\quad \forall\,t\ge0\,.
\end{equation}

%\begin{comment} %ENTFERNEN !!
\begin{figure}[htbp]
\begin{center}
\includegraphics[width=12cm]{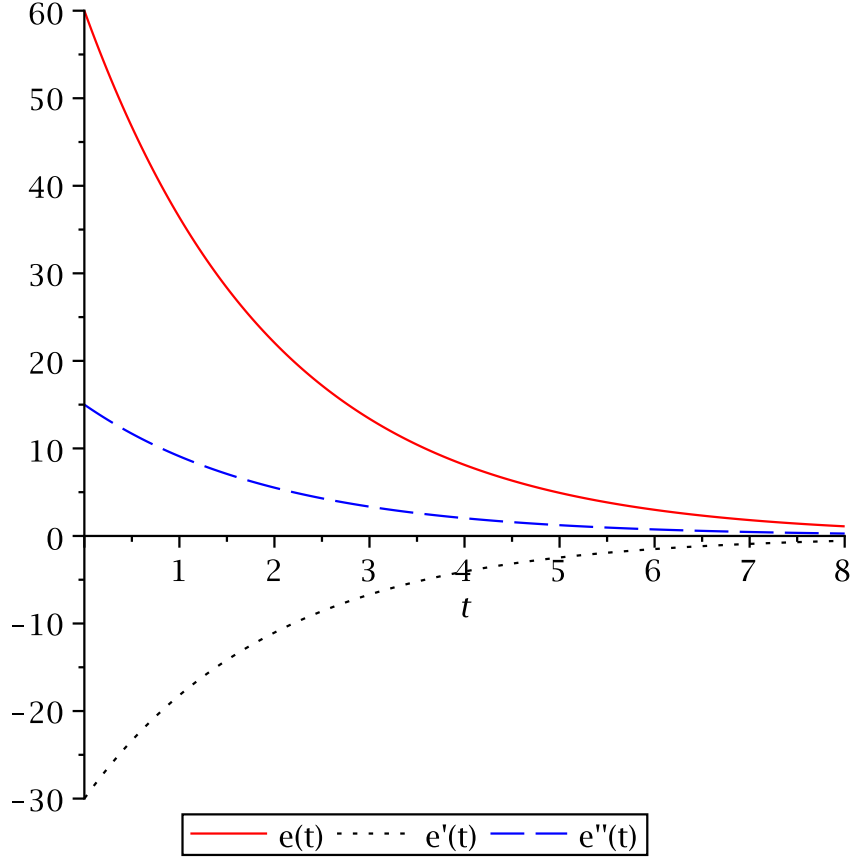}%\hfill
\end{center}
\vspace{-0.3cm}
\caption{\label{fig3a} Prototypical behavior of the logarithmic relative entropy $e_1(f(t)|\finf)$ (solid red curve), its first (dotted black), and second time derivative (dashed blue) for a non-degenerate, symmetric FPE: The inequalities $e'\leq-2\lambda e$, $e''\geq -2\lambda e'$ can be obtained.  (colors only online)
}
\end{figure}

\begin{figure}[htbp]
\begin{center}
\includegraphics[width=12cm]{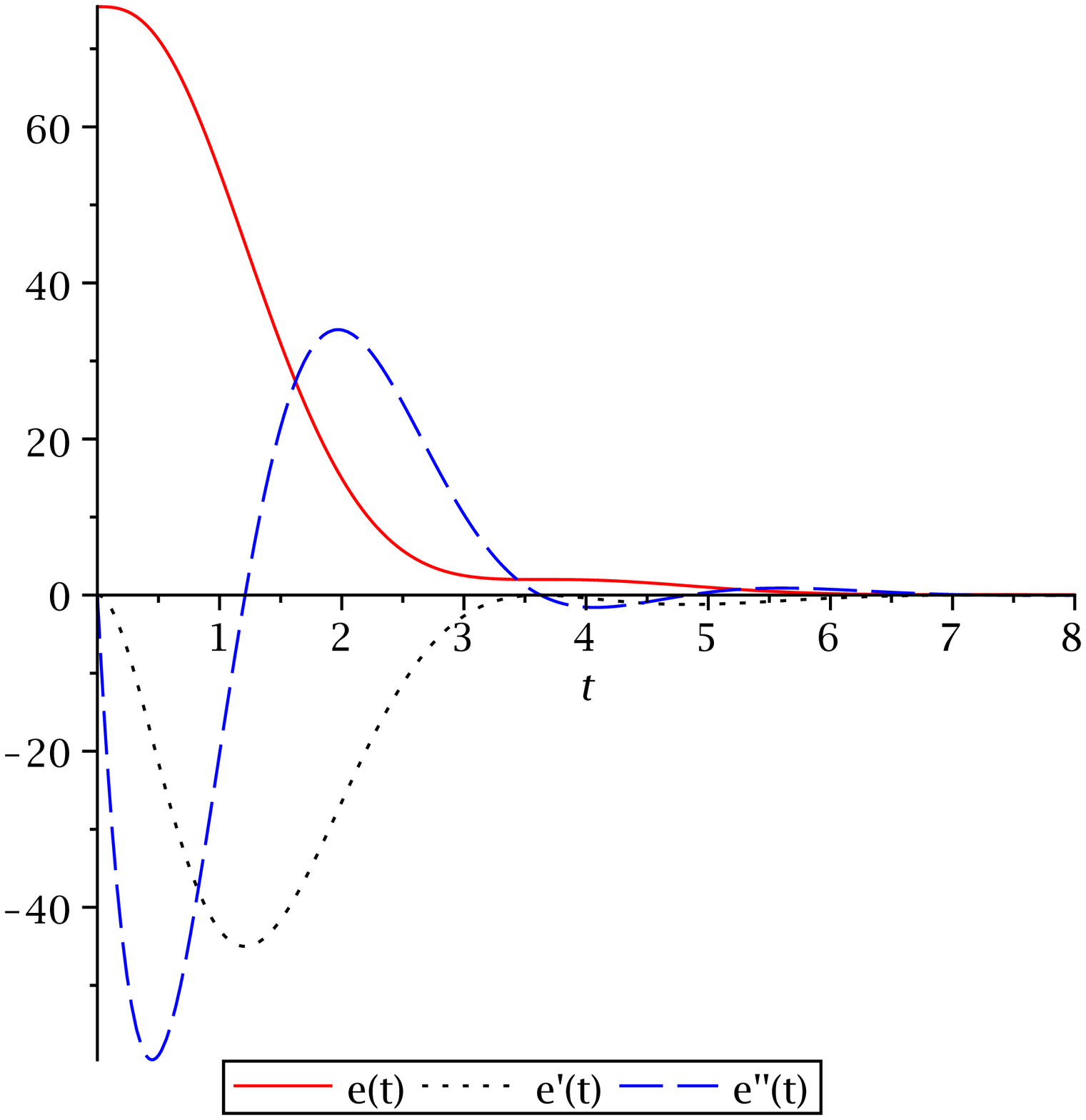}
\end{center}
\vspace{-0.3cm}
\caption{\label{fig3b} Prototypical behavior of the logarithmic relative entropy $e_1(f(t)|\finf)$ (solid red curve), its first (dotted black), and second time derivative (dashed blue) for the degenerate, hypocoercive FPE from Example \ref{Ex1} with 
 $\D=\diag(1,\,0)$, %  \left(\begin{array}{cc} 1 & 0 \\ 0 & 0 \end{array}\right)$,
 $\CC=[1\;-1\,;\,1\;0]$ %\left(\begin{array}{cc} 1 & -1 \\ 1 & 0 \end{array}\right)$
: The inequalities $e'\leq-2\lambda e$, $e''\geq -2\lambda e'$ are wrong, in general. (colors only online)
}
\end{figure}
%\end{comment} %ENTFERNEN !!
\medskip

Next we illustrate how the situation changes from a symmetric FPE to a non-symmetric or even hypocoercive FPE. In a symmetric FPE with $\D>0$, the relative entropy is a convex function of time, and the entropy dissipation satisfies $e'_\psi(f|\finf)<0$ for all probability densities $f\ne\finf$ (cf.\ Figure \ref{fig3a}). For a hypocoercive FPE with a singular diffusion matrix $\D$, however, $e(t)$ is not convex. In fact, it decays in a ``wavy'' fashion, and it may have horizontal tangents at equally spaced points in time (cf.\ Figure \ref{fig3b}). 
This oscillatory behavior is also known from space-inhomogeneous kinetic equations
(cf. \S3.7 of \cite{Vi02}; and \cite{FMP06} for a numerical study on the Boltzmann equation).

So we observe that the entropy dissipation $e'_\psi(f|\finf)$ may vanish for certain probability densities $f\ne\finf$. 
This can also be seen from the form of the Fisher information in \eqref{Fisher-info}: choose $f(x)=(1+c\cdot x)\finf(x)$ with a vector $c\in\ker\D$. Hence, an entropy inequality of the form \eqref{entropy-ineq} cannot hold for degenerate, hypocoercive FPEs! We also see: While the Fisher information $I_\psi(f(t)|\finf)$ is a Lyapunov functional for symmetric FPEs, its non-monotonicity in the hypocoercive case makes it ``useless'' there. As a remedy, we present now a \emph{modified entropy method} for FPEs of the form \eqref{linmasterequ}, normalized as introduced in \S\ref{S23}.

Since the above problems stem from the singularity of $\D$, we now define a modified entropy dissipation functional:
\begin{equation}\label{Sdefinition} 
  S_\psi(f) :=  \intRd \psi''\Big(\fu\Big)\ \Big(\nabla\fu\Big)^T \P \Big(\nabla\fu\Big)\,f_\infty\,\d\ge0\,, 
\end{equation}
where the positive definite matrix $\P\in\R^{d\times d}$ still has to be determined. Note that the only difference to the Fisher information is the replacement of the matrix $\D$ there by $\P$ here.
This auxiliary functional will take over the role of $I_\psi$ in the first step of the entropy method. 
So, our goal is to derive a differential inequality between $S_\psi(f(t))$ and $\frac{\d[]}{\d[t]}S_\psi(f(t))$ for a ``good'' choice of $\P>0$. Then, once exponential decay of $S_\psi(f(t))$ is obtained, the trivial estimate $\P\ge c_P\D$ (with some $c_P>0$) implies
$$
  S_\psi(f(t)) \ge c_P I_\psi(f(t)|\finf)\,,
$$
and also exponential decay of $I_\psi(f(t))$ follows.\\

The key question for using the modified entropy dissipation functional $S_\psi(f)$ is how to choose the matrix $\P$ for a given, normalized FPE. To determine $\P$ we shall need the following algebraic result:
\begin{lemma}
\label{Pdefinition}
For any fixed matrix $\Q\in\R^{d\times d}$, let $\mu:=\min\{\Real\{\lambda\}|\lambda$ is an eigenvalue of $\Q\}$. Let $\{\lambda_{m}|1\leq m\leq m_0\}$ be all the eigenvalues of $\Q$ with $\Real\{\lambda_m\}=\mu$, only counting their geometric multiplicity.\\
 \begin{itemize}
 \item[(i)] If all $\lambda_m$, $m\in\{1,\dots,m_0\}$, are non-defective\footnote{An eigenvalue is defective if its geometric multiplicity is strictly less than its algebraic multiplicity.}, then there exists a symmetric, positive definite matrix $\P\in\R^{d\times d}$ with
\begin{align}
\label{matrixestimate} \P\Q+\Q^T\P &\geq 2\mu \P\,.
\end{align}
\item[(ii)] If $\lambda_m$ is defective for at least one $m\in\{1,\dots,m_0\}$, then for any $\eps>0$ there exists a symmetric, positive definite matrix $\P=\P(\eps)\in\R^{d\times d}$ with
\begin{align}
\label{degeneratematrixestimate} \P\Q+\Q^T\P &\geq 2(\mu-\eps) \P\,.
\end{align}
 \end{itemize}
\end{lemma}
\begin{proof}
Here we only give the proof for the case that $\Q$ is not defective (and hence diagonalizable) and refer to Lemma 4.3 in \cite{ArEr14} for the general case. Let $w_1,\dots,w_d$ denote the eigenvectors of $\Q^T$. Then one can choose $\P$ as a weighted sum of the following rank 1 matrices:
\begin{align}
\label{simpleP} \P:= \sum\limits_{j=1}^d b_j \,w_j\otimes \overline{w_j}^T\,,
\end{align}
with $b_j\in\R^+$; $j=1,\dots,d$. As $\{w_j\}_{j=1,\dots,d}$ is a basis of $\C^d$, $\P$ is positive definite. If any $w_j$ is complex, its complex conjugate $\overline{w_j}$ is also an eigenvector of $\Q^T$, since $\Q$ is real. By taking the same coefficient $b_j$ for both, we obtain a real matrix $\P$. Apart from this restriction, the choice of $b_j>0$ is arbitrary.
For $\P$ from (\ref{simpleP}), we have
\begin{align*}
 \P\Q+\Q^T\P &= %\sum\limits_{j=1}^d b_j \big(Qw_j w_j^H + w_j w_j^HQ^H\big) = 
 \sum\limits_{j=1}^d b_j (\overline{\lambda_j}+\lambda_j) w_j\otimes \overline{w_j}^T %\\
%				 &= 2 \sum\limits_{j=1}^d 2\Real\{\lambda_j\} b_j \,w_j\otimes \overline{w_j} \\&
         \geq 2\mu \sum\limits_{j=1}^d b_j \,w_j\otimes \overline{w_j}^T = 
				 2\mu \P\,.
\end{align*}
\end{proof}
We remark that $\P$ is, in general, not unique, not even up to a multiplicative constant. But this will be irrelevant for the decay rate of FPEs.\\

Applying this lemma to $\Q:=\CC$ now yields exponential decay of the functional $S_\psi(f(t))$, defined with the matrix $\P$ from the above lemma:
\begin{proposition}
 \label{Sconvergence}
 Assume condition (A). Let $\psi$ generate an admissible entropy, let $f$ be the solution to (\ref{linmasterequ}) with an initial state satisfying $S_\psi(f_0)<\infty$, and let $\mu:=\min\left\{\Real\{\lambda\}|\lambda\text{ is an eigenvalue of } \CC\right\}$ (which is positive by condition (A)). Let $\{\lambda_m|1\leq m\leq m_0\}$ be the eigenvalues of $\CC$ with $\Real\{\lambda_m\}=\mu$, and let $\P$ be defined as in Lemma \ref{Pdefinition}.
 \begin{itemize}
  \item[(i)] If all $\lambda_m$, $1\leq m\leq m_0$, are non-defective, then
\begin{align*}
 S_\psi(f(t)) &\leq S_\psi(f_0)\e^{-2\mu t},\quad t\geq 0.
\end{align*}
  \item[(ii)] If $\lambda_m$ is defective for at least one $m\in\{1,\dots,m_0\}$, then
\begin{align*}
 S_\psi(f(t),\eps) &\leq S_\psi(f_0,\eps)\e^{-2(\mu-\eps)t},\quad t\geq 0,
\end{align*}
for any $\eps\in(0,\mu)$. Here, $S_\psi(f,\eps)$ denotes the modified entropy dissipation functional \eqref{Sdefinition} with the matrix $\P=\P(\eps)$.
 \end{itemize}
\end{proposition}
\begin{proof}
In a tedious computation the time derivative of the functional $S(\psi(f(t))$ can be written as follows (cf.\ Proposition 4.5 in \cite{ArEr14}). Using the notation $u:=\nabla\frac{f}{f_{\infty}}$ we have:
\begin{align}\label{S-diff-ineq}
 \ddt S_\psi(f(t)) =& - 2\intRd{\psi '' \Big( \frac{f}{\finf}\Big)\ u^T 
  \big[\P\CC+\CC^T\P\big]   u \,f_{\infty}\d}\nonumber\\
 &-  2 \intRd \tr{(\X \Y_P)}\, \finf \d \\
 \le& - 2\kappa\intRd{\psi '' \Big( \frac{f}{\finf}\Big)\ u^T  \P   u \,f_{\infty}\d}\nonumber
 =-2\kappa\,S_\psi(f(t))\,,
\end{align}
where $\kappa:=\mu$ for case (i), and $\kappa:=\mu-\eps$ for the defective case (ii).
In the last estimate we used the matrix inequality \eqref{matrixestimate} in case (i) and \eqref{degeneratematrixestimate} for case (ii). This inequality replaces the Bakry-\'Emery condition \eqref{BEC} used in the standard entropy method (compare to the estimate \eqref{I-diff-ineq}). In \eqref{S-diff-ineq} we also used $\tr{(\X \Y_P)}\ge0$, where the matrix $\X$ is defined in \eqref{Xdef}, and the matrix $\Y_P\in\R^{2\times2}$ is now defined as follows:
\begin{align*}
		\Y_P(x)&:= \left( \begin{array}{cc} 
		\tr{\big(\D\frac{\partial u}{\partial x}\P\frac{\partial u}{\partial x}\big)} & u^T  \D\frac{\partial u}{\partial x}\P   u \\[2mm]
		u^T   \D\frac{\partial u}{\partial x}\P   u & (u^T \P  u)(u^T \D  u)
		\end{array}\right) \ge\0\,,\quad \forall\,x\in\R^d\,.
\end{align*}
The positivity of $\Y_P$ follows from the Cauchy-Schwarz inequality using \linebreak 
$(u^T\D\frac{\partial u}{\partial x}\P u)^2=
\tr(\sqrt\P (u\otimes u^T)\sqrt\D\ \sqrt\D\frac{\partial u}{\partial x}\sqrt\P)^2$. Note that, for $\P:=\D$, the matrix $\Y_P$ would simplify to $\Y$ from Lemma \ref{Le-I-decay}.

The differential inequality \eqref{S-diff-ineq} for $S_\psi(f(t))$ then yields the claimed exponential decay of $S_\psi(f(t))$.
\end{proof}

This concludes the first step of the modified entropy method. In the second step we want to prove exponential decay of the relative entropy $e_\psi(f(t)|\finf)$. In the standard entropy method this is achieved by integrating the differential inequality \eqref{I-diff-ineq} for $I_\psi(f(t))$ in time, since $e'(t)=-I_\psi(f(t))$. But here, this is not possible, since $S_\psi(f(t))$ is not the time derivative of $e(t)$. Instead, we shall use \emph{convex Sobolev inequalities} (cf.\ \S3 of \cite{ArMaToUn01}; \cite{UAMT00}), which give a simple relation between these two functionals. In fact, the functional $S_\psi(f)$ controls the relative entropy $e_\psi(f|\finf)$:
\begin{lemma}
\label{Sfiniteefinite}
Let  $\P$ be some fixed positive definite matrix. % let $S_\psi(f)<\infty$. 
Then, the following \emph{convex Sobolev inequality} holds $\forall\,g\in L^1_+(\R^d)$ with $\intRd g\d=1$:
\begin{align}\label{convSobinequ}  
  e_\psi(g|f_\infty) &\leq \frac1{2\lambda_P} S_\psi(g)\,,
\end{align}
where both sides may be infinite. The constant $\lambda_P>0$ is the smallest eigenvalue of $\P$, i.e.\ 
\begin{equation}\label{gBEC}
  \P\ge \lambda_P\I>\0 \,.
\end{equation}
\end{lemma}
\begin{proof}
As an auxiliary problem we consider the following symmetric non-degenerate FPE for $g=g(t,x)$ on $L^2(\finf^{-1})$:
\begin{equation}\label{auxFP}
   \partial_t g=\div\Big(f_\infty\P  \nabla\frac{g}{\finf}\Big)\,,
\end{equation} 
with $\finf=(2\pi)^{-\frac{d}{2}}\e^{-\frac{|x|^2}{2}}$. 
This is motivated by the fact that $S_\psi(g)$ is the true Fisher information for the evolution under \eqref{auxFP}.
Obviously, we have $g_\infty=\finf$. We also note that \eqref{gBEC} is the (standard) Bakry-\'Emery condition for \eqref{auxFP}, since its steady state potential is $A(x)=|x|^2/2$ (cf.\ \eqref{BEC}).

Hence, the entropy method implies exponential decay of $g(t)$ to $g_\infty$ with rate $2\lambda_P$ (cf.\ \eqref{e-decay}). 
Moreover, the entropy inequality \eqref{entropy-ineq} is already the claimed result.
\end{proof}
Combining this lemma with Proposition \ref{Sconvergence} readily yields exponential decay of the relative entropy, provided that $S_\psi(f_0)<\infty$:

\begin{theorem}\label{entropydecay}
Assume condition (A). Let $\psi$ generate an admissible entropy, let $f$ be the solution to (\ref{linmasterequ}) with an initial state satisfying $S_\psi(f_0)<\infty$, and let $\mu:=\min\left\{\Real\{\lambda\}|\lambda\text{ is an eigenvalue of } \CC\right\}$. Let $\{\lambda_m|1\leq m\leq m_0\}$ be the eigenvalues of $\CC$ with $\Real\{\lambda_m\}=\mu$, and let $\P$ be defined as in Lemma \ref{Pdefinition}.
 \begin{itemize}
  \item[(i)] If all $\lambda_m$, $1\leq m\leq m_0$, are non-defective, then
\begin{align}\label{entropydecay-nondeg}
 e_\psi(f(t)|f_\infty) &\leq \frac{1}{2\lambda_P} S_\psi(f_0)\e^{-2\mu t},\quad t\geq 0.
\end{align}
  \item[(ii)] If $\lambda_m$ is defective for at least one $m\in\{1,\dots,m_0\}$, then
\begin{align}\label{entropydecay-deg}
 e_\psi(f(t)|f_\infty) &\leq \frac{1}{2\lambda_P} S_\psi(f_0,\eps)\e^{-2(\mu-\eps)t},\quad t\geq 0,
\end{align}
for any $\eps\in(0,\mu)$. Here, $S_\psi(f,\eps)$ denotes the modified entropy dissipation functional \eqref{Sdefinition} with the matrix $\P=\P(\eps)$.
 \end{itemize}
\end{theorem}

We remark that the multiplicative constant in \eqref{entropydecay-deg} is $\eps$--dependent, with $S_\psi(f_0,\eps) \to \infty$ as $\eps\searrow0$. In \eqref{entropydecay-nondeg} the exponential decay rate is indeed sharp (cf.\ \S6 of \cite{ArEr14}). 
%We remark that this exponential decay rate is indeed sharp (cf.\ \S6 of \cite{ArEr14}). 
%But in \eqref{entropydecay-deg}, the multiplicative constant is $\eps$--dependent, with $S_\psi(f_0,\eps) \to \infty$ as $\eps\searrow0$. 
Also, it is independent of the normalizing transformation in \S\ref{S23}, since the drift matrices $\CC$ and $\widehat\CC$ are similar. But compared to the standard entropy method, the above result is not yet fully satisfactory: In the decay estimate \eqref{e-decay} the initial condition is only required to have finite relative entropy. By contrast, Theorem \ref{entropydecay} requires the initial state to satisfy $S_\psi(f_0)<\infty$, and this functional is closely related to a weighted $H^1$--norm. This ``deficiency'' of Theorem \ref{entropydecay} can be lifted by exploiting the hypoelliptic regularization of \eqref{linmasterequ}, cf.\ also Proposition \ref{prop1}(a). The following result is a generalization of Theorems A.12, A.15 in \cite{ViH06} (expressed for quadratic and logarithmic entropies) to all admissible $\psi$-entropies. For its proof we refer to Theorem 4.8 in \cite{ArEr14}.

\begin{lemma}
\label{entropyregularisation}
 Let condition (A) hold, $f_0\in L^1_+(\R^d)$ with $\intRd f_0 \d =1$ and $e_\psi(f_0|f_\infty)<\infty$. Let $f(t)$ be the solution of (\ref{linmasterequ}) with initial condition $f_0$, and let $\tau$ be the minimal constant such that \eqref{sumTdefinite} (or, equivalently, \eqref{rank-cond}) holds. Then there exists a positive constant $c_r>0$ such that
 \begin{align}
 \label{regularityestimate}  S_\psi(f(t)) &\leq c_rt^{-(2\tau+1)}e_\psi(f_0|f_\infty)\,,\qquad\forall\, t\in(0,1]\,.
 \end{align}
\end{lemma}

With this ingredient we are ready to state our final result:
\begin{theorem}
 \label{convergencerate}  Assume condition (A). Let $\psi$ generate an admissible relative entropy and let $f$ be the solution to (\ref{linmasterequ}) with initial state $f_0\in L^1_+(\R^d)$ such that $e_\psi(f_0|f_\infty)<\infty$. Let $\mu:=\min\{\Real\{\lambda\}|\lambda\text{ is an}$ $\text{eigenvalue of } \CC\}$. Let $\{\lambda_m|1\leq m\leq m_0\}$ be the eigenvalues of $\CC$ with $\mu=\Real\{\lambda_m\}$, and let
 \begin{align*}
  e(t) &:= e_\psi(f(t)|f_\infty).
 \end{align*}
 Then:
 \begin{itemize}
  \item[(i)] If all $\lambda_m$, $1\leq m\leq m_0$, are non-defective, then there is a constant $c\ge1$ such that
  \begin{align}\label{entropy-decay-nondeg}
   e(t) &\leq c\,\e^{-2\mu t}e_\psi(f_0|f_\infty)\,, \qquad\forall \,t\geq0\,.
  \end{align}
  \item[(ii)] If $\lambda_m$ is defective for at least one $m\in\{1,\dots,m_0\}$, then, for all $\eps\in(0,\mu)$, there is $c_\eps\ge1$ such that
  \begin{align}\label{entropy-decay-deg}
   e(t) &\leq c_\eps \e^{-2(\mu-\eps)t}e_\psi(f_0|f_\infty)\,, \qquad\forall \,t\geq0\,. 
  \end{align}
 \end{itemize}
 %$\lambda_P$ is the constant from the Bakry-Émery condition in Proposition (\ref{Sfiniteefinite}).
\end{theorem}
\begin{proof}
Let $\P$ be defined as in Lemma \ref{Pdefinition}. Fix some $\delta>0$, and let $\kappa:=\mu$ in case (i), and $\kappa:=\mu-\eps$ in case (ii). Using the convex Sobolev inequality (\ref{convSobinequ}), Proposition \ref{Sconvergence}, and Lemma \ref{entropyregularisation}, we compute for $t\geq\delta$:
\begin{align}
\nonumber e(t) &\leq \frac1{2\lambda_P}S_\psi(f(t))\leq \frac1{2\lambda_P}S_\psi(f(\delta))\e^{-2\kappa(t-\delta)} \\
\label{tgdestimate} &\leq \e^{2\kappa\delta}\frac{c_r}{2\lambda_P\delta^{2\tau+1}} e(0)\e^{-2\kappa t}.
\end{align}
For $t\leq \delta$, the monotonicity of $e(t)$ (cf.\ (\ref{Fisher-info})) implies
\begin{align}
\label{tldestimate} e(t) &\leq e(0)\,. %\exp(-2\kappa(t-\delta)).
\end{align}
Writing $c_\delta:=\e^{2\kappa\delta}\max\{1,\frac{c_r}{2\lambda_P\delta^{2\tau+1}}\}$ and combining (\ref{tgdestimate}), (\ref{tldestimate}) yields
\begin{align*}
 e(t) &\leq c_\delta e(0)\e^{-2\kappa t}\,,\qquad\forall \,t\geq0\,. 
\end{align*}
%$c_\delta$ can now be optimized for $\delta>0$, completing the proof.
\end{proof}
We remark that the exponential decay rate $2\kappa$ is sharp here, but the multiplicative constant $c$ will in general not be sharp.\\

To close this section we shall now briefly illustrate the mechanism of the presented \emph{modified entropy method}. To this end we return to Example \ref{Ex1} and the ``distorted'' vector norm 
$$
  |x|_{P}:= \sqrt{\langle x,\P x\rangle}\,,
$$
with $\P>\0$, that was already used in Figure \ref{fig1}. The drift characteristics $x(t)$ corresponding to \eqref{FP-Ex1} satisfy $x_t=-\CC x$. For the decay of this $\P$--norm along a characteristic we obtain
\begin{align}\label{Pnorm-decay}
  \ddt\;|x|^2_{P} &= -2x^T\P\CC x = -x^T\big(\P\CC +\CC^T\P\big) x 
%  &= -(\sqrt\P x)^T\,\big(\sqrt\P\CC\sqrt\P^{\,-1} + \sqrt\P^{\,-1}\CC^T\sqrt\P \big) \,(\sqrt\P x)\\ &
  \le -2\mu|x|^2_{P}\,,
\end{align}
where we used in the last step the matrix estimate \eqref{matrixestimate} for $\Q:=\CC$ and the notation $\mu:=\min\{\Real\{\lambda\}|\lambda$ is an eigenvalue of $\CC\}$. So, $\mu$ is the spectral gap of $\CC$, i.e.\ the distance of $\sigma(\CC)$ from the imaginary axis, and it determines the best possible decay of $x(t)$. Due to \eqref{Pnorm-decay}, $|x|_P$ realizes this optimal decay uniformly in time. 

The matrix $\P$ determining this ``distorted'' vector norm  is defined via \eqref{matrixestimate}, and hence it is the same matrix as in the definition of the modified entropy dissipation functional $S_\psi(f)$.\hfill $\square$

%%%%%%%%%%%%%%%%%%%%%%%%%%%%%%%%%%%%%%%%%%%%%%%%%%%%%%%%%%%%%%%%%%%%%%%%%%%%%%%%%%%%%%%%%%%%%%%%
\subsection{Entropy methods for non-degenerate Fokker-Planck equations}\label{S25}

We remark that the new entropy method from \S\ref{S24} is not restricted to degenerate FPEs. For non-degenerate FPEs it is in fact a generalization of the standard entropy method: For a symmetric FPE with constant diffusion and drift matrices, the normalization of \S\ref{S23} yields $\D=\CC_s$ and $\D$ is symmetric positive definite. Applying Lemma \ref{Pdefinition}(i) to $\Q:=\CC$ with $\mu:=\lambda_{\min}(\CC)$ admits the choice $\P:=\D$. Hence, $S_\psi(f)=I_\psi(f|\finf)$ and the method of \S\ref{S24} reduces to the standard entropy method.

For non-symmetric FPEs, however, the standard and modified entropy methods differ. For regular diffusion matrices $\D>\0$, both methods are applicable and yield exponential decay of the solution towards equilibrium. So it is natural to compare their performances: For applying the standard entropy method to \eqref{linmasterequ} in normalized form (i.e.\ with $A(x)=|x|^2/2$) we consider the corresponding Bakry-\'Emery condition \eqref{BEC}:
$$
  \I\ge\lambda_D\D^{-1}\,,
$$
i.e.\ $\lambda_D>0$ is the smallest eigenvalue of $\D$. Then, \S2.4 in \cite{ArMaToUn01} (or the analogue of the convex Sobolev inequality \eqref{entropy-ineq}) implies exponential decay of the relative entropy:
\begin{equation}\label{e-decay-nsFP}
  e_\psi(f(t)|\finf) \le \e^{-2\lambda_D t}e_\psi(f_0|\finf)\,,\quad t\ge0\,.
\end{equation}
Note that the multiplicative constant in this estimate is 1.

For the modified entropy method, Theorem \ref{entropydecay} yields the decay estimate
%$\lambda_m$; $1\leq m\leq m_0$, are non-defective, then there is a constant $c\ge1$ such that
  \begin{align}\label{entropy-decay-nondeg2}
   e_\psi(f(t)|\finf) &\leq \frac1{2\lambda_P}S_\psi(f_0)\e^{-2\mu t} \qquad\forall \,t\geq0\,
  \end{align}
in the non-defective case (i), with $\mu:=\min\{\Real\{\lambda\}|\lambda\text{ is an}$ $\text{eigenvalue of } \CC\}$. For the comparison of the two obtained decay rates we have the following result:

\begin{proposition}\label{rate-comparison} Let the coefficients of a non-degenerate, normalized FPE satisfy condition (A). 
%Let $\mu:=\min\{\Real\{\lambda\}|\lambda\text{ is an}$ $\text{eigenvalue of } \CC\}$. 
With $\mu$ defined above, let $\{\lambda_m|1\leq m\leq m_0\}$ be the eigenvalues of $\CC$ with $\mu=\Real\{\lambda_m\}$.  Then:
 \begin{itemize}
  \item[(i)] If all $\lambda_m$, $1\leq m\leq m_0$, are non-defective, then 
\begin{equation}\label{rate-inequality1}
  0<\lambda_D\le\mu\,.
\end{equation}
  \item[(ii)] If $\lambda_m$ is defective for at least one $m\in\{1,\dots,m_0\}$, then
\begin{equation}\label{rate-inequality2}
  0<\lambda_D<\mu\,.
\end{equation}
 \end{itemize}
\end{proposition}
\begin{proof} For case (ii), let $\lambda$ with $\Real\{\lambda\}=\mu$ be a defective eigenvalue. Let $p\in\C^d$ with $|p|=1$ be a corresponding eigenvector, and $q\in\C^d$ a corresponding generalized eigenvector. W.l.o.g.\ we assume that $\langle q,\,p\rangle=0$, and $q$ satisfies $(\lambda\I-\CC)  q=p$.

Next we consider a family of generalized eigenvectors, $q_\delta:=q+\delta p$, $\delta\in\R$, which also satisfy  $(\lambda\I-\CC)  q_\delta=p$. We compute
$$
  \bar q_\delta^T (\CC+\CC^T)  q_\delta= \bar q_\delta^T   (\lambda q_\delta-p) + (\bar\lambda \bar q_\delta^T-\bar p^T)  q_\delta=2\Real\{\lambda\}\,|q_\delta|^2-2\delta\,.
$$
Using $\D=\CC_s$ and $|q_\delta|^2=|q|^2+\delta^2$ we obtain for the Rayleigh quotient of $\D$:
$$
  \lambda_D \le \frac{q_\delta^T \D  q_\delta}{|q_\delta|^2}
  =\mu-\frac{\delta}{|q|^2+\delta^2}\,,
$$
and \eqref{rate-inequality2} follows for any $\delta>0$.

For case (i) we only need to replace $q_\delta$ by $p$ in the above computation.
\end{proof}

For the non-defective case (i), the inequality \eqref{rate-inequality1} will in general not be strict, as can be verified on the following simple example:
\begin{align*}
 \CC:=\left(\begin{array}{ccc} 1/5 & 0 & 0 \\ 0 & 1/4 & -4 \\ 0 & 4 & 1  \end{array}\right)\,,
\end{align*}
with the eigenvalues $\frac15,\,\frac58\pm i\sqrt{1015}/8$, and $\D=\CC_s=\diag(\frac15,\,\frac14,\,1)$.\\

%For a normalized FPE we have $\D=\CC_s$. Comparing the spectrum of the matrix $\CC$ with the spectrum of its symmetric part $\CC_s$, one easily sees that the two obtained decay rates are related by

So, the exponential decay rate from the new entropy method is always at least as good as the rate from the standard entropy method, but often better.
The first rate $2\lambda_D$ from \eqref{e-decay-nsFP} gives an estimate for the \emph{local decay rate} of the relative entropy. 
It reflects the (in absolute value) smallest slope of the relative entropy at any $t\ge0$. More precisely, it is, pointwise in time, a lower bound for the local decay rate, i.e.\ $-\frac{e'(t)}{e(t)}$. For non-symmetric FPEs with linear drift it is well known (cf.\ \S2.4, \S3.5 in \cite{ArMaToUn01}) that this rate is optimal (as a pointwise estimate). In Figure \ref{fig:non-symmetric} the initial condition is chosen such that the function on the r.h.s.\ of \eqref{e-decay-nsFP} is indeed tangent to $e(t)$ at $t=0$.

%If $\lambda_m$ is defective for at least one $m\in\{1,\dots,m_0\}$, then, for all $\eps\in(0,\mu)$, there is $c_\eps\ge1$ such that
%  \begin{align}\label{entropy-decay-deg}
%   e(t) &\leq c_\eps \e^{-2(\mu-\eps)t}e_\psi(f_0|f_\infty) \qquad\forall \,t\geq0\,. 
%  \end{align}

By contrast, the estimate \eqref{entropy-decay-nondeg2} describes the \emph{global decay}. Hence, its multiplicative constant has to be larger than 1 for non-symmetric FPEs. In some examples, the r.h.s.\ of \eqref{entropy-decay-nondeg2} is even the perfect envelope of $e(t)$, see Figure \ref{fig:non-symmetric}.

%\begin{comment} %ENTFERNEN !!
\begin{figure}[ht!]
\begin{center}
 \includegraphics[scale=.6]{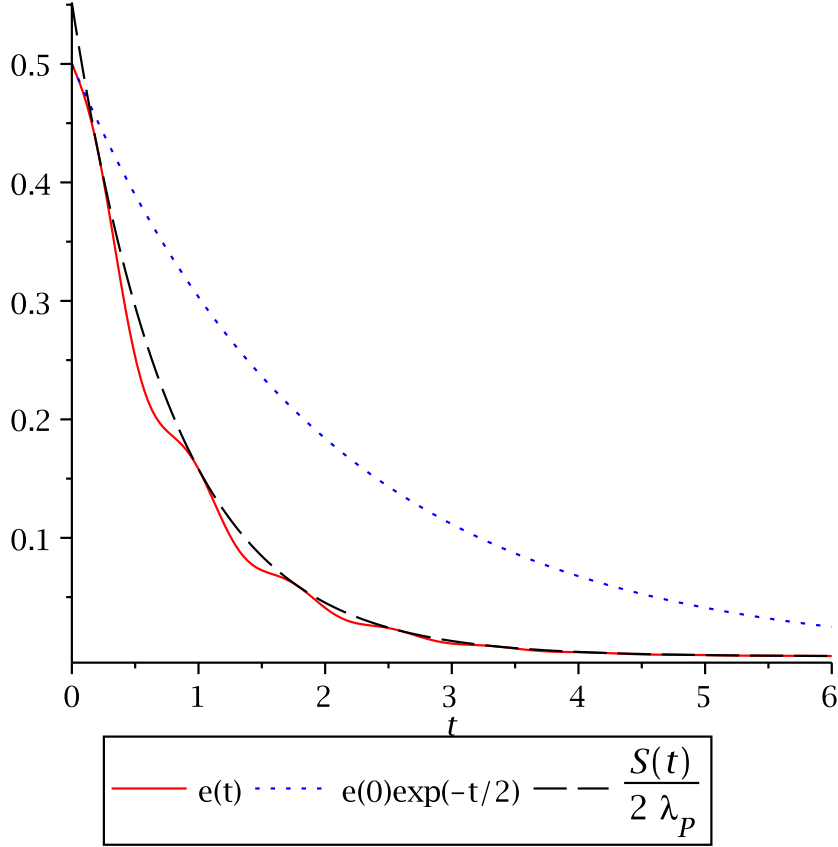}
 \caption{Entropy decay for the non-degenerate, non-symmetric Fokker-Planck equation \eqref{linmasterequ} 
 with $\D=\diag(1/4,\;1)$, $\CC=[1/4 \;\, -4 \; ;\; 4 \;\, 1]\,.$ %$C=\left(\begin{array}{cc} 1/4 & -8 \\ 2 & 1 \end{array}\right)$
 Solid red curve: decay of the logarithmic entropy $e_1(t)$; dotted blue: The estimate of the \emph{local decay rate} from the standard entropy method is tangent at $t=0$; dashed black: estimate of the \emph{global decay rate} from the hypocoercive entropy method.  (colors only online)}
 \label{fig:non-symmetric}
\end{center}
\end{figure}
%\end{comment} %ENTFERNEN !!

\begin{example}\label{FP-Ex5}
We consider the non-degenerate, non-symmetric Fokker-Planck equation \eqref{linmasterequ} with 
$$
  \D=\diag(1/4,\;1)\,,\quad \CC=\left(\begin{array}{cc} 1/4 & -4 \\ 4 & 1 \end{array}\right)\,,
$$
which is normalized. Here we have $\lambda_D=\frac14$ and $\mu=\frac58$, and the local and global decay estimates are shown in Figure \ref{fig:non-symmetric}. \hfill $\square$
\end{example}

\medskip

So far, we only discussed the modified entropy method for FPEs with constant diffusion and drift matrices. Its generalization to some cases of non-symmetric FPEs with non-constant coefficients is the topic of the subsequent chapter.

% version 20.3.2015, submitted to Rev Parma
% version 24.4.2015, submitted to Rev Parma
% version 11.5.2015, Franz, changes following suggestions by DS, submitted to Rev Parma
% version 3.9.2015, 5 notational details

\section{Kinetic Fokker-Planck equation with non-quadratic potentials}
\label{S3}

In this chapter we shall illustrate how the modified entropy method from \S\ref{S24} can be extended to kinetic Fokker-Planck equations \eqref{kinFP}
% \begin{equation} \label{kinFP}  
% \partial_t f +v\cdot\nabla_{\vektor x} f-\nabla_{\vektor x} V\cdot\nabla_v f = \nu \div_v (vf)+\sigma\Delta_v f\,;\quad x,\,v\in\R^n;\,t>0,
% \end{equation}
with non-quadratic potentials $V=V(x)$ (i.e.\ a drift term that is nonlinear in the position variable). 
A motivation for the following analysis is 
% that it can be a first step in
 its possible application to a future study of Fokker-Planck-Poisson equations
 with a quadratic confinement potential and the self-consistent potential acting as a perturbation.
Refer to~\cite[$\S$4.2]{ArMaToUn01},
 for the large time analysis of a non-degenerate drift-diffusion Poisson model.

Several proofs of the entropy-- and $L^2$--decay of this equation have already been obtained in the last few years: In \cite{DeVi01}, algebraic decay was proved for potentials that are asymptotically quadratic (as $|x|\to\infty$) and for initial conditions that are bounded below and above by Gaussians. The authors used logarithmic Sobolev inequalities and entropy methods. In \cite{HeNi04}, exponential decay was obtained also for faster growing potentials and more general initial conditions. That proof is based on hypoellipticity techniques. 
% In \S2 of \cite{BaB13}, exponential convergence is proved with a modified $\Gamma_2$--approach for potentials with a bounded Hessian. 
In \cite{DoMoScH10} exponential decay in $L^2$ was proved, allowing for potentials with linear or super-linear growth.
This chapter will now provide an alternative proof of exponential entropy decay for \eqref{kinFP} with a certain class of non-quadratic potentials and for all admissible relative entropies $e_\psi$.

The kinetic Fokker-Planck equation~\eqref{kinFP} has a unique normalized steady state
\begin{equation}\label{kinFP-steady}
  f_\infty(x,v)=\exp\left\{-\frac{\nu}{\sigma}\Big[V(x)+\frac{|v|^2}{2}\Big]\right\}\,,\quad x,v\in\R^n\,,
\end{equation}
for potentials $V(x)$ with $\lim_{|x|\to\infty} V(x)=\infty$ sufficiently fast such that $f_\infty\in L^1(\R^{2n})$, see~\cite{Vi02}.
An additive normalization constant is included in $V$.
% For well-posedness and instantaneous smoothing results of the kinetic Fokker-Planck equation \eqref{kinFP}
%  we refer to \cite{HeNi04,DeVi01,Vi02} as well as \S A.20, A.21 of \cite{ViH06}.

We rewrite \eqref{kinFP} again in the form \eqref{kinFP:nonsymmFP}, such that
\begin{equation}\label{kinFP2}
   \partial_t f = Lf := \div_\xi [\mathbf{D}\nabla_\xi f + G(\xi) f],
\end{equation}
where $\xi:=(x,\,v)^T\in\R^d,\,d=2n$, $\mathbf{D}$ is a block diagonal diffusion matrix and $G$ a  drift vector field given by 
 \[ \mathbf{D}=\left(\begin{array}{cc} 0 & 0 \\ 0 & \sigma\,\Id \end{array}\right) \XX{and}
     G(x,v)=\left(\begin{array}{c} -v \\ \nabla_{\vektor x} V+\nu v  \end{array}\right),
 \]
 respectively.
% Moreover, we shall use the abbreviation $E(\xi):=\frac{\nu}{\sigma}[V(x)+\frac{|v|^2}{2}]$.\\

The positivity of solutions of \eqref{kinFP} with non-negative initial datum
 can be proved using the sharp maximum principle~\cite{Hi70}; see also~\cite[Proposition 7.1]{ArEr14}.

We introduce the modified entropy dissipation functional $S_\psi(f)$ as in \eqref{Sdefinition},
\[ S_\psi(f) :=  \intRd \psi''\Big(\fu\Big) \Big(\nabla\fu\Big)^T \P \Big(\nabla\fu\Big)\,f_\infty\,\d[\xi] ,
%  S_\psi(f) &:=  \int\limits_{\fu>0} \psi''(\fu) \nabla(\fu)^T \mathbf{P}\nabla(\fu)f_\infty\d[\xi], 
\]
with a positive definite and $\xi$--independent matrix $\mathbf{P}\in \R^{d\times d}$ to be chosen later.
The time derivative of $S_\psi(f(t))$ is estimated as
 in the proof of Proposition~\ref{Sconvergence}---apart from not normalizing the equation---and it satisfies
\begin{equation}\label{kinFP-Sineq}
  \ddt S_\psi(f(t)) \le - \ild \psi''(\tfrac{f}{f_\infty}) u^T[(\mathbf{D}-\mathbf{R})\frac{\partial^2 E}{\partial \xi^2}\mathbf{P}+\mathbf{P}\frac{\partial^2 E}{\partial \xi^2}(\mathbf{D}+\mathbf{R})]u f_\infty \d[\xi]\,,
\end{equation}
where $u:=\nabla_\xi \frac{f}{f_\infty}$, $E(\xi):=\frac{\nu}{\sigma}[V(x)+\frac{|v|^2}{2}]$
 and $\mathbf{R}=\frac{\sigma}{\nu}\left(\begin{array}{cc} 0 & -\Id \\ \Id & 0 \end{array}\right) \in \R^{d\times d}$.
In analogy to \S\ref{S24} we define the matrix
\begin{equation}\label{matrixQ}
  \mathbf{Q}(x):=(\mathbf{D}-\mathbf{R})\frac{\partial^2 E}{\partial \xi^2}=\left(\begin{array}{cc} 0 & \Id \\ -\frac{\partial^2 V}{\partial x^2}(x) & \nu\,\Id \end{array}\right). 
\end{equation}
If we can find an $x$--independent, symmetric, positive definite matrix $\mathbf{P}>0$ and a constant $\kappa\geq 0$, such that
 \begin{equation} \label{C:detR}
   \mathbf{Q}(x)\mathbf{P}+\mathbf{P}\mathbf{Q}^T(x) -2\kappa \mathbf{P} \geq \mathbf{0} \quad\forall\,x\in\R^n\,,
 \end{equation}
then the right-hand-side of \eqref{kinFP-Sineq} can be estimated as
 \begin{equation} \label{kinFP-Sineq-2}
  \ddt S_\psi(f(t)) \le - 2\kappa  \ild \psi''(\tfrac{f}{f_\infty})\ u^T \mathbf{P} u\ f_\infty \d[\xi]
  =- 2\kappa S_\psi(f(t))  \,.
 \end{equation}
If additionally $\kappa>0$, this would imply exponential decay of $S_\psi(f(t))$.

\subsection{Potential $V(x)$ with bounded second order derivatives} \label{subsection:31}
In this section we % consider potentials $V(x)$ with bounded second order derivatives and
 prove in Theorem~\ref{kinFPdecay} the exponential convergence of solutions of~\eqref{kinFP} to the steady state
 via the modified entropy method. 

To keep the presentation simple, we shall consider from now on only the 1D case, i.e.\ $x,\,v\in\R$ ($d=2$). 
Furthermore, we shall consider non-quadratic potentials $V(x)$ with bounded second order derivatives satisfying
\begin{equation} \label{as:potential}
 \exists\, \gamma_1<\gamma_2 \quad \text{such that} \quad \gamma_1 \leq V''(x) \leq\gamma_2 \quad \forall x\in\R .
\end{equation}

To apply the modified entropy method, 
 we need to find a symmetric, positive definite matrix $\mathbf{P}$ and $\kappa\geq 0$ such that~\eqref{C:detR} is satisfied.
We define 
 \[ \mathbf{Q}_{\gamma} := \begin{pmatrix} 0 & 1 \\ -\gamma & \nu \end{pmatrix}
    % \XX{and} \mathbf{R}_{\gamma} := \mathbf{Q}_{\gamma} \mathbf{P} + \mathbf{P} \mathbf{Q}_{\gamma}^T - 2\kappa \mathbf{P} ;
    \XX{such that} \mathbf{Q}(x) = \mathbf{Q}_{\gamma} \big|_{\gamma=V''(x)}.
 \]
Then, for potentials $V$ satisfying~\eqref{as:potential} with $\gamma_1 = \inf_{x\in\R} V''(x)$ and $\gamma_2 = \sup_{x\in\R} V''(x)$,
 condition~\eqref{C:detR} is equivalent to 
 \begin{equation} \label{C:detR:gamma}
   \mathbf{Q}_{\gamma}\mathbf{P}+\mathbf{P}\mathbf{Q}^T_{\gamma} -2\kappa \mathbf{P}\geq \mathbf{0}
   \quad\forall\,\gamma\in [\gamma_1,\gamma_2]\,.
 \end{equation}

Next we collect the conditions on $\kappa$ and on the coefficients of the matrix $\mathbf{P}$:
A symmetric matrix $\mathbf{P}\in\R^{2\times 2}$ is positive definite
 iff its % eigenvalues are positive or, equivalently, the
 first diagonal element and its determinant are positive.
Condition~\eqref{C:detR} is linear in~$\mathbf{P}$, therefore,
 we consider---without loss of generality---matrices
 \begin{equation} \label{cond:P:0:mu}
   \mathbf{P}=\begin{pmatrix} 1 & p_{12} \\ p_{12} & p_{22} \end{pmatrix}\in\R^{2\times 2}
     \quad \text{with } \det(\mathbf{P})=p_{22}-p_{12}^2>0 \,.
 \end{equation}
% where the matrix is normalized such that the positive first diagonal element is equal to $1$.
%Depending on the sign of $\nu^2 - 4\gamma$,
For given $0<\nu$ and $\gamma_1 < \gamma_2$,
 we want to determine $\kappa\geq 0$ and symmetric, positive definite matrices $\mathbf{P}$
 such that~\eqref{C:detR:gamma} holds.
 % $\mathbf{Q}_{\gamma}\mathbf{P}+\mathbf{P}\mathbf{Q}^T_{\gamma} -2\kappa \mathbf{P}\geq \mathbf{0}$
The matrix 
\[ \mathbf{Q}_{\gamma}\mathbf{P}+\mathbf{P}\mathbf{Q}^T_{\gamma} -2\kappa \mathbf{P}
    = \begin{pmatrix}
      2\ (p_{12}-\kappa) & -\gamma + (\nu -2\kappa) p_{12} + p_{22} \\
       -\gamma + (\nu -2\kappa) p_{12} + p_{22} & 2\ ( -\gamma p_{12} + (\nu - \kappa) p_{22} )
      \end{pmatrix}
\]
is again real symmetric.
Hence it is positive semi-definite iff its diagonal elements and its determinant are non-negative, i.e.
% \begin{equation} \label{cond:P:1:mu}
  $p_{12}-\kappa \geq 0$, $-\gamma p_{12} + (\nu - \kappa) p_{22}\geq 0$,
% \end{equation}
 and 
 \begin{multline} \label{cond:P:2:mu}
  0 \leq \detR :=\det( \mathbf{Q}_{\gamma}\mathbf{P}+\mathbf{P}\mathbf{Q}^T_{\gamma} -2\kappa \mathbf{P} )\\
    = 4\ (p_{12}-\kappa) ( -\gamma p_{12} + (\nu - \kappa) p_{22} )
      - (-\gamma + (\nu - 2\kappa) p_{12} + p_{22})^2
 \end{multline}
for all $\gamma\in [\gamma_1,\gamma_2]$.
We summarize the conditions on the parameters $(p_{12},p_{22},\kappa)$:
\begin{enumerate}[label=\textbf{(C\arabic*)}]
% \item \label{C:p11} $p_{11}>0$
% \item \label{C:p22} $p_{22}>0$
 \item \label{C:detP} $\det(\mathbf{P})=p_{22}-p_{12}^2>0 \qquad \Leftrightarrow \qquad p_{22}>p_{12}^2\geq 0$,
 \item \label{C:kappa} $\kappa\geq0$,
 \item \label{C:Q11} $p_{12} \geq\kappa \,(\geq 0)$, % $p_{12}-\kappa \geq 0$
 \item \label{C:detQ} $\detR \geq 0$ for all $\gamma\in [\gamma_1,\gamma_2]$,
  % $\det( \mathbf{Q}_{\gamma}\mathbf{P}+\mathbf{P}\mathbf{Q}^T_{\gamma} -2\kappa \mathbf{P} ) \geq 0$ for all $\gamma\in [\gamma_1,\gamma_2]$
 \item \label{C:Q22} $-\gamma p_{12} + (\nu - \kappa) p_{22}\geq 0$ for all $\gamma\in [\gamma_1,\gamma_2]$.
\end{enumerate}
\begin{remark} \label{remark:C:Q22}
Condition~\ref{C:Q11} and a strict inequality in Condition~\ref{C:detQ} imply Condition~\ref{C:Q22}. 
Let, for some fixed $(p_{12},p_{22},\kappa)$, 
 the Conditions~\ref{C:Q11}--\ref{C:detQ} hold for a $\gamma$-interval with interior $\Gamma$.
Then~\ref{C:detQ} holds on $\Gamma$ with strict inequality;
 hence also \ref{C:Q22} holds on $\Gamma$.
By continuity \ref{C:detQ}--\ref{C:Q22} then also hold on $\overline \Gamma$.
Thus (except for the case of $\Gamma$ being the empty set) Condition~\ref{C:Q22} follows from Conditions~\ref{C:Q11}--\ref{C:detQ}.
% and will be neglected in the following. 
\end{remark}
\begin{definition}
A pair $(p_{12},p_{22})\in \Rpo\times\Rp$ is \emph{admissible},
 if there exist $\kappa_0\geq 0$ and $\gamma_0\in\R$
 such that \allConditions hold with $\kappa=\kappa_0$ and $\gamma_1=\gamma_2=\gamma_0$.
\end{definition}
\begin{lemma} \label{lemma:kappa:0}
If $(p_{12},p_{22})$ is admissible for some $\kappa_0\geq 0$ and $\gamma_0\in\R$,
 then $(p_{12},p_{22})$ is admissible also for all $\kappa\in[0,\kappa_0]$ and given $\gamma_0$.
\end{lemma}
\begin{proof}
The Conditions~\firstThreeConditions continue to hold for all $\kappa\in[0,\kappa_0]$ and given $\gamma_0$.
The admissible parameters $(p_{12},p_{22})$ define a symmetric positive definite matrix~$\mathbf{P}$
 satisfying $\mathbf{Q}_{\gamma_0}\mathbf{P}+\mathbf{P}\mathbf{Q}^T_{\gamma_0} \geq 2\kappa_0 \mathbf{P}$.
Due to $\mathbf{P}\geq \mathbf{0}$,
 $\mathbf{Q}_{\gamma_0}\mathbf{P}+\mathbf{P}\mathbf{Q}^T_{\gamma_0} \geq 2\kappa_0 \mathbf{P} \geq 2\kappa \mathbf{P}$ 
 for all $\kappa\in[0,\kappa_0]$.
Hence also Condition~\ref{C:detQ} is satisfied for all $\kappa\in[0,\kappa_0]$ and given $\gamma_0$. 
Since $p_{22}>0$, Condition~\ref{C:Q22} carries over to $\kappa\in[0,\kappa_0]$.
\end{proof}

We rewrite $\detR$ with respect to powers of $\gamma$ as
 \begin{equation} \label{detR:gamma}
  \detR = -\gamma^2 - (4 p_{12}^2 -2\nu p_{12} -2p_{22})\gamma - c(\kappa)
 \end{equation}
 with
%  \begin{align*} 
%   c(\kappa) :&= -4(p_{12}-\kappa)(\nu-\kappa)p_{22} + [(\nu-2\kappa)p_{12}+p_{22}]^2 \\
%    &= 4\kappa(\nu-\kappa) (p_{22}-p_{12}^2) + (\nu p_{12}-p_{22})^2
%     = 4\kappa(\nu-\kappa) \alpha_1 + \alpha_2 \,, 
%  \end{align*}
 \begin{equation} \label{eq:c} 
  c(\kappa) := % -4(p_{12}-\kappa)(\nu-\kappa)p_{22} + [(\nu-2\kappa)p_{12}+p_{22}]^2 &=
    4\kappa(\nu-\kappa) (p_{22}-p_{12}^2) + (\nu p_{12}-p_{22})^2
     = 4\kappa(\nu-\kappa) \alpha_1 + \alpha_2 \,, 
 \end{equation}
 with $\alpha_1:= (p_{22}-p_{12}^2)>0$ due to Condition~\ref{C:detP},
 and $\alpha_2:=(\nu p_{12}-p_{22})^2\geq 0$.
The function $c(\kappa)$ satisfies $c(0) =c(\nu) =\alpha_2\geq 0$,
 hence, $c(\kappa)$ is non-negative for all $\kappa\in[0,\nu]$
 and monotonically increasing for all $\kappa\in[0,\tfrac\nu 2]$.

\begin{lemma} \label{lemma:kappa}
Admissible pairs $(p_{12},p_{22})$ exist only for $\kappa\in [0,\tfrac\nu 2]$.
\end{lemma}
\begin{proof}
Assume $(p_{12},p_{22})$ is admissible for some $\kappa_0>\tfrac\nu 2$. % and $\gamma$.
Then $\gamma$ can be increased until $\detRoo=0$.
Due to Lemma~\ref{lemma:kappa:0}, $0 \geq \big(\diff{}{\kappa} \detRo\big) \big|_{\kappa=\kappa_0}$.
Moreover,
 \[ 0 \geq \big(\diff{}{\kappa} \detRo\big) \big|_{\kappa=\kappa_0} = -\Diff{c}{\kappa}(\kappa_0) = 8\alpha_1 (\kappa_0-\tfrac\nu 2) \]
 and $\alpha_1>0$ imply $\kappa_0-\tfrac\nu 2\leq 0$, contradicting our initial assumption.
\end{proof}
\begin{remark} \label{remark:gamma}
$\detR$ describes a parabola (as the function~\eqref{detR:gamma} of $\gamma$)
 and $\detR |_{\gamma=0} = -c(\kappa) \leq 0$ for $\kappa\in [0,\nu]$.
Therefore, each $\gamma$-interval with $\detR\geq 0$ is \underline{either} included in $\Rpo$ \underline{or} in $\Rno$.
But, in the latter case, $V''(x)\leq 0$ for all $x\in\R$, which would not give an integrable steady state.
Hence, only $\gamma\geq 0$ is relevant.
\end{remark}
Next we establish an important condition: $\sqrt{\gamma_2}-\sqrt{\gamma_1}\leq\nu$.
\begin{proposition} \label{prop:restriction}
Let $0\leq \gamma_1 < \gamma_2$ be given.
If and only if they satisfy the condition $\sqrt{\gamma_2}-\sqrt{\gamma_1}\leq\nu$,
 then there exists an admissible pair $(p_{12},p_{22})$
 satisfying Conditions~\allConditions for some $\kappa_0\geq 0$ and for all $\gamma\in[\gamma_1,\gamma_2]$.
\end{proposition}
The proof is deferred to Section~\ref{subsection:32}.

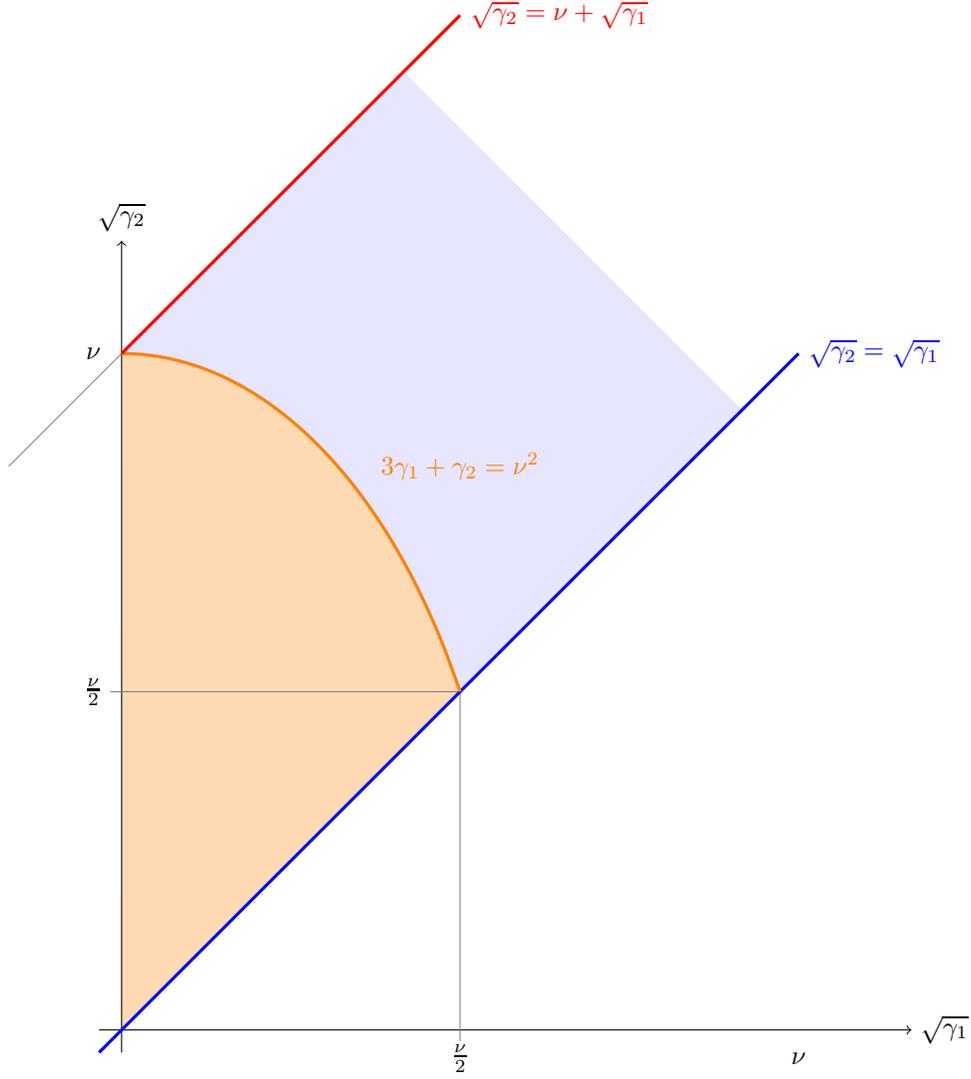
\begin{figure}
  \begin{tikzpicture}[scale=1.5]

%\node (BigCircle) at (6.5,2) [color=gray] {$\gamma_1 +\gamma_2 = \nu^2$};
%\node (SmallCircle) at (3,1.5) [color=gray] {$\gamma_1 +\gamma_2 = \tfrac{\nu^2}2$};
\node (alphaNuHalf) at (3,-0.25) {$\tfrac{\nu}2$};
%\node (alphaNuSqrtHalf) at (4.25,-0.25) {$\tfrac{\nu}{\sqrt 2}$};
\node (alphaNu) at (6,-0.25) {$\nu$};
\node (betaNuHalf) at (-0.25,3) {$\tfrac{\nu}2$};
%\node (betaNuSqrtHalf) at (-0.25,4.25) {$\tfrac{\nu}{\sqrt 2}$};
\node (betaNu) at (-0.25,6) {$\nu$};

\pgfpathmoveto{\pgfpointorigin}
\pgfpathlineto{\pgfpoint{3cm}{3cm}}
\pgfpatharcto{3.4641cm}{6cm}{0}{0}{1}{\pgfpoint{0cm}{6cm}}
%\pgfpatharc{45}{90}{3.4641cm and 6cm} \pgfusepath{stroke}
\pgfpathlineto{\pgfpointorigin}
\color{orange!30!white}
%\pgfusepath{stroke}
\pgfusepath{fill}
%\pgfpathcircle{\pgfpoint{2cm}{4.899cm}}{2pt} %test

\pgfpathmoveto{\pgfpoint{3cm}{3cm}}
%\pgfpatharc{45}{90}{3.4641cm and 6cm}
\pgfpatharcto{3.4641cm}{6cm}{0}{0}{1}{\pgfpoint{0cm}{6cm}}
%\pgfusepath{stroke}
\pgfpathlineto{\pgfpoint{2.5cm}{8.5cm}}
\pgfpathlineto{\pgfpoint{5.5cm}{5.5cm}}
\pgfpathlineto{\pgfpoint{3cm}{3cm}}
\color{blue!10!white}
%\pgfusepath{stroke}
\pgfusepath{fill}

  \color{black}
%  \draw[very thin,color=gray] (-1,-1) grid (2,1);
  \draw[->] (-0.2,0) -- (7,0) node[right] {$\sqrt{\gamma_1}$};
  \draw[->] (0,-0.2) -- (0,7) node[above] {$\sqrt{\gamma_2}$};
  \draw[color=gray] (-0.1,3) -- (3,3);
  \draw[color=gray] (3,-0.1) -- (3,3);

  \draw[very thick,domain=-0.2:6,smooth,color=blue] plot (\x,\x) node[right] {$\sqrt{\gamma_2}=\sqrt{\gamma_1}$};
%  \draw[thin,color=green!60!black,smooth,domain=0:(6/sqrt(2))] plot (\x,{sqrt(36-(\x)*(\x))});
  %\draw[thin,color=gray,smooth,domain=0:(6/sqrt(2))] plot (\x,{sqrt(36-(\x)*(\x))});
  %\draw[thin,color=gray,smooth,domain=(6/sqrt(2)):6] plot (\x,{sqrt(36-(\x)*(\x))});
  %\draw[thin,color=gray,smooth,domain=0:3] plot (\x,{sqrt(18-(\x)*(\x))});
  %\draw[thin,color=gray,smooth,domain=3:(6/sqrt(2))] plot (\x,{sqrt(18-(\x)*(\x))});
  \draw[very thick,color=orange,smooth,domain=0:3] plot (\x,{sqrt(36-3*(\x)*(\x))});
%  \draw[thin,color=gray,smooth,domain=0:3] plot (\x,{6-\x});
%  \draw[thin,color=gray,smooth,domain=-1:0] plot (\x,{6-\x});
%  \draw[thin,color=gray,smooth,domain=3:7] plot (\x,{6-\x}) node[right] {$\sqrt{\gamma_2}=\nu-\sqrt{\gamma_1}$};
  \draw[very thick,color=red,smooth,domain=0:3] plot (\x,{6+\x}) node[right] {$\sqrt{\gamma_2}=\nu+\sqrt{\gamma_1}$};
  \draw[thin,color=gray,smooth,domain=-1:0] plot (\x,{6+\x});

\node (Ellipse) at (3,5) [color=orange] {$3\gamma_1 +\gamma_2 = \nu^2$};

  \end{tikzpicture}
 \caption{A visualization of the $(\gamma_1, \gamma_2)$ subset
   such that for a given $0<\nu$ there exist parameters $(p_{11},p_{12},p_{22},\kappa)\in \Rp\times\Rpo\times\Rp\times [0,\tfrac\nu 2]$
   satisfying conditions~\allConditions according to Theorem~\ref{thm:kappa}.
  }
 \label{fig:f:modification}
\end{figure}

\begin{theorem} \label{thm:kappa}
Suppose $0<\nu$ and $0\leq\gamma_1 <\gamma_2$ satisfy $\sqrt{\gamma_2}-\sqrt{\gamma_1}\leq\nu$.
Then the following $(p_{12},p_{22})\in\Rpo\times\Rp$ are all admissible pairs for $\kappa_{max}\in[0,\tfrac\nu 2]$,
 the maximal possible value of $\kappa$, and for all $\gamma\in[\gamma_1,\gamma_2]$:
\begin{enumerate}[label=\textbf{(B\arabic*)}]
\item \label{cone-tip}
If $3\gamma_1+\gamma_2 \leq \nu^2$ then $\kappa_{max} = \tfrac\nu2-\tfrac12\sqrt{\nu^2-4\gamma_1}$ and
 \[ (p_{12},p_{22})
     = \big( \tfrac{\nu}2 + \tfrac\tau 2 \sqrt{\nu^2-3\gamma_1-\gamma_2}, \tfrac12 (\nu^2 -2\gamma_1 + \tau \nu\sqrt{\nu^2-3\gamma_1-\gamma_2}) \big)
 \]  
 with $\tau\in[-1,1]$ satisfy the conditions~\allConditions.
\item \label{outer-region}
If $3\gamma_1+\gamma_2 > \nu^2$ then %and $\sqrt{\gamma_2}\leq \nu+\sqrt{\gamma_1}$ then
 $\kappa_{max} = \tfrac{\nu}2 - \tfrac{\gamma_2-\gamma_1}{2\sqrt{2(\gamma_1+\gamma_2)-\nu^2}}$,
 $p_{12} = \tfrac\nu 2$ and $p_{22} = \tfrac{\gamma_2 +\gamma_1}2$ satisfy the conditions~\allConditions. 
\end{enumerate}
\end{theorem}
The proof is deferred to Section~\ref{subsection:33}.
\begin{remark}
The expressions for $\kappa_{max}$, $p_{12}$ and $p_{22}$ are continuous at the interface $3\gamma_1+\gamma_2 =\nu^2$.
\end{remark}

Following Theorem~\ref{thm:kappa},
 we obtain for given $\nu>0$ and $0\leq\gamma_1 < \gamma_2\leq (\nu+\sqrt{\gamma_1})^2$
 that a symmetric positive definite matrix~$\mathbf{P}$ and $\kappa=\kappa_{max}\geq 0$ exist
 such that~\eqref{C:detR} holds.
Hence, the modified entropy method yields the following theorem.

\begin{theorem}\label{kinFPdecay}
Let $\psi$ generate an admissible entropy and let $f$ be the solution to the kinetic Fokker-Planck equation (\ref{kinFP})
 with a potential $V(x)$ satisfying~\eqref{as:potential} and
 an initial state~$f_0$ satisfying $S_\psi(f_0)<\infty$. 
Under the assumptions of Theorem~\ref{thm:kappa} we then have:
\begin{equation}\label{e-convergence-kinFP}
  e_\psi(f(t)|f_\infty) \le c\,S_\psi(f_0) \e^{-2\kappa_{max} t},\quad t\geq 0\,,
\end{equation}
for some constant $c>0$ independent of $f_0$ and $\kappa_{max}$ given in Theorem~\ref{thm:kappa}.
\end{theorem} 
\begin{proof}
We already noticed that, following Theorem~\ref{thm:kappa},
 we obtain for given $\nu>0$ and $0\leq\gamma_1 < \gamma_2\leq (\nu+\sqrt{\gamma_1})^2$
 a symmetric positive definite matrix~$\mathbf{P}$ and $\kappa=\kappa_{max}\geq 0$
 such that~\eqref{C:detR} holds.
Consequently, inequality~\eqref{kinFP-Sineq-2} follows 
 and implies the exponential decay of the modified entropy dissipation functional %$S_\psi(f(t))$
 \begin{equation}\label{S-convergence-kinFP}
  S_\psi(f(t)) \leq S_\psi(f_0) \e^{-2\kappa t},\quad t\geq 0.
 \end{equation}
Moreover, due to Lemma~\ref{Sfiniteefinite}, the convex Sobolev inequality 
 \[ e_\psi(g|f_\infty) \leq \frac1{2\lambda_P} S_\psi(g)\,, \qquad \forall g\in L^1_+(\R^d) \xx{with} \intRd g\d[\xi] =1 \]
 holds, where $\lambda_P>0$ is the smallest eigenvalue of $\P$.
Thus \eqref{e-convergence-kinFP} follows from \eqref{S-convergence-kinFP}.
\end{proof}

In a previous work~\cite[\S 7]{ArEr14} the authors considered potentials of the form
\begin{equation}\label{potential}
   V(x)=\omega_0^2\,\frac{x^2}{2}+\widetilde V(x) \quad\xx{with} \norm{\widetilde V''}_{L^\infty}<\infty, \quad\xx{and}\omega_0\ne0.
\end{equation}
% \begin{equation}\label{potential}
%    V(x)=\omega_0^2\,\frac{x^2}{2}+\widetilde V(x)\quad\mbox{ with }\;|\widetilde \gamma|\le const.\; \forall\,x\in\R,
%    \quad\mbox{ and }\;\omega_0\ne0.
% \end{equation}
% Corresponding to the ``unperturbed'' potential $\omega_0^2\,\frac{x^2}{2}$, we define the constant matrix 
% $$
%   Q_0:= \left(\begin{array}{cc} 0 & 1 \\ -\omega_0^2 & \nu \end{array}\right)\in\R^{2\times 2}\,,
% $$
% having the (real or complex) eigenvalues $\lambda_{1,2}=\frac{\nu}{2}\pm \sqrt{\frac{\nu^2}{4}-\omega_0^2}$.
Following the proof of Lemma \ref{Pdefinition},
 a matrix~$\mathbf{P}$, corresponding to the potential term $\omega_0^2\,\frac{x^2}{2}$,
 can be constructed as
% we choose the positive definite matrix $P$ corresponding to $Q_0$, using $b_j=1$ in \eqref{simpleP}. 
%This choice of $b_j$ is for simplicity of the presentation only, and the final result could be optimised w.r.t.\ the quotient $b_1/b_2$. Let
\begin{equation}\label{Pdefinition1}
  \mathbf{P} := \begin{cases}
         \begin{pmatrix} 2 & \nu \\ \nu & \nu^2-2\omega_0^2 \end{pmatrix} &\Xx{if} 4\omega_0^2<\nu^2\,, \\[4mm]
         \begin{pmatrix} 2 & \nu \\ \nu & 2\omega_0^2 \end{pmatrix}       &\Xx{if} 4\omega_0^2>\nu^2\,,
       \end{cases}
\end{equation}
and  
\begin{equation}\label{kappa-def}
  2\kappa_0 := \begin{cases}
                \nu-\sqrt{\nu^2-4\omega_0^2}, &\xx{if} 4\omega_0^2<\nu^2\,, \\
                \nu, &\xx{if} 4\omega_0^2>\nu^2\,.
               \end{cases} 
\end{equation}
% \begin{proposition}[{\cite[Proposition 7.3]{ArEr14}}]
% \label{Sconvergence-kinFP-ArEr14}
% Let $4\omega_0^2\ne \nu^2$and let $\widetilde V$ from \eqref{potential} satisfy for some fixed $\lambda\in(0,2\kappa_0)$ and $\forall\,x\in\R$:
% \[ |\widetilde \gamma| \le \sqrt{|\omega_0^2-\nu^2/4|}\,\lambda\, \]
% for the matrix $\mathbf{P}$ chosen in \eqref{Pdefinition1}. Then
% \begin{align}\label{S-convergence-kinFP-ArEr14}
%  S_\psi(f(t)) &\leq S_\psi(f_0)e^{-(2\kappa_0-\lambda)t},\quad t\geq 0,
% \end{align}
% with $\kappa_0$ defined in \eqref{kappa-def}.
% \end{proposition}
 \begin{proposition}[{\cite[Proposition 7.3]{ArEr14}}] \label{prop:kinFPdecay:ArEr14}
  Let $4\omega_0^2\ne \nu^2$ and let $\widetilde V$ from~\eqref{potential}
   satisfy $\norm{\widetilde V''}_{L^\infty} < \sqrt{\abs{\nu^2 - 4\omega_0^2}} \kappa_0$
   with $\kappa_0$ defined in~\eqref{kappa-def}.
  Then the modified entropy dissipation~$S_\psi (f(t))$ with the matrix~$\mathbf{P}$ chosen in~\eqref{Pdefinition1} satisfies
   \[ S_\psi (f(t)) \leq S_\psi(f_0)\, \e^{-2\Big(\kappa_0-\frac{\norm{\widetilde V''}_{L^\infty}}{\sqrt{\abs{\nu^2 - 4\omega_0^2}}}\Big)t}
       \quad \text{for } t\geq 0.
   \]
 \end{proposition}
\begin{theorem}[{\cite[Theorem 7.4]{ArEr14}}] \label{thm:kinFPdecay:ArEr14}
Let $\psi$ generate an admissible entropy and let $f$ be the solution to the kinetic Fokker-Planck equation (\ref{kinFP})
 with an initial state~$f_0$ satisfying $S_\psi(f_0)<\infty$. 
Under the assumptions of Proposition~\ref{prop:kinFPdecay:ArEr14} we then have:
\begin{equation}\label{e-convergence-kinFP:ArEr14}
  e_\psi(f(t)|f_\infty) \le c\,S_\psi(f_0)\, \e^{-2\Big(\kappa_0-\frac{\norm{\widetilde V''}_{L^\infty}}{\sqrt{\abs{\nu^2 - 4\omega_0^2}}}\Big)t},\quad t\geq 0\,,
\end{equation}
for some constant $c>0$ independent of $f_0$.
\end{theorem}

The defective case $4\omega_0^2=\nu^2$ is omitted in~\cite{ArEr14}; 
 but it is noted that a matrix $\mathbf{P}=\mathbf{P}(\eps)$ could easily be constructed from the proof of Lemma \ref{Pdefinition} (ii). 

To compare the decay rates in Theorem~\ref{kinFPdecay} and Theorem~\ref{thm:kinFPdecay:ArEr14}, % Proposition~\ref{prop:kinFPdecay:ArEr14},
 we have to relate the parameters in Theorem~\ref{thm:kinFPdecay:ArEr14}
 with the parameters $\gamma_1$ and $\gamma_2$ in Theorem~\ref{kinFPdecay}.
Moreover, in Theorem~\ref{thm:kinFPdecay:ArEr14},
 the parameter $\omega_0$ has to be chosen as to optimize the decay rate.
% to $\gamma_1:=\inf_{x\in\R} \gamma$ and $\gamma_2:=\sup_{x\in\R} \gamma$.
\begin{proposition} \label{prop:kinFPdecay:ArEr14:alpha+beta}
If { }$0<\nu$ and $V(x)$ with $0<\gamma_1 :=\inf V''<\sup V''=:~\gamma_2$ are given,
 then the largest rate $\widetilde \kappa := \sup_{\omega_0} \kappa_0-\frac{\norm{V''-\omega_0^2}_{L^\infty}}{\sqrt{\abs{\nu^2 - 4\omega_0^2}}}$
 in Proposition~\ref{prop:kinFPdecay:ArEr14} is equal to $\kappa_{max}$ in Theorem~\ref{thm:kappa}.
\end{proposition}
\begin{proof}
For given $0<\nu$ and $\gamma_1\leq V''(x)\leq \gamma_2$,
 we need to decompose $V''$ as $V''=\omega_0^2 + \widetilde V''$ 
 such as to maximize the function
 \[ \widetilde \kappa (\omega_0)
     = \kappa_0(\omega_0) -\frac{\norm{\widetilde V''}_{L^\infty}}{\sqrt{\abs{\nu^2 - 4\omega_0^2}}}
     = \kappa_0(\omega_0) -\frac{\max\{\abs{\gamma_2-\omega_0^2},\abs{\gamma_1-\omega_0^2} \}}{\sqrt{\abs{\nu^2 - 4\omega_0^2}}}
 \]
 with $\kappa_0(\omega_0)$ given in~\eqref{kappa-def}.
After distinguishing several cases, one obtains that $\widetilde \kappa(\omega_0)=\kappa_{max}$ for  
 \[ \omega_0^2 = 
      \begin{cases}
       -\gamma_1+\tfrac{\nu^2}{2} &\xx{for} \nu^2 \leq 3\gamma_1+\gamma_2 \,, \\
         % \makegray{\Xx{satifies $\gamma_1\leq \omega_0^2\leq \gamma_2$}},, \\
       \tfrac{\gamma_1+\gamma_2}{2} &\xx{for} 3\gamma_1+\gamma_2<\nu^2 \,.
      \end{cases}
 \]
% \[ \omega_0^2 = 
%      \begin{cases}
%       -\gamma_1+\tfrac{\nu^2}{2} &\xx{for} 2\gamma_1 +2\gamma_2 <\nu^2 \,, \\
%         % \Xx{satifies $\gamma_2\leq \omega_0^2$}}\,, \\
%       -\gamma_1+\tfrac{\nu^2}{2} &\xx{for} 2\gamma_1 +2\gamma_2 <\nu^2<3\gamma_1+\gamma_2 \,, \\
%         % \makegray{\Xx{satifies $\gamma_1\leq \omega_0^2\leq \gamma_2$}},, \\
%       \tfrac{\gamma_1+\gamma_2}{2} &\xx{for} 3\gamma_1+\gamma_2<\nu^2 \,.
%      \end{cases}
% \]
\end{proof}
In case $\gamma_1=\gamma_2$,
 the admissible potentials in $\S \ref{subsection:31}$ are $V(x)=\gamma_1 \tfrac{x^2}{2} +c_1 x +c_2$ for any constants $c_1$, $c_2\in\R$.
Consider the limit $\gamma_1\to\gamma_2$ in Theorem~\ref{thm:kappa}:
 we recover in the limit $\gamma_1\to\gamma_2$ the decay rate and matrix $\mathbf{P}$ from~\cite[$\S$7]{ArEr14}
 by choosing $\tau=0$ in the case \ref{cone-tip}.

\subsection{Proof of Proposition~\ref{prop:restriction}} \label{subsection:32}

\begin{lemma} \label{lemma:gamma}
Let $(p_{12},p_{22})$ be admissible for some $\kappa_0\geq 0$ and $\gamma_0>0$.
Then $(p_{12},p_{22})$ is also admissible for $\kappa_0$
 and \emph{exactly} for $\gamma\in[\gamma_1,\gamma_2]$ with 
 \begin{equation} \label{gamma:pm}
  \gamma_{1,2} = -2 p_{12}^2 +\nu p_{12} +p_{22} \mp \sqrt{(-2 p_{12}^2 +\nu p_{12} +p_{22})^2-c(\kappa_0)} \geq 0 \,.
 \end{equation}
\end{lemma}
\begin{proof}
Conditions~\ref{C:detP}--\ref{C:Q11} and $\detRoo\geq 0$ hold, since $(p_{12},p_{22})$ is admissible.
%$\detRoo\geq 0$, since $(p_{12},p_{22})$ is admissible.
Consequently, the equation $\detRko=0$ has (one or two) real solutions $\gamma_1 \leq\gamma_2$,
 satisfying $0<\gamma_0\in[\gamma_1,\gamma_2]$
 and Condition~\ref{C:detQ} holds for all $\gamma\in[\gamma_1,\gamma_2]$.
Due to Remark~\ref{remark:gamma}, $0\leq\gamma_1 \leq\gamma_2$.  
Moreover, $\detRko>0$ for $\gamma\in(\gamma_1,\gamma_2)$ and Condition~\ref{C:Q11} holds.
Hence, Condition~\ref{C:Q22} follows for all $\gamma\in(\gamma_1,\gamma_2)$,
 and %Condition~\ref{C:Q22} follows
 for all $\gamma\in[\gamma_1,\gamma_2]$ by continuity,
 see Remark~\ref{remark:C:Q22}.
\end{proof}
\begin{remark} \label{remark:gamma:max}
Due to \eqref{eq:c}, Lemma~\ref{lemma:kappa} and Lemma~\ref{lemma:gamma},
 the possible $\gamma$-interval decreases strictly monotonically with $\kappa$
 (as expected from $\mathbf{Q}_\gamma \mathbf{P}+\mathbf{P}\mathbf{Q}_\gamma^T\geq 2\kappa \mathbf{P}$).
For any fixed $\nu,p_{12},p_{22}$, 
 the largest possible $\gamma$-interval is obtained for $\kappa=0$,
 i.e. with $c(0) =\alpha_2 =(\nu p_{12}-p_{22})^2 \geq 0$.
\end{remark}

\begin{proof}[Proof of Proposition~\ref{prop:restriction}]
Following Lemma~\ref{lemma:gamma} and Remark~\ref{remark:gamma:max},
 we seek the largest $\gamma$-interval $[\gamma_1,\gamma_2]$
 (which maximizes $\gamma_2-\gamma_1$ for fixed $p_{12},p_{22}$) 
 and consequently set $\kappa=0$. % (which ensures Condition~\ref{C:kappa}).
The expressions for $\gamma_1 < \gamma_2$ in~\eqref{gamma:pm} and $\kappa=0$ yield 
\begin{align*}
 -2 p_{12}^2 +\nu p_{12} +p_{22} &= \tfrac{\gamma_1 +\gamma_2}2 \,, \\
 2 \sqrt{(-2 p_{12}^2 +\nu p_{12} +p_{22})^2-c(0)} &= \gamma_2 -\gamma_1 \,,
\intertext{or, equivalently with $\alpha_1=p_{22}-p_{12}^2$ and $\alpha_3:=p_{12}(\nu-p_{12})$,}
 \alpha_1 + \alpha_3 = (p_{22}-p_{12}^2) + (p_{12}(\nu-p_{12})) &= \tfrac{\gamma_1 +\gamma_2}2 =: \beta_2 \,, \\
 \alpha_1 \ \alpha_3 = (p_{22}-p_{12}^2) (p_{12}(\nu-p_{12})) &= \big(\tfrac{\gamma_2-\gamma_1}4\big)^2 =: \beta_1 \geq 0 \,.
\end{align*}
Combining the last two equations, we derive 
 \[ -\alpha_3 \ (\beta_2 -\alpha_3) + \beta_1 = 0 \]
 which has two real solutions
 $\alpha_{3,\pm}
%      = \tfrac{\beta_2}2 \pm \sqrt{\tfrac{\beta_2^2}4 - \beta_1}
     = \tfrac{\beta_2}2 \pm \tfrac12 \sqrt{\beta_2^2 - 4\beta_1}
     = \tfrac14 \big( \sqrt{\gamma_2} \pm\sqrt{\gamma_1}\big)^2 $.
We recall $\alpha_3=p_{12}(\nu-p_{12})$,
 which has real solutions $p_{12}$ if and only if $\alpha_3\leq \tfrac{\nu^2}4$.
Due to $0\leq \gamma_1 < \gamma_2$, this restriction is equivalent to 
 \begin{equation} \label{eq:restriction}
  2 \sqrt{\alpha_{3,\pm}} = \sqrt{\gamma_2} \pm\sqrt{\gamma_1} \leq \nu \,.
 \end{equation}
For $0\leq\gamma_1 < \gamma_2$, Condition~\eqref{eq:restriction} with ``$+$'' is more restrictive than with ``$-$''.
Therefore, we consider in the sequel $\alpha_{3,-}=\tfrac14 ( \sqrt{\gamma_2} -\sqrt{\gamma_1})^2$,
 in accordance with the key assumption in Proposition~\ref{prop:restriction}.

Condition~\ref{C:detP} is equivalent to $\alpha_1>0$.
Due to $\alpha_1 +\alpha_3 = \tfrac{\gamma_1 +\gamma_2}2$ and $\alpha_3 =\alpha_{3,-} =\tfrac14 ( \sqrt{\gamma_2} -\sqrt{\gamma_1})^2$,
 we deduce $\alpha_1 = \tfrac14 ( \sqrt{\gamma_2} +\sqrt{\gamma_1})^2>0$
 since $0\leq\gamma_1 < \gamma_2$.
Condition~\ref{C:kappa} is satisfied due to our choice $\kappa=0$.
Next, $p_{12}(\nu-p_{12})=\alpha_{3,-}$ has two real solutions $0\leq p_{12,-}\leq p_{12,+} \leq \nu$ with $p_{12,-} +p_{12,+} =\nu$;
 hence, Condition~\ref{C:Q11} holds.
Due to our construction starting from~\eqref{gamma:pm},
 Condition~\ref{C:detQ} holds.
Finally, Condition~\ref{C:Q22} follows again from Conditions~\ref{C:Q11}--\ref{C:detQ} and Remark~\ref{remark:C:Q22}.
\end{proof}

\subsection{Proof of Theorem~\ref{thm:kappa}} \label{subsection:33}
\begin{lemma} \label{lemma:p12}
Let $(p_{12},p_{22})$ be admissible for some $\kappa_0\geq 0$. Then $p_{12}\leq \nu-\kappa_0$.
\end{lemma}
\begin{proof}
By Lemma~\ref{lemma:gamma}, $\gamma_{1,2}\geq 0$.
Hence, the discriminant in~\eqref{gamma:pm} satisfies 
\begin{align*}
 0 &\leq (-2 p_{12}^2 +\nu p_{12} +p_{22})^2 -(\nu p_{12} -p_{22})^2 - 4\kappa_0 \alpha_1 (\nu-\kappa_0) \\
   &= 4\alpha_1 [p_{12} (\nu-p_{12}) - \kappa_0(\nu-\kappa_0) ]
\end{align*}
which is equivalent to $\kappa_0\leq p_{12}\leq \nu-\kappa_0$ since $\alpha_1>0$.
\end{proof}
\begin{remark}
The maximal value of $\kappa$ given by Lemma~\ref{lemma:kappa}, i.e. $\kappa=\tfrac\nu 2$, is possible, but only for quadratic potentials:
It implies $p_{12}=\tfrac\nu 2$, $\gamma=\gamma_1=\gamma_2=p_{22}>\tfrac{\nu^2}4$ (due to \ref{C:detP}).
\end{remark}
\begin{lemma} \label{lemma:P}
For $\gamma,\gamma_1,\gamma_2$ given as in Theorem~\ref{thm:kappa},
 let $\kappa_{max}$ denote the maximal decay rate
 and let $\mathcal{P}$ denote the set of admissible pairs $(p_{12},p_{22})\in\Rpo\times\Rp$
 (w.r.t. the whole interval $[\gamma_1,\gamma_2]$). Then,
\begin{enumerate}[label=(\alph*)]
\item \label{P:convex+compact}
 $\mathcal{P}$ is convex and compact;
  and $\mathcal{P}$ lies in the interior of the set defined by the inequalities \ref{C:detP} and \ref{C:Q11};
\item \label{P:line-segment}
 $\mathcal{P}$ is a finite, possibly one-pointed, line segment with $p_{12}\in[p_{12}^-,p_{12}^+]$.
\end{enumerate}
\end{lemma}
\begin{proof}
\ref{P:convex+compact} 
 The convexity is clear from~\eqref{C:detR:gamma}.
 
 \ref{C:Q11} and Lemma~\ref{lemma:p12} imply the boundedness of $p_{12}$.
 \eqref{cond:P:2:mu} yields an upper bound for $p_{22}$
  (by considering the balance of $p_{22}^2$ and the linear terms in $p_{22}$).
 For $0\leq \gamma_1 < \gamma_2$,
  no points of $\overline{\mathcal{P}}$ can lie on the curve $p_{22}=p_{12}^2$ (cf.~\ref{C:detP}),
  since otherwise we would obtain: $\detR=-(\gamma+p_{12}^2-\nu p_{12})^2$,
  and~\ref{C:detQ} would only be true for a single value of $\gamma$.
 Hence, the strict inequality \ref{C:detP} also holds for accumulation points of $\mathcal{P}$
  (for \ref{C:kappa}--\ref{C:Q22} this is trivial).
 This implies that the bounded set $\mathcal{P}$ is closed. Hence, $\mathcal{P}$ is compact.

 By the same argument we have for all $(p_{12},p_{22})\in\mathcal{P}$:
 \begin{equation} \label{eq:C:Q11:strict}
  p_{12} > \kappa ,
 \end{equation}
 since otherwise $\detR=-(\gamma -(\nu -2\kappa)\kappa -p_{22})^2$.
 Hence, $\mathcal{P}=\overline{\mathcal{P}}$ lies in the interior of the set defined by the inequalities \ref{C:detP} and \ref{C:Q11}.
 
\ref{P:line-segment}
For each fixed $(p_{12},p_{22})\in\mathcal{P}$, we have
 \begin{equation} \label{C:detR:endpoints}
  \detRR{\kappa_{max}}{\gamma_1}=0 \quad \text{or} \quad \detRR{\kappa_{max}}{\gamma_2}=0
 \end{equation} 
 (or both): 
 Otherwise, due to Remark~\ref{remark:gamma:max} and~\eqref{eq:C:Q11:strict},
  $\kappa_{max}$ could be increased slightly, which contradicts maximality of $\kappa_{max}$.
 For fixed $p_{12}$, assume now that $\mathcal{P}\big|_{p_{12}} := \{ p>0 \,|\, (p_{12},p)\in\mathcal{P} \}$
  is \underline{not one point}, but rather a closed interval (due to the convexity of $\mathcal{P}$).
 Then, one of the equations in~\eqref{C:detR:endpoints} holds for more than two values of $p_{22}$.
 But this is impossible, since $\detR=0$ is a quadratic equation for $p_{22}$ (cf.~\eqref{C:detR:gamma}).
 Hence, $\mathcal{P}\big|_{p_{12}}$ consists only of one point and $\mathcal{P}$ is a line segment.
\end{proof}

By Lemma~\ref{lemma:P}, $\mathcal{P}$ is uniquely determined by its endpoints.

\begin{lemma} \label{lemma:P:endpoints}
Let $\gamma,\gamma_1,\gamma_2$ be given as in Lemma~\ref{lemma:P}.
For an endpoint $(\overline{p}_{12},\overline{p}_{22})\in\mathcal{P}$ we have $\detRR{\kappa_{max}}{\gamma_1}=\detRR{\kappa_{max}}{\gamma_2}=0$.
\end{lemma}
\begin{proof}
W.l.o.g. we now assume that $\detRR{\kappa_{max}}{\gamma_1}=0$ and $\detRR{\kappa_{max}}{\gamma_2}>0$.
So the inequalities \ref{C:detQ} for $\gamma=\gamma_2$ and \ref{C:Q11} hold strictly,
 as well as \ref{C:Q22} for $\gamma=\gamma_2$ (due to Remark~\ref{remark:C:Q22}).
Hence, \allConditions also hold for $\gamma=\gamma_2$
 and all $(\tilde p_{12},\tilde p_{22})$ in a small neighborhood of $(\overline{p}_{12},\overline{p}_{22})$.
 
Finally we consider, for $p_{12}$ fixed, 
 $\detRR{\kappa_{max}}{\gamma_1}=0$ as a quadratic equation for $p_{22}$.
The discriminant for its real solvability reads 
 \[ [p_{12}\nu - 2\kappa(\nu-\kappa) +\gamma_1]^2 + [-\gamma_1+(\nu-2\kappa)p_{12}]^2 +4(p_{12}-\kappa)\gamma_1 p_{12}\,. \]
For $\gamma_1>0$ this is positive due to \eqref{eq:C:Q11:strict},
 and for $\gamma_1=0$ since $\kappa<\tfrac\nu 2$.
Hence, $\detRR{\kappa_{max}}{\gamma_1}=0$ is also solvable for $p_{22}$,
 if $p_{12}$ lies in a small neighborhood of $\overline{p}_{12}$.
Thus, $\overline{p}_{12}$ is not an endpoint of the line segment $\mathcal{P}$.
\end{proof}

\begin{proof}[Proof of Theorem~\ref{thm:kappa}]
\underline{S\,t\,e\,p\; 1\,:}
For given $0\leq \gamma_1 < \gamma_2$,
 we shall first find admissible endpoints $(p_{12},p_{22})\in \mathcal{P}$ such that~\allConditions hold
 exactly for all $\gamma\in[\gamma_1,\gamma_2]$ with the maximal $\kappa\in[0,\tfrac\nu 2]$.
The expressions for $\gamma_1 < \gamma_2$ in \eqref{gamma:pm} yield 
\begin{align}
 -2 p_{12}^2 +\nu p_{12} +p_{22} &= \tfrac{\gamma_1 +\gamma_2}2 \,, \nonumber \\
 \sqrt{(-2 p_{12}^2 +\nu p_{12} +p_{22})^2-c(\kappa)} &= \tfrac{\gamma_2 -\gamma_1}2 \,, \nonumber
\intertext{or, equivalently with $\alpha_1=p_{22}-p_{12}^2$ and $\alpha_3:=p_{12}(\nu-p_{12})$,}
 \alpha_1 + \alpha_3 = (p_{22}-p_{12}^2) + (p_{12}(\nu-p_{12})) &= \tfrac{\gamma_1 +\gamma_2}2 =: \beta_2 \,, \label{eq:beta2} \\
 \alpha_1 \ [\alpha_3 - \kappa\ (\nu-\kappa)] &= \big(\tfrac{\gamma_2-\gamma_1}4\big)^2 =: \beta_1 > 0 \,. \label{eq:beta1}
\end{align}
For the line $\alpha_3 =\beta_2 -\alpha_1$ to intersect
 the hyperbola $\alpha_3 =\tfrac{\beta_1}{\alpha_1} + \kappa\ (\nu-\kappa)$ at some $\alpha_1>0$,
 we require that $0\leq \kappa\ (\nu-\kappa)< \beta_2$, see also Figure~\ref{fig:hyperbola+line}.

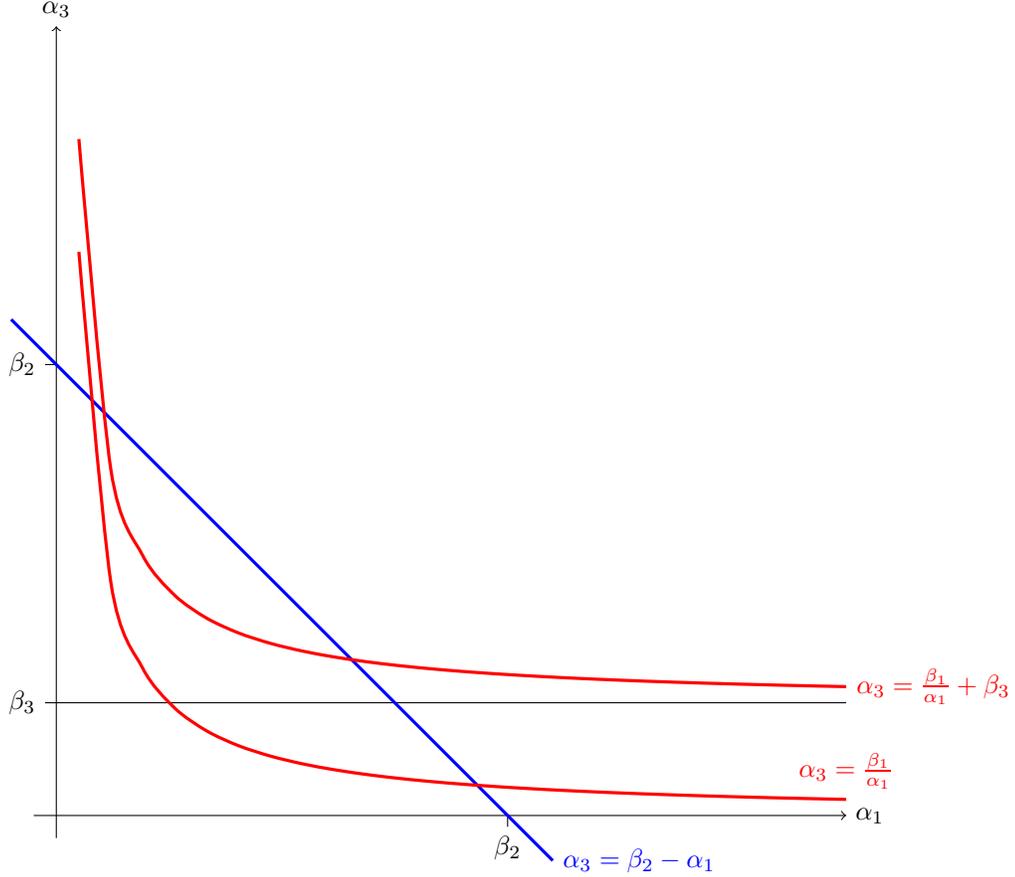
\begin{figure}
  \begin{tikzpicture}[scale=1.5]

  \color{black}
  \draw[->] (-0.2,0) -- (7,0) node[right] {$\alpha_1$};
  \draw[->] (0,-0.2) -- (0,7) node[above] {$\alpha_3$};

  \draw[-] (4,0) -- (4,-0.1) node[below] {$\beta_2$};
  \draw[-] (0,4) -- (-0.1,4) node[left] {$\beta_2$};
  \draw[-] (7,1) -- (-0.1,1) node[left] {$\beta_3$};

  \draw[very thick,domain=-0.4:4.4,smooth,color=blue] plot (\x,4-\x) node[right] {$\alpha_3=\beta_2-\alpha_1$};
  \draw[very thick,color=red,smooth,domain=0.2:7] plot (\x,{(1/\x)+1}) node[right] {$\alpha_3=\tfrac{\beta_1}{\alpha_1} + \beta_3$};
  \draw[very thick,color=red,smooth,domain=0.2:7] plot (\x,{(1/\x)}) node[above] {$\alpha_3=\tfrac{\beta_1}{\alpha_1}$};

  \end{tikzpicture}
 \caption{For the line $\alpha_3 =\beta_2 -\alpha_1$ to intersect
 the hyperbola $\alpha_3 =\tfrac{\beta_1}{\alpha_1} + \beta_3$ at some $\alpha_1>0$,
 we require that $0\leq \beta_3< \beta_2$.
  }
 \label{fig:hyperbola+line}
\end{figure}

The solutions of~\eqref{eq:beta2}--\eqref{eq:beta1} read
 \begin{equation} \label{alpha3:kappa}
  \alpha_{3,\pm} = \tfrac{\beta_{2} + \kappa\ (\nu-\kappa)}2 \pm \tfrac12 \sqrt{(\beta_2 - \kappa\ (\nu-\kappa))^2 - 4\beta_1}\,.
 \end{equation}
We seek the maximum $\kappa\in[0,\tfrac\nu 2]$ such that $\alpha_{3,\pm}\in\R$
 (for $\kappa=0$ this always holds by the proof of Proposition~\ref{prop:restriction}).
This maximal value is either obtained as $\kappa=\tfrac\nu 2$ 
 or when the discriminant of~\eqref{alpha3:kappa} is zero.
The latter case implies 
 \[ -\kappa\ (\nu-\kappa) = 2\sqrt{\beta_1} -\beta_2 = -\gamma_1 \,. \]
This is solvable (for $\kappa$) in $\R$ iff $\sqrt{\gamma_1}\leq \tfrac\nu 2$,
 yielding $\kappa=\tfrac\nu 2 - \sqrt{\tfrac{\nu^2}4 -\gamma_1} \in [0,\tfrac\nu 2]$.
Hence, for the solvability of \eqref{eq:beta2}--\eqref{eq:beta1} in $\R$, we obtain
 \[ \kappa_{max} \leq \widehat \kappa :=
     \begin{cases}
      \tfrac\nu 2 - \sqrt{\tfrac{\nu^2}4 -\gamma_1} &\xx{for} \sqrt{\gamma_1}\leq \tfrac\nu 2 \,, \\
      \tfrac\nu 2                                   &\xx{for} \sqrt{\gamma_1}>    \tfrac\nu 2 \,.                               
     \end{cases}
 \]
Using $\widehat \kappa$ in \eqref{alpha3:kappa}
 yields one or two values for $\alpha_3\geq 0$.
Next, we need to check the solvability of $\alpha_3=p_{12}(\nu-p_{12})$:
To obtain $p_{12}\in\R$, we must have
 \begin{equation} \label{C:alpha3} \alpha_3\leq \tfrac{\nu^2}4 \,. \end{equation}
Since $\alpha_3$ with the negative sign gives the weaker constraint, 
 we shall use only $\alpha_{3,-}$ in the sequel.
Now, we have to distinguish between three cases:
\begin{enumerate}[label=\textbf{(A\arabic*)}]
\item \label{case:a} $\sqrt{\gamma_1}\leq \tfrac\nu 2$ and $3\gamma_1 +\gamma_2 \leq \nu^2$:
  The unique $\alpha_{3,-}(\widehat \kappa) = \tfrac{\beta_{2} + \widehat\kappa\ (\nu-\widehat\kappa)}2 = \tfrac{3\gamma_1 +\gamma_2}4$
   satisfies condition~\eqref{C:alpha3}.
  Hence, \eqref{eq:beta2}--\eqref{eq:beta1} yield the two endpoints for $(p_{12},p_{22})$ given in~\ref{cone-tip}. 
\item \label{case:b} $\sqrt{\gamma_1}\leq \tfrac\nu 2$ and $3\gamma_1 +\gamma_2 > \nu^2$:
  Here $\alpha_{3,-}(\widehat \kappa)$ violates condition~\eqref{C:alpha3}.
  Hence, $\kappa_{max}$ has to be chosen smaller than $\widehat \kappa$.
  Since $\alpha_{3,-}(\kappa)$ is monotonically increasing,
   the obvious choice $\alpha_3:=\tfrac{\nu^2}4$ also yields the maximal value of $\kappa$:
  Equations \eqref{eq:beta2}--\eqref{eq:beta1} give $\alpha_1 = \beta_2 -\tfrac{\nu^2}4$ 
   and hence $\kappa\ (\nu-\kappa) = \tfrac{\nu^2}4 - \tfrac{4\beta_1}{4\beta_2-\nu^2}$
   with the solution $\kappa_{max}\in [0,\tfrac\nu 2]$ in case~\ref{outer-region}.
\item \label{case:c} $\sqrt{\gamma_1}> \tfrac\nu 2$:
  Using $\widehat \kappa =\tfrac\nu 2$ yields from \eqref{alpha3:kappa} 
  \[ \alpha_{3,-} = \tfrac{\beta_{2}}2 + \tfrac{\nu^2}8 - \tfrac12 \sqrt{(\beta_2 - \tfrac{\nu^2}4)^2 - 4\beta_1}>0. \]
  But one easily checks that it violates again condition~\eqref{C:alpha3}.
  As in case \ref{case:b},
   one chooses $\alpha_3:=\tfrac{\nu^2}4$ and the expressions for $(p_{12},p_{22},\kappa_{max})$ in case~\ref{outer-region} follow.
\end{enumerate}
Finally, Conditions~\allConditions are easily verified for each subcase.

\underline{S\,t\,e\,p\; 2\,:}
The whole interval of solutions in~\ref{cone-tip} is obtained due to the convexity of~$\mathcal{P}$. 
\end{proof}

% version 20.3.2015, submitted to Rev Parma
% version ca. 17.4.2015
% version 24.4.2015, submitted to Rev Parma
% 29.05.2015 Dominik
% 29.05.2015 Anton
% 31.05.2015 14:30 Dominik 
% 31.05.2015 Anton, submitted to Rev Parma
% 1.06.2015 Anton
% 3.06.2015 Dominik
% 4.06.2015 Anton
% 5.06.2015 Dominik
% 5.06.2015 Franz (Verallgemeinerung von Lemma \label{uniB}) + (Beweis von Proposition 4.1 bzw. Abschaetzung \label{decay_etl})
% 5.06.2015 Dominik 22:25
% 5.06.2015 Anton, final submission to Rev Parma
% 22.06.2015 Dominik

\section{Fokker-Planck equations with non-local perturbations}
\label{S4}

\subsection{Introduction}

In this chapter we investigate properties of the following class of perturbed Fokker-Planck equations:
\begin{subequations}\label{orig_equat}
\begin{align}
f_t &= \nabla\cdot(\D\nabla f+\CC\x f)+\Theta f\equiv Lf+\Theta f,\label{orig_equat_a}\\
f(t=0,\x) & =\varphi(\x).\label{orig_equat_b}
\end{align}
\end{subequations}
Thereby $f=f(t,\x)$, and $t\ge 0$ and $\x\in\R^n$, with $n\in\N$. The matrices $\D^{-1}\mathbf C,\D\in\R^{n\times n}$ are symmetric and positive definite, hence $L$ is a symmetric Fokker-Planck operator in $L^2(\R^n; \exp(\frac 12 \mathbf x^T\D^{-1}\mathbf C\mathbf x))$, in fact it is a special case of \eqref{symmFP}. The perturbation is given by a convolution $\Theta f=\vartheta * f$ with respect to $\x$. The convolution kernel $\vartheta$ is assumed to be $t$-independent, and massless, i.e.~$\int_{\R^n}\vartheta(\x)\md\x=0$. 
To keep the solution $f$ real valued we shall consider here only real valued kernels $\vartheta$, but the analysis would be equally valid for complex $\vartheta$'s. 
Further, technical assumptions are specified in the beginning of Section \ref{sec1.4}.

The aim of this chapter is to make a spectral analysis of the perturbed Fokker-Planck operator in an appropriate weighted $L^2$-space, and to show the existence of a unique (up to normalization) stationary solution. Furthermore, the exponential decay of any solution of \eqref{orig_equat} to the stationary solution is proven.

The following analysis is structured as follows. After notational preliminaries in Section \ref{prelimin} we investigate in Section \ref{sec2} the unperturbed Fokker-Planck operator in several functional spaces. First, we recall some of its properties in the $L^2$-space weighted with the reciprocal of the zero eigenfunction (this weight grows super-exponentially), in which the Fokker-Planck operator is self-adjoint. Then, a spectral analysis in a larger, exponentially weighted space is given for this operator. Finally, in Section \ref{sec1.4} we consider the influence of the perturbation $\Theta$ on the spectral properties of the unperturbed Fokker-Planck operator in the exponentially weighted space.

\medskip

Equation \eqref{orig_equat} is a toy model for the Wigner-Fokker-Planck equation, see \cite{Arnold2010}. Other examples for equations of this form can be found in \cite{gw} and \cite{lk}. The following analysis of \eqref{orig_equat} is a generalization of the results published in \cite{StAr14}, where only the case $\CC=\D=\I$ was considered. In this chapter we use a similar approach for proving the desired results. However, several proofs and technicalities differ from \cite{StAr14}.

\subsection{Preliminaries}\label{prelimin}

We use the convention $\N=\{1,2,\ldots\}$, and we write $\N_0:={\N}\cup\{0\}$. Given a complex number $z\in\C$ the complex conjugate is denoted by $\overline z$. For $n\in{\N}$ the elements of $\C^n$ are denoted by bold lowercase letters. Given some vector $\mathbf z\in\C^n$, the $i$-th component is denoted by $z_i$, and we write $\mathbf z=[z_1,\ldots, z_n]^T$ as a column vector. For a multiindex $\kk\in\N_0^n$ we use the notation $\mathbf z^\kk:=z_1^{k_1}\cdots z_n^{k_n}$. Given a real number $s>0$ we define 
\[s^{\mathbf z}:=[s^{z_1}, \ldots, s^{z_n}]^T.\] For $i\in\{1,\ldots, n\}$ the $i$-th unit vector in  $\C^n$ is denoted by $\mathbf e_i$.
For every $1\le p\le\infty$ we define the corresponding $p$-norm on $\C^n$ by
\begin{align*}
|\mathbf z|_p&:=\Big(\sum_{i=1}^n |z_i|^p\Big)^{\frac 1p},\quad 1\le p<\infty,\\
|\mathbf z|_\infty&:=\max_{1\le i\le n}|z_i|.
\end{align*}
With respect to the norm $|\cdot|_p$ the open ball in $\C^n$ with radius $r>0$ and center $\mathbf a\in\C^n$ is defined by
\[B^p_r(\mathbf a)=\{\mathbf z\in\C^n:|\mathbf z-\mathbf a|_p<r\}.\]
Its complement in $\C^n$ is denoted by $B^p_r(\mathbf a)^c:=\C^n\setminus B_r^p(\mathbf a)$. Whenever we work in $\R^n$ instead of $\C^n$ we use the same notation. Matrices are denoted by bold capital letters. For a matrix $\mathbf M\in\C^{n\times n}$ and a real number $s>0$ we define $s^{\mathbf M}:=\exp(\mathbf M\ln s)$, using the matrix exponential.

\medskip

On a domain $\Omega\subseteq\R^n$ we call a real-valued function $w\in L^\infty_{\mathrm{loc}}(\Omega)$ a {\em weight function} if $\frac 1w\in L^\infty_{\mathrm{loc}}(\Omega)$. The corresponding weighted $L^2$-space  $L^2(\Omega;w)$ is the set of all measurable functions $f\colon \Omega\to \C$ such that the norm
\[\|f\|_{\Omega;w}:=\Big(\int_\Omega |f(\x)|^2w(\x)\md \x\Big)^{\frac 12}\]
is finite, and the corresponding inner product is denoted by $\la\cdot,\cdot\ra_{\Omega;w}$.

Also, we introduce weighted Sobolev spaces. For two weight functions $w_0$ and $w_1$ the space $H^1(\Omega;w_0,w_1)$ consists of all functions $f\in L^2(\Omega;w_0)$ whose distributional first order derivatives satisfy ${\partial f}/{\partial x_j}\in L^2(\Omega;w_1)$ for all $1\le j\le n$. We equip the space $H^1(\Omega;w_0,w_1)$ with the norm
\[
      \|f\|_{\Omega;w_0,w_1}:=\big(\|f\|_{\Omega,w_0}^2+\|\nabla f\|_{\Omega,w_1}^2 \big)^{\frac12},
\]
which makes it a Hilbert space, see Theorem~1.11 in \cite{Kufner1984}. If $\Omega=\R^n$ we shall omit the symbol $\Omega$ in these notations. We call two sets of weight functions equivalent if the corresponding weighted spaces are the same. In the case where the weight functions are equivalent to the constant function, we omit the weight function in the notation, e.g.~$L^2(\Omega;1)\equiv L^2(\Omega)$.

\medskip

For functions $f\in L^1(\R^n)$ we define the Fourier transform of $f$ as
\[\F[f](\bfxi)\equiv \hat f(\bfxi):=\int_{\R^n}f(\x)\e^{-\ii\x\cdot\bfxi}\md\x.\]
We use the same notation for the natural extension of the Fourier transform to tempered distributions $f\in\mathscr S'(\R^n)$.
With this scaling we may identify $\hat f(\mathbf 0)$ with the {\em mass} (or \emph{mean}) of $f$. For a tempered distribution $f\in\mathscr S'(\R^n)$ and a multiindex $\kk\in\N_0^n$ we define \[\nabla^{\kk} f(\x):=\frac{\partial^{|{\kk}|_1}f}{\partial x_1^{k_1}\cdots \partial x_n^{k_n}}(\x)\]
as a distributional derivative.

\medskip

Furthermore, we present some definitions and properties concerning linear operators and their spectrum. Let $X,\mathcal X$ be Hilbert spaces. If $X$ is continuously and densely embedded in $\mathcal X$ we write $X\hookrightarrow \mathcal X$, and $X\hookrightarrow \hookrightarrow \mathcal X$ indicates that the embedding is compact.
Given a subset $Y\subset X$, the closure of $Y$ in $X$ is denoted by either $\overline{Y}$ or $\cl_X Y$. $\mathscr C(X)$ denotes the set of all closed operators $A$ in $X$ with dense domain $D(A)$. The set of all bounded operators $A\colon X\to \mathcal X$ is $\mathscr B(X,\mathcal X)$; if $X=\mathcal X$ we just write $\mathscr B(X)$. Thereby $\|\cdot\|_{\mathscr B(X)}$ denotes the operator norm. For an operator $A\in\mathscr C(X)$ its range is $\ran A$ and its null space is $\ker A$. Note that there always holds $\ker A\subset D(A)$. A closed, linear subspace $Y\subset X$ is said to be {\em invariant} under $A\in\mathscr C(X)$ (or {\em $A$-invariant}) iff $D(A)\cap Y$ is dense in $Y$ and $\ran A|_{Y}\subset Y$, see e.g.~\cite{Albrecht2003}. For any $\zeta\in\C$ lying in the resolvent set $\rho(A)$, we denote the resolvent by $R_A(\zeta):=(\zeta-A)^{-1}$. The complement of $\rho(A)$ is the spectrum $\sigma(A)$, and $\sigma_p(A)$ is the point spectrum. For an 
isolated subset $\sigma'\subset\sigma(A)$ the corresponding {\em spectral projection} $\Rho_{A, \sigma'}$ is defined via the line integral 
\begin{equation}\label{def:spec_proj}
 \Rho_{A, \sigma'}:=\frac 1{2\pi\ii}\oint_\Gamma R_A(\zeta)\md\zeta,
\end{equation}
where $\Gamma$ is a closed Jordan curve with counter-clockwise orientation, strictly separating $\sigma'$ from $\sigma(A)\setminus\sigma'$, with $\sigma'$ in the inside of $\Gamma$ and $\sigma(A)\setminus\sigma'$ on the outside. The following results can be found in \cite[Section III.6.4]{kato} and \cite[Section V.9]{taylay}: The spectral projection is a bounded projection operator, decomposing $X$ into two $A$-invariant subspaces, namely $\ran \Rho_{A,\sigma'}$ and $\ker \Rho_{A,\sigma'}$. This property is referred to as the {\em reduction of $A$ by $\Rho_{A, \sigma'}$.} A remarkable property of this decomposition is the fact that $\sigma(A|_{\ran \Rho_{A,\sigma'}})= \sigma'$ and $\sigma(A|_{\ker \Rho_{A, \sigma'}})=\sigma(A)\backslash \sigma'$. Most of the time we will be concerned with the situation where $\sigma'=\{\lambda\}$ is an isolated point of the spectrum.

\medskip
A final remark concerns constants occurring in estimates: Throughout this chapter, $C$ denotes some positive constant, not necessarily always the same. Dependence on certain parameters will be indicated in brackets, e.g.~$C(t)$ for dependence on $t$.

\subsection{Analysis of the Fokker-Planck operator}\label{sec2}

In this section we investigate the (unperturbed) Fokker-Planck equation
\begin{equation}\label{unpert_fp}
f_t = \nabla\cdot(\D\nabla f +  \maD\x f).
\end{equation}
Indeed we can find coordinates that simplify this equation. To this end we proceed similarly to the ``normalization'' of the Fokker-Planck operator after Theorem \ref{Theorem21}. Since $\D$ is symmetric and positive definite we may introduce the coordinate transformation $\mathbf y=\sqrt\D^{-1} \x$. With $g(\mathbf y):=f(\x)$ equation \eqref{unpert_fp} transforms to 
\begin{equation}\label{gggg}g_t= \nabla_{\mathbf y}\cdot(\nabla_{\mathbf y} g + \tilde\maD\mathbf y   g),\end{equation}
with $\tilde \maD=\sqrt \D^{-1}\CC\sqrt \D$. Since $\tilde \CC$ is symmetric and positive definite, we may express the variable $\mathbf y$ in terms of an eigenfunction basis of $\tilde\CC$. Applying this change of coordinates to \eqref{gggg} yields an equation of the same form, but now the matrix $\tilde\CC$ is diagonal (compare to the situation in \eqref{normalizedFP}).

Therefore, without loss of generality we shall always assume that $\D=\I$, and $\mathbf C$ is diagonal in the following, i.e.~$\maD=\diag(c_1,\ldots, c_n)$ with the entries $0<c_1\le c_2\le \cdots\le c_n$. We introduce $\dd:=[c_1,\ldots, c_n]^T$. The unperturbed Fokker-Planck operator $L$ is then
\[L=\Delta +\x^T\CC\nabla +\tr\CC.\]
Note that the perturbation $\Theta$ in \eqref{orig_equat} still is a convolution in the new coordinates.

\medskip

One can check that \[
\mu:=\textstyle \exp(-\frac 12\x^T \maD\x)\]%\marginpar{Faktor!!} 
is a steady state of \eqref{unpert_fp}, i.e.~a zero eigenfunction of $L$. The natural (self-adjoint) setting for $L$ is the space $H:=L^2(1/\mu)$, with the inner product denoted by $\la\cdot,\cdot\ra_H$. There, $L$ is properly defined as the closure of $L|_{C_0^\infty(\R^n)}$. This procedure also yields the domain $D(L)$. The behavior of $L$ in $H$ is well studied (cf.~\cite{MPP02,bakry,helffernier,RiFP89}), we list its main properties in the following theorem. For the case $\maD=\mathbf I$ an analogous result has been published in \cite{StAr14}. A complete proof of the following theorem can be found in \cite{thesis}.

\begin{theorem}\label{LinH}
The Fokker-Planck operator $L$ in $H$ has the following properties:
\begin{enumerate}
\renewcommand{\theenumi}{\roman{enumi}}
\renewcommand{\labelenumi}{(\theenumi)}
\item The operator $L=\cl_HL|_{C_0^\infty}$ on the domain $D(L)$ is self-adjoint and has a compact resolvent.
\item The spectrum consists entirely of isolated eigenvalues and it is given by
\[\sigma(L)=\{-\dd\cdot\kk:\kk\in\N_0^n\}.\]
\item \label{LinH_iii} The zero eigenspace is spanned by $\mu_{\mathbf 0}(\x):=\det(\CC/(2\pi))^{1/2}\exp(-\frac 12\x^T\maD\x)$, and for every $\kk\in\N_0^n$ the function 
$\mu_\kk(\x):=\nabla^\kk\mu_{\mathbf 0}(\x)$
is an eigenfunction to the eigenvalue $-\dd\cdot\kk$.
\item For every $\zeta\in\sigma(L)$ we have
$\ker(\zeta-L)=\spn\{\mu_\kk:\zeta=-\dd\cdot\kk\}$.
\item The family of eigenfunctions $\{\mu_\kk:\kk\in\N_0^n\}$ is an orthogonal basis of $H$.
\item\label{bvi} $L$ generates a $C_0$-semigroup of contractions $(\e^{tL})_{t\ge 0}$ in $H$, and \[\|e^{tL}|_{H_k}\|_{\mathscr B(H)}=\e^{-kc_1t}, k\in\N_0,\]
where $c_1$ is the smallest entry of $\dd$, and $H_k:=\spn\{\mu_\kk:|\kk|_1\le k-1\}^\perp$.
\end{enumerate}
\end{theorem}

The following result is useful in the subsequent analysis:

\begin{lemma}
For every $\kk\in\N_0^n$ the eigenfunction $\mu_\kk$ is of the form
\begin{equation}\label{2.3}
\mu_\kk(\x)=\mu_{\mathbf 0}(\x)\prod_{j=1}^n p_j^{k_j}(x_j),
\end{equation}
where $p_j^{k_j}(x_j)$ is a polynomial of order $k_j$.
\end{lemma}

\begin{proof}
We prove this by induction. For $\kk=\mathbf 0$ the statement clearly holds true. Let it now hold true for some $\kk\in\N_0^n$, and we deduce the validity for $\kk+{\mathbf e}_\ell$ for any $\ell\in\{1,\ldots,n\}$. According to the property $\mu_\kk=\nabla^\kk\mu_{\mathbf0}$ and the induction hypothesis we have
\begin{align*}
\mu_{\kk+{\mathbf e}_\ell}(\x)&=\partial_\ell\Big(\mu_{\mathbf 0}(\x)\prod_{j=1}^n p_j^{k_j}(x_j)\Big)\\
&=\Big(\mu_{\mathbf 0}(\x)\prod_{j\neq \ell} p_j^{k_j}(x_j)\Big)\big(-c_\ell x_\ell  p_\ell^{k_\ell}(x_\ell)+p_\ell^{k_\ell}(x_\ell)'\big).
\end{align*} 
We define the new polynomial  $p_\ell^{k_\ell+1}(x_\ell):=- c_\ell x_\ell p_\ell^{k_\ell}(x_\ell)+(p_\ell^{k_\ell}(x_\ell))'$ and it is obviously of order $k_\ell+1$, since $c_\ell>0$. This proves \eqref{2.3}.
\end{proof}

For the subspaces $H_k$, $k\in{\N_0}$, which were introduced in Theorem~\ref{LinH} \eqref{bvi} we find the following characterization:

\begin{lemma}\label{lem:e_k}
Let $k\in{\N_0}$. There holds $f\in H_k$ iff
\begin{equation}\label{E_k:exp}
\int_{\R^n} f(\x)\x^\kk\md\x=0,\quad\forall |\kk|_1\le k-1.
\end{equation}
\end{lemma}

\begin{proof}
For this we will rely on the representation \eqref{2.3} for the $\mu_\kk$. The result is then shown by induction. Clearly, we have $H_0=H$ and for $k=1$ we obtain
\[H_1=\mu_{\mathbf 0}^\perp =\Big\{f\in H:\int_{\R^n} f(\x)\md \x=0\Big\}.\]
Let us assume now that \eqref{E_k:exp} holds for some $k\in{\N_0}$. According to \eqref{2.3} we have 
\[H_{k+1}=  \Big\{f\in H_k: \int_{\R^n} f(\x)\prod_{j=1}^n p_j^{k_j}(x_j)\md \x=0,\,\forall|\kk|_1=k\Big\}.\]
For $f\in H_k$ and  $|\kk|_1=k$ we get due to the induction hypothesis
\[0=\int_{\R^n} f(\x)\prod_{j=1}^n p_j^{k_j}(x_j)\md \x=a_\kk\int_{\R^n} f(\x)\x^\kk\md \x,\]
where $a_\kk\neq 0$ is the leading coefficient of the polynomial in the integral. All other parts of the first integral vanish due to the induction assumption \eqref{E_k:exp}. Since this holds for all $|\kk|_1=k$ this proves the desired condition for $f\in H_{k+1}$.
\end{proof}

For every $\kk\in\N_0^n$ we define the projection operator $\Pi_{L,\kk}$ corresponding to $\mu_\kk$ by the orthogonal projection
\[\Pi_{L,\kk}:=\la \cdot,\mu_\kk\ra_H\frac{\mu_\kk}{\|\mu_\kk\|_H^2}.\]
With this, the spectral projection corresponding to an eigenvalue $\zeta=-\dd\cdot\kk$ is given by the orthogonal sum
\[\Pi_{L,\zeta}:=\sum_{\kk \in\N_0^n\atop -\dd\cdot\kk=\zeta}\Pi_{L,\kk}.\]

So far we have discussed the operator $L$ in $H$. However, for investigating the perturbed Fokker-Planck operator $L+\Theta$ the space $H$ is not convenient. This can be illustrated in the one-dimensional case with $\Theta f:=f(x+\alpha)-f(x-\alpha)$, for any $\alpha>0$. There one can explicitly show that the zero eigenfunction of $L+\Theta$ does not lie in $H$, for more details see \cite{thesis}. Thus we are forced to investigate $L+\Theta$ in a weighted $L^2$-space with a weight which grows more slowly than $1/\mu_{\mathbf 0}$ as $|\x|_1 \to\infty$. It turns out that 
\begin{equation}\label{weight}
\omega(\x):=\sum_{i=1}^n\cosh\beta x_i
\end{equation}
is a convenient weight function. Thereby $\beta>0$ is an arbitrary constant which is not yet specified. Note that this differs slightly from the weight function chosen in \cite{StAr14}. However, this choice is more practical for the subsequent analysis. In the following we analyze $L+\Theta$ in the weighted space $\HH:=L^2(\omega)$. The natural norm and the inner product in $\HH$ are denoted by $\|\cdot\|_\omega$ and $\la\cdot,\cdot\ra_\omega$, respectively.

The space $\HH$ possesses a useful characterization via the Fourier transform.

\begin{proposition}\label{prop3.4} There holds $f\in\HH$ iff its Fourier transform $\hat f$ possesses an analytic continuation (still denoted by $\hat f$) to the open set $\Omega_{\beta/2}:=\{\mathbf z\in\C^n: |\Imag \mathbf z|_1<\beta/2\}$, with the property
\begin{equation}\label{unif_l2}
\sup_{\substack{ \mathbf b\in\R^n\\ |\mathbf b|_1<\beta/2}} \|\hat f(\cdot + \ii \mathbf b)\|_{L^2(\R^n)}<\infty.
\end{equation}
In this case we have:
\begin{enumerate}
\renewcommand{\theenumi}{\roman{enumi}}
\renewcommand{\labelenumi}{(\theenumi)}
\item\label{obeta_ii} For every $\mathbf b\in\R^n$ with $|\mathbf b|_1<\beta/2$ there holds
\begin{equation}\label{idkn}
\hat f(\bfxi+\ii\mathbf b)=\F[f(\x)\exp(\mathbf b\cdot\x)](\bfxi),\quad\bfxi\in\R^n.
\end{equation}

\item\label{obeta_iii} For every $\mathbf b\in\R^n$ with $|\mathbf b|_1=\beta/2$ we define $\hat f(\bfxi+\ii\mathbf b):=\F[f(\x)\exp(\mathbf b\cdot\x)](\bfxi)$, which lies in $L^2(\R^n)$. With this there holds $\mathbf b\mapsto \hat f(\cdot+\ii\mathbf b)\in C(\overline{B^1_{\beta/2}(\mathbf 0)};L^2(\R^n))$.
\end{enumerate}
\end{proposition}

See Theorem~IX.13 in \cite{resi2} for a very similar result. For a detailed proof see \cite{thesis}. Often we shall use the following norm, which is equivalent to $\|\cdot\|_\omega$ due to the Plancherel theorem: 
\begin{equation}\label{f_norm:d}
 \nn f\nn^2_\omega:=\sum_{\ell=1}^n \Big\|\hat f\Big(\cdot+\ii\frac\beta2{\mathbf e}_\ell\Big)\Big\|^2_{L^2(\R^n)}+\Big\|\hat f\Big(\cdot-\ii\frac\beta2{\mathbf e}_\ell\Big)\Big\|^2_{L^2(\R^n)}.
\end{equation}

A useful property of $\HH$ is the validity of the following Poincar\'e inequality:

\begin{lemma}\label{pr_pc:d}
There exists a constant $C_p>0$ such that for every $f\in H^1(\omega,\omega)$ there holds
\begin{equation}\label{poincare:d}
 \|f\|_\omega\le C_p\|\nabla f\|_\omega.
\end{equation}
\end{lemma}

\begin{proof}
For this we use the norm $\nn\cdot\nn_\omega$. We compute
\begin{align*}
\nn\nabla f\nn_\omega^2 &=\sum_{j=1}^n\sum_{\ell=1}^n\Big( \textstyle \big\|\big(\xi_j+\ii\frac\beta 2 \delta_{j\ell}\big)\hat f\big(\bfxi+\ii\frac\beta2{\mathbf e}_\ell\big)\big\|^2_{L^2(\R^n)}\\
&\qquad\qquad+\big\|\big(\textstyle \xi_j-\ii\frac\beta 2 \delta_{j\ell}\big)\hat f\big(\bfxi-\ii\frac\beta2{\mathbf e}_\ell\big)\big\|^2_{L^2(\R^n)}\Big)\\
&\ge \sum_{\ell=1}^n\Big( \textstyle \big\|\big(\xi_\ell+\ii\frac\beta 2\big)\hat f\big(\bfxi+\ii\frac\beta2{\mathbf e}_\ell\big)\big\|^2_{L^2(\R^n)}\\
&\qquad\qquad+\big\|\big(\textstyle \xi_\ell-\ii\frac\beta 2 \big)\hat f\big(\bfxi-\ii\frac\beta2{\mathbf e}_\ell\big)\big\|^2_{L^2(\R^n)}\Big)\\
&\ge\big( {\textstyle{\frac\beta 2}}\big)^2\sum_{\ell=1}^n\Big( \textstyle \big\|\hat f\big(\bfxi+\ii\frac\beta2{\mathbf e}_\ell\big)\big\|^2_{L^2(\R^n)}+\big\|\hat f\big(\bfxi-\ii\frac\beta2{\mathbf e}_\ell\big)\big\|^2_{L^2(\R^n)}\Big)\\
&= \big( {\textstyle{\frac\beta 2}}\big)^2 \nn f\nn_\omega^2.
\end{align*}
This proves the Poincar\'e inequality with the constant $C_p=\frac 2\beta$.
\end{proof}

Using the above properties of $\HH$ we can investigate $L$ in $\HH$. The following theorem is the main result of this section and describes the (unperturbed) Fokker-Planck operator in $\HH$:

\begin{theorem}\label{trm_sec_1}
Let $\omega(\x)$ be the weight function defined in \eqref{weight} for any $\beta>0$, and $\HH:=L^2(\omega)$ is the corresponding weighted space. Then the Fokker-Planck operator $L|_{C_0^\infty(\R^n)}$ is closable in $\HH$, we write $\LL:=\cl_\HH L|_{C_0^\infty(\R^n)}$. In $\HH$ the operator $\LL$ has the following properties:
\begin{enumerate}\renewcommand{\theenumi}{\roman{enumi}}\renewcommand{\labelenumi}{(\theenumi)}
\item The resolvent of $\LL$ is compact, and $\sigma(\LL)$ consists entirely of isolated eigenvalues. 
\item The spectrum of $\LL$ is given by
\[\sigma(\LL)=\{-\dd\cdot\kk:\kk\in\N_0^n\},\]
where $\dd$ is the column vector containing the diagonal entries of $\maD$.
\item For every $\lambda\in\sigma(\LL)$ the corresponding eigenspace of $\LL$ is given by
\[\spn\{\mu_\kk:\kk\in\N_0^n\wedge-\dd\cdot\kk=\lambda\},\]
where the eigenfunctions $\mu_\kk$ were introduced in Theorem~\ref{LinH}.
\item For every $k\in{\N_0}$ the following is a closed subspace of $\HH$:
\[
\HH_k:=\Big\{f\in\HH: \int_{\R^n} f(\x)\x^\kk\md\x=0,\quad \forall \kk\in\N_0^n\,\,\text{with}\,\, |\kk|_1\le k-1\Big\}.
\]
$\HH_k$ is $\LL$-invariant, and $\sigma(\LL|_{\HH_k})=\{-\dd\cdot\kk:|\kk|_1\ge k\}$. There holds the identity
\[\HH=\HH_k\oplus\spn\{\mu_\kk:|\kk|_1\le k-1\}.\]
\item\label{sec_1_v} $\LL$ generates a $C_0$-semigroup of bounded operators $(\e^{t\LL})_{t\ge 0}$ on $\HH$. For every $k\in{\N_0}$ there exists a constant $C_k$ such that
\begin{equation}\label{CK}
\|\e^{t\LL}|_{\HH_k}\|_{\mathscr B(\HH)}\le C_k\e^{-tkc_1},\quad\forall t\ge 0.
\end{equation}
\end{enumerate}
\end{theorem}

The rest of this section is dedicated to proving Theorem~\ref{trm_sec_1}. The proof is structured into several lemmata and propositions. To this end we begin by showing that the Fokker-Planck operator can be defined as a closed operator in $\HH$ and we characterize its domain. The first preparatory result is the following lemma, which is also essential for showing the compactness of the resolvent of the Fokker-Planck operator in $\HH$.

\begin{lemma}\label{lem_comp_est}
Let $\Real\zeta\ge \frac 12(1+\beta^2+\tr \maD)$, and $f,g\in C_0^\infty(\R^n)$ such that $(\zeta- L)f=g$. Then there exists a constant $C>0$, independent of $f,g$, such that
\begin{equation}\label{reso_est}
\|f\|_\varpi+\|\nabla f\|_\omega \le C\|g\|_\omega.
\end{equation}
Thereby $\varpi(\x):=(1+|\x|_2)\omega(\x)$.
\end{lemma}

\begin{proof}
For $f\equiv 0$, $g\equiv 0$ \eqref{reso_est} holds trivially. For $f\not\equiv 0$ we apply $\la\cdot,f\ra_\omega$ to $(\zeta-L)f=g$, and compute
\begin{align}
\Real \int_{\R^n} g\overline{f}\omega\md \x &= \Real\int_{\R^n}\big(\zeta f-\nabla\cdot(\nabla f+\CC \x f)\big)\overline f\omega\md \x\nonumber\\
&=\Real \zeta \int_{\R^n}|f|^2\omega\md\x +\Real \int_{\R^n}(\nabla f+\CC \x f)\cdot(\omega\nabla \overline f+\overline f\nabla \omega)\md\x\nonumber\\
&= \|\nabla f\|_\omega^2+ \frac 12\int_{\R^n} |f|^2(2\Real\zeta\omega-\Delta\omega-\omega \tr \maD+ \x^T\maD\nabla\omega)\md\x \nonumber\\
 &=\|\nabla f\|_\omega^2+ \frac 12\int_{\R^n} |f|^2\nu\md\x.\label{comp_estimate}
\end{align}
Thereby we temporarily define $\nu(\x):=2\Real\zeta\omega-\Delta\omega-\omega \tr \maD+ \x^T\maD\nabla\omega$. We observe that $\Delta \omega = \beta^2\omega$ and $\x^T \maD\nabla \omega=\beta\sum_{i=1}^n c_ix_i\sinh \beta x_i\ge 0$ for all $\x\in\R^n$. So if $\Real \zeta\ge \frac 12(1+\beta^2+\tr \maD)$, the function $\nu(\x)$ is a weight function with $\nu(\x)\ge \omega(\x)$ on $\R^n$. Next we apply the Cauchy-Schwarz inequality to the left hand side of \eqref{comp_estimate}, which yields
\[\|\nabla f\|_\omega^2+ \frac 12\|f\|^2_\nu\le \|f\|_\omega\|g\|_\omega.\]
We now use the Poincar\'e inequality on the first term and $\nu(\x)\ge \omega(\x)$ on the second term, and divide by $\|f\|_\omega$:
\[\|\nabla f\|_\omega+ \|f\|_\nu\le C\|g\|_\omega.\]
Finally we observe that for any fixed $\Real\zeta\ge \frac 12(1+\beta^2+\tr \maD)$ there is a constant $C>0$ such that $\nu( \x)\ge C\varpi(\x)$ for all $\x\in\R^n$. This concludes the proof.
\end{proof}

Before we properly define the Fokker-Planck operator as a closed operator in $\HH$, we need the lemma below. It determines all formal eigenfunction of the Fokker-Planck operator, i.e.~the eigenfunctions of the {\em distributional} Fokker-Planck operator $\mathfrak L$ in $\HH$. Thereby, we define the distributional Fokker-Planck operator as $\mathfrak L:=\Delta + \x^T\mathbf C\nabla + \tr\mathbf C$ in the sense of tempered distributions. $\mathfrak L$ is then a well-defined linear map from $\HH$ into $\mathscr S'$, defined on the whole space $\HH$. As a consequence of the following lemma it will be straightforward to determine the spectrum of the Fokker-Planck operator in $\HH$.

\begin{lemma}\label{distr_spec_L}
The distributional Fokker-Planck operator $\mathfrak L$ satisfies the eigenvalue equation $\mathfrak Lf=\zeta f$ for some $\zeta\in\C$ and some $f\in\HH\setminus\{0\}$ 
%There exists some $f\in\HH\setminus\{0\}$ such that $\mathfrak Lf=\zeta f$ for some $\zeta\in\C$ 
iff $\zeta\in\{-\dd\cdot\kk:\kk\in\N_0^n\}$. For such values of $\zeta$, there holds $f\in\spn\{\mu_\kk:-\dd\cdot\kk=\zeta\}$.
\end{lemma}

\begin{proof}
Since all the functions $\mu_\kk$ are eigenfunctions of $L$ and lie in $\HH$ it is clear that they are also eigenfunctions of $\mathfrak L$. In order to show that they already span all eigenspaces we consider the Fourier transform of $(\zeta-\mathfrak L)f=0$ for any $\zeta\in\C$, which reads
\begin{equation}\label{f_traf:eval}
(\zeta+|\bfxi|_2^2)\hat f+\bfxi^T\maD\nabla \hat f =0.
\end{equation}
Now we are looking for $f\in\HH$ and $\zeta\in\C$ satisfying this (eigenvalue) equation. 
This means that we are interested in solutions $\hat f$ which are analytic in $\Omega_{\beta/2}$. Expecting $f$ to be generated from $\mu_{\mathbf 0}$ by repeated differentiation (see Theorem~\ref{LinH}~\eqref{LinH_iii}), we make the ansatz $\hat f= p \hat \mu_{\mathbf 0}$,  with $p$ analytic in $\Omega_{\beta/2}$. This is admissible (and not restrictive) since $\hat \mu_{\mathbf 0}$ is nonzero and analytic in $\Omega_{\beta/2}$. We know that $\hat \mu_{\mathbf 0}$ satisfies the zero eigenvalue equation $|\bfxi|_2^2\hat \mu_{\mathbf 0}+\bfxi^T\maD\nabla \hat \mu_{\mathbf 0} =0$, so after inserting $\hat f= p \hat \mu_{\mathbf 0}$ in \eqref{f_traf:eval} we obtain the following equation for $p$:
\begin{equation}\label{equ:p}
\bfxi^T\maD\nabla p=-\zeta p.
\end{equation}
To solve this first order PDE we consider its characteristics: We introduce the (unique) solution $\bfxi(t)$ of the ordinary differential equation $\dot \bfxi=\maD\bfxi$ with $\bfxi(0)=\bfxi_0\in\C^n$. It is verified by application of the chain rule that for any such curve and any differentiable function $p$ we have
\[\frac{\md}{\md t}p(\bfxi(t))=\bfxi(t)^T\maD\nabla p(\bfxi(t)).\]
In particular, any solution of $\eqref{equ:p}$ fulfills the ordinary differential equation 
\[\frac{\md}{\md t} p(\bfxi(t))=-\zeta p(\bfxi(t))\]
along these curves, and it follows $p(\bfxi(t))=p(\bfxi_0)\e^{-\zeta t}$. Using the fact that $\bfxi(t)=\e^{t\maD}\bfxi_0$ and introducing $s=\e^t$ (with $s\in\R^+$) we obtain $\bfxi(t)= s^\maD\bfxi_0$ (see Section \ref{prelimin} concerning the notation), and consequently we obtain
\begin{equation}\label{pole_zeta}
p(s^\maD\bfxi_0 )=p(\bfxi_0)s^{-\zeta}.
\end{equation}
Now $p$ needs to be analytic in $\Omega_{\beta/2}$. So \eqref{pole_zeta} implies that $\Real\zeta\le 0$ is necessary, otherwise $p$ would have a singularity at the origin $\bfxi=\mathbf 0$ (corresponding to $s\searrow 0$), which is a contradiction. By induction we deduce from \eqref{equ:p} that for all $\kk\in\N_0^n$
\[\bfxi^T\maD\nabla\big(\nabla^\kk p\big)=-(\zeta+\dd\cdot\kk)\nabla^\kk p.\]
Since all derivatives $\nabla^\kk p$ need to be analytic in $\Omega_{\beta/2}$ as well, the above argument proves that either $\Real \zeta\le -\dd\cdot\kk$ for all $\kk\in\N_0^n$ (which is impossible since $\maD>0$) or $\nabla^\kk p\equiv 0$ in $\Omega_{\beta/2}$ for some $\kk\in{\N_0}$. So $p$ has to be a polynomial, and we make the ansatz
\[p(\bfxi)=\sum_{\kk\in\N_0^n}p_\kk\bfxi^\kk,\]
where $p_\kk=0$ for almost all $\kk\in\N_0^n$. We now insert this in \eqref{equ:p} and obtain
\[\sum_{\kk\in\N_0^n}(\dd\cdot\kk) p_\kk\bfxi^\kk=-\zeta\sum_{\kk\in\N_0^n}p_\kk\bfxi^\kk.\]
This holds true iff $\zeta=-\dd\cdot\kk$ for all $\kk\in\N_0^n$ for which $p_\kk\neq 0$. This proves the first statement of the lemma. 

{}From the above analysis we conclude
$$
  \hat f(\bfxi)=\Big(\sum_{\substack{\kk\in\N_0^n\\ \dd\cdot\kk=-\zeta}}p_\kk\bfxi^\kk\Big)\hat \mu_{\mathbf 0}(\bfxi)\,.
$$
Now recall from Theorem~\ref{LinH} \eqref{LinH_iii} that $\hat\mu_\kk=\ii^{|\kk|_1}\bfxi^\kk\hat\mu_{\mathbf 0}$ holds for all $\kk\in\N_0^n$. Hence, $f\in\spn\{\mu_\kk:-\dd\cdot\kk=\zeta\}$.
So we conclude that the eigenspaces of $\mathfrak L$ in $\HH$ are precisely spanned by the $\mu_\kk$.
\end{proof}

Now we can properly define the Fokker-Planck operator in the space $\HH$.

\begin{lemma}\label{DL}
The operator $L|_{C_0^\infty}$ is closable in $\HH$, and  $\LL:=\cl_\HH L|_{C_0^\infty}$. The domain is $D(\LL)=\{ f\in\HH: \mathfrak{L}f\in\HH\}$, and for $f\in D(\LL)$ we have $\LL f=\mathfrak{L}f$.
\end{lemma}

The following proof is based on the proof of Lemma~2.6 in \cite{StAr14}.

\begin{proof}
According to \eqref{comp_estimate} we have that $(L-\zeta)|_{C_0^\infty}$ is dissipative in $\HH$ if $\Real\zeta\ge \frac 12(1+\beta^2+\tr \maD)$. This implies (cf.~\cite[Theorem~1.4.5 (c)]{pazy}) that $(L-\zeta)|_{C_0^\infty}$ and consequently also $L|_{C_0^\infty}$ is closable in $\HH$.

Now we define $\LL:=\cl_\HH L|_{C_0^\infty}$. The domain $D(\LL)$ consists of all $f\in\HH$ for which there exists some $g\in \HH$ and a sequence $(f_n)_{n\in{\N_0}}\subset C_0^\infty(\R^n)$ such that 
\begin{equation}\label{cs_lf}
\begin{cases}
\displaystyle \lim_{n\to\infty}\|f_n-f\|_\omega=0,\\
\displaystyle \lim_{n\to\infty}\|Lf_n-g\|_\omega=0.
\end{cases}
\end{equation}

This also implies that $((\zeta-L)f_n)_{n\in{\N_0}}$ is a Cauchy sequence in $\HH$. Thus, according to \eqref{reso_est} $(\nabla f_n)_{n\in{\N_0}}$ is a Cauchy sequence in $\HH$. So altogether, $(f_n)_{n\in{\N_0}}$ is a Cauchy sequence in the Hilbert space $H^1(\omega,\omega)$. But since we already know that $f_n\to f$ in $\HH$, this implies that even $f\in H^1(\omega,\omega)$. Next we temporarily introduce the weight $\omega_2(\x):=\omega(\frac\x 2)$ and the corresponding weighted space $\HH_2:=L^2(\omega_2)$. Due to the previous results $(\x^T \maD\nabla f_n+\tr \maD f_n)_{n\in{\N_0}}$ is a Cauchy sequence in $\HH_2$. According to \eqref{cs_lf},  $(Lf_n)_{n\in{\N_0}}$ is also a Cauchy sequence in $\HH_2$. Altogether, this implies that $(\Delta f_n)_{n\in{\N_0} }$ is a Cauchy sequence in $\HH_2$. Applying the Fourier transform and the norm \eqref{f_norm:d} we have that, for every $\ell\in\{1,\ldots,n\}$, the two sequences
\[\Big(\Big(\bfxi\pm\ii\frac\beta 4{\mathbf e}_\ell\Big)^2\hat f_n\Big(\bfxi\pm\ii\frac\beta 4{\mathbf e}_\ell\Big) \Big)_{n\in{\N_0}} \]
are Cauchy sequences in $L^2(\R^n)$. But we also know that $\hat f_n(\cdot\pm\ii\frac\beta 4{\mathbf e}_\ell)$ converges to $\hat f(\cdot\pm\ii\frac\beta 4{\mathbf e}_\ell)$ in $L^2(\R^n)$. Thus it is clear that $\Delta f\in \HH_2$, and $\Delta f_n\to \Delta f$ in $\HH_2$ and also $Lf_n\to \mathfrak{L}f$ in $\HH_2$. According to \eqref{cs_lf} $\mathfrak{L}f=g$ in $\HH_2$, and since $g\in\HH$, we conclude that $Lf_n\to \mathfrak Lf$ in $\HH$. This proves the inclusion $D(\LL)\subseteq \{ f\in\HH: \mathfrak{L}f\in\HH\}$.

Finally we prove that this inclusion indeed is an equality. First we note that $D(L)\subset D(\LL)$ since $L=\cl_HL|_{C_0^\infty}$ and $H\hookrightarrow \HH$. So we have the inclusion $L\subset \LL$ for the graphs. Let us then take $\zeta>0$ so large that the estimate \eqref{reso_est} holds. As we have mentioned in the beginning of the proof the operator $(\LL-\zeta)|_{C_0^\infty}$ is (uniformly) dissipative in $\HH$, and from Theorem~1.4.5 in \cite{pazy} it follows that the closure,  $\LL-\zeta$, is also (uniformly) dissipative. In particular it is injective and thus invertible. So $(\zeta-\LL)^{-1}$ exists. Now according to Theorem~\ref{LinH} $\zeta-L\colon D(L)\to H$ is a bijection, so $\ran (\zeta-\LL)\supset H$, which is dense in $\HH$. Due to this and the estimate \eqref{reso_est} $(\zeta-\LL)^{-1}$ is a densely defined bounded operator in $\HH$. But by definition  $(\zeta-\LL)^{-1}$  is already closed, so $\ran(\zeta-\LL)=\HH$ and $\zeta\in\rho(\LL)$ 
(and thus $\rho(\mathcal{L})\neq\emptyset$). 

For the proof by contradiction we take now 
this $\zeta\in\rho(\LL)$, and assume there exists some $f^*\in\HH\setminus D(\LL)$ such that $f^*\in\HH$. Hence also $(\zeta-\mathfrak L)f^*\in\HH$. Since $\zeta\in\rho(\LL)$ we have $(\zeta-\LL)^{-1}(\zeta-\mathfrak L)f^*\in D(\LL)$. Since $D(\mathcal L)$ is a linear space we have $f^\sharp:=(\zeta-\LL)^{-1}(\zeta-\mathfrak L)f^*-f^*\in\HH\setminus D(\LL)$ with $(\zeta-\mathfrak L)f^\sharp=0$. But according to Lemma~\ref{distr_spec_L} we know that $\zeta\in\rho(\LL)$ cannot be an eigenvalue of $\mathfrak L$ in $\HH$. So $f^\sharp=0$, contradicting $f^\sharp\in\HH\setminus D(\mathcal L)$. Hence we conclude  $D(\LL)=\{f\in\HH:\mathfrak L f\in\HH\}$.
\end{proof}

\begin{lemma}
For any $\zeta\in\rho(\LL)$ the resolvent $(\zeta-\LL)^{-1}$ is compact in $\HH$.
\end{lemma}

\begin{proof}
We fix $\zeta>0$, and first show the result for this given $\zeta$. Choosing $\zeta$ large enough we can apply Lemma~\ref{lem_comp_est} which proves that $(\zeta-\LL)^{-1}$ is an element of $\mathscr B(\HH,H^1(\varpi,\omega))$. Note that this requires the density of $C_0^\infty(\R^n)$ in $\HH$, which is assured by Lemma~\ref{dens_test_fun_2} in the Appendix.

Now we shall show that $H^1(\varpi,\omega)$ is compactly embedded in $\HH$. By the definition of $\varpi$ (in Lemma~\ref{lem_comp_est}) it is clear that for all $n\in{\N_0}$ there holds
\[\sup_{|\x|_2>n}\frac{\omega(\x)}{\varpi(\x)}=\frac 1{1+n},\]
which tends to zero as $n\to\infty$. Thus we can apply Lemma~\ref{comp_emb} in the appendix, which proves the compact embedding $H^1(\varpi,\omega)\hookrightarrow\hookrightarrow \HH$. Hence, the resolvent $(\zeta-\LL)^{-1}\colon \HH\to\HH$ is compact. Finally we remark that, according to Theorem~III.6.29 in \cite{kato}, the compactness of $(\zeta-\LL)^{-1}$ follows for all other $\zeta\in\rho(\LL)$.
\end{proof}

\begin{corollary}\label{cor_spec}
The spectrum $\sigma(\LL)$ consists entirely of eigenvalues, and $\sigma(\LL)=\{-\dd\cdot\kk:\kk\in\N_0^n\}$. The eigenspace corresponding to the eigenvalue $\zeta\in\sigma(\LL)$ is given by $\spn\{\mu_\kk:\zeta=-\dd\cdot\kk\}$.
\end{corollary}

\begin{proof}
We apply Theorem~III.6.29 in \cite{kato} which states that $\sigma(\LL)$ consists entirely of eigenvalues, and the corresponding eigenspaces are finite-dimensional. According to Lemma~\ref{DL} the eigenfunctions of $\LL$ in $D(\LL)$ are precisely the (formal) eigenfunctions of $\mathfrak L$ in $\HH$. With this, Lemma~\ref{distr_spec_L} concludes the proof.
\end{proof}

We introduce the closed subspaces $\HH_k\subset\HH$ for every $k\in{\N_0}$, which we define as $\HH_k:=\cl_\HH H_k$, where the subspaces $H_k$ were specified in Theorem~\ref{LinH}. The following lemma gives a characterization of the spaces $\HH_k$, compare Lemma~\ref{lem:e_k} for an analogous result in $H$.

\begin{lemma}\label{lemma48}
For every $k\in{\N_0}$ there holds
\begin{equation}\label{HH_kk}
\HH_k=\Big\{f\in\HH: \int_{\R^n} f(\x)\x^\kk\md\x=0,\quad \forall \kk\in\N_0^n\,\,\text{with}\,\, |\kk|_1\le k-1\Big\}.
\end{equation}
\end{lemma}

\begin{proof}
We start from the characterization of the $H_k$ in Lemma~\ref{lem:e_k}. Our plan is to apply Lemma~\ref{lem:funct} in the appendix. For every $\kk\in\N_0^n$ we define the functional
\[\eta_\kk\colon \HH\to\C\colon f\mapsto \int_{\R^n}f(\x)\x^\kk\md\x.\]
We first prove the continuity of the $\eta_\kk$. For $\kk\in\N_0^n$ and $f\in\HH$ we have
\begin{align*}
\Big|\int_{\R^n}f(\x)\x^\kk\md\x\Big| &\le \int_{\R^n}|f(\x)\omega(\x)^{1/2}|\cdot \Big|\frac{\x^\kk}{\omega(\x)^{1/2}}\Big|\md\x\\
&\le \|f\|_\omega\cdot \Big(\int_{\R^n} \frac{\x^{2\kk}}{\omega(\x)}\md\x\Big)^{\frac 12}.
\end{align*}
Since $\omega$ grows exponentially in every direction it is clear that the last integral on the right hand side is finite for every $\kk\in\N_0^n$. Thus the $\eta_\kk$ are bounded linear functionals in $\HH$. Next we shall verify that the family $\{\eta_\kk:\kk\in\N_0^n\}$ is linearly independent. If the family would be linearly dependent, there would exist a polynomial $p(\x)\not\equiv 0$ such that
\[\int_{\R^n}f(\x)p(\x)\md\x=0,\quad\forall f\in\HH.\]
But this implies $p\equiv 0$, since $C_0^\infty(\R^n)\subset\HH$.

Now we have verified the assumptions of Lemma~\ref{lem:funct}. Since 
\[H_k=\bigcap_{|\kk|_1\le k-1}\ker \eta_{\kk}|_H,\]
we conclude that 
\[\HH_k:=\cl_\HH H_k= \bigcap_{|\kk|_1\le k-1}\ker \eta_{\kk}.\]
The intersection on the right is exactly the set \eqref{HH_kk}.
\end{proof}

\begin{corollary}
For $k\in{\N_0}$ there holds the identity
\begin{equation}\label{HH_kk_fourier}
\HH_k=\big\{f\in\HH: \nabla^\kk\hat f(\mathbf 0)=0,\quad \forall |\kk|_1\le k-1\big\}.
\end{equation}
\end{corollary}

\begin{proof}
This follows immediately from the fact that for $f\in\HH$ and $\kk\in\N_0^n$
\[\int_{\R^n}\x^\kk f(\x)\md\x=\F[\x^\kk f(\x)](\mathbf 0)= \ii^{|\kk|_1}\nabla^\kk \hat f(\mathbf 0).\]
We use this in \eqref{HH_kk} and the result follows.
\end{proof}

At every $\lambda\in\sigma(\LL)$ the resolvent map $\zeta\mapsto R_\LL(\zeta)$ has an isolated singularity. We denote the corresponding spectral projection of $\LL$ by $\Pi_{\LL,\lambda}$, which satisfies \eqref{def:spec_proj}. In particular there holds $\Pi_{\LL,\lambda}=\cl_\HH \Pi_{L,\lambda}$, as we will see in the following.

\begin{proposition}\label{prop_pi}
For every $k\in{\N_0}$ we have the following facts:
\begin{enumerate}
\renewcommand{\theenumi}{\roman{enumi}}
\renewcommand{\labelenumi}{(\theenumi)}
\item The space $\HH$ can be written as the following direct sum: $\HH=\HH_k\oplus\spn\{\mu_\kk:|\kk|_1\le k-1\}$.
\item Both spaces $\HH_k$ and $\spn\{\mu_\kk:|\kk|_1\le k-1\}$ are closed in $\HH$ and  $\LL$-invariant. In particular $\sigma(\LL|_{\HH_k})=\{-\dd\cdot\kk:|\kk|_1\ge k\}$.
\end{enumerate}
\end{proposition}

\begin{proof}
\underline{Step 1  (decomposition of $H_k$):} 
%In $H$ we have due to the orthogonality of the eigenfunctions 
In $H$ there holds for any fixed $k\in\N$
\begin{equation}\label{3105_2}
H_k^\perp=\spn\big\{\mu_\kk:|\kk|_1\le k-1\big\},
\end{equation}
and for every $\lambda\in\sigma(L)$ we have for the corresponding spectral projection
\begin{subequations}\label{3105_1}
\begin{align}
\ran\Pi_{L,\lambda}&=\spn\big\{\mu_\kk:-\mathbf{c}\cdot\kk=\lambda\big\}, \label{ran_pi_new}\\
\ker\Pi_{L,\lambda}&=\spn\big\{\mu_\kk:-\mathbf{c}\cdot\kk\neq\lambda\big\}.\label{ker_pi_new}
\end{align}
\end{subequations}
For a given $k\in\N$ we define the set 
\[\sigma_k:=\{-\mathbf{c}\cdot\kk:|\kk|_1\le k-1\}\subset\R_0^-\,,\]
which is the set of all eigenvalues which ``contribute'' to $H_k^\perp$ (note that there may be $\kk\in\N_0^n$ such that $-\dd\cdot\kk\in\sigma_k$ but $|\kk|_1\ge k$). 
{}From \eqref{ran_pi_new} we conclude that 
\[
    \bigcup_{\lambda\in\sigma_k}\ran\Pi_{L,\lambda}\supset H_k^\perp\,.
\]
Taking the orthogonal complement of this relation yields:
%According to \eqref{ran_pi_new} we have $\lambda\in \sigma_k$ iff $\ran\Pi_{L,\lambda}\cap H_k^\perp\neq\{\mathbf 0\}$. This shows (with \eqref{ker_pi_new}) that
\begin{equation}\label{union-range}
\bigcap_{\lambda\in\sigma_k}\ker\Pi_{L,\lambda}\subset H_k.
\end{equation}
Next we investigate which eigenfunctions $\mu_\kk$ need to be added to the left hand side of \eqref{union-range} such that the corresponding span equals $H_k$. First we observe that, according to \eqref{3105_1}, there holds $\mu_\kk\in\big(\bigcap_{\lambda\in\sigma_k}\ker\Pi_{L,\lambda}\big)^\perp$ iff $\mu_\kk\in\ran\Pi_{L,\lambda}$ for some $\lambda\in\sigma_k$. This is also equivalent to the condition $-\mathbf{c}\cdot\kk\in\sigma_k$. To complement the left hand side of \eqref{union-range}, we also require $\mu_\kk\in H_k$, which gives the constraint $|\kk|_1\ge k$, see \eqref{3105_2}. Hence, we conclude that 
\begin{equation}\label{Hk-decomp}
  H_k=\Big(\bigcap_{\lambda\in\sigma_k}\ker\Pi_{L,\lambda}\Big)\oplus_\perp
  \spn\{\mu_\kk: -\mathbf{c}\cdot\kk\in\sigma_k \wedge |\kk|_1\ge k \}.
\end{equation}\\

\underline{Step 2  (decomposition of $\HH$):} 
For $\zeta\in\rho(\LL)$ we have $R_L(\zeta)\subset R_\LL(\zeta)$ (in the sense of graphs), and as a consequence the spectral projection for $\lambda\in\sigma(\LL)$ satisfies $\Pi_{L,\lambda}\subset\Pi_{\LL,\lambda}$, see \eqref{def:spec_proj}. Furthermore, both $\Pi_{L,\lambda}$ and $\Pi_{\LL,\lambda}$ are bounded projections in $H$ and $\HH$, respectively. Due to Lemma~\ref{lem:proj} in the appendix there holds
\begin{equation}\label{closure_proj}
\ker \Pi_{\LL,\lambda}=\cl_\HH\ker\Pi_{L,\lambda} \quad\text{and}\quad \ran \Pi_{\LL,\lambda}=\cl_\HH\ran \Pi_{L,\lambda}.
\end{equation}
Since the projections are bounded we have $\HH=\ker \Pi_{\LL,\lambda}\oplus\ran \Pi_{\LL,\lambda}$, and both components of the direct sum are closed subspaces of $\HH$, see Section III.3.4 in \cite{kato}.\\

\underline{Step 3  (decomposition of $\HH_k$):} 
Due to the arguments of Step 2 we obtain, by applying the closure in $\HH$ to \eqref{Hk-decomp}:
\begin{equation}\label{ker_Pi}
  \HH_k=\Big(\bigcap_{\lambda\in\sigma_k}\ker\Pi_{\LL,\lambda}\Big)\oplus
  \spn\{\mu_\kk: -\dd\cdot\kk\in\sigma_k\wedge |\kk|_1\ge k\}.
\end{equation}
Notice that $\sigma_k$ is finite. The sum is still a direct sum, since every $\mu_\kk$ in the ``span-term'' of the right hand side lies in the range of some $\Pi_{\LL,\lambda}$ with $\lambda\in\sigma_k$. %on%%%%%since BLA lies in \ran??
%hich can be seen by the following argument: Due to \eqref{closure_proj} there holds $\ran \Pi_{\LL,\lambda}=\spn\{\mu_\kk:-\dd\cdot\kk=\lambda\}$, and since $\Pi_{\LL,\lambda}$ is bounded we conclude that $\ker\Pi_{\LL,\lambda}\cap \spn\{\mu_\kk:-\dd\cdot\kk=\lambda\}=\{0\}$. 
Altogether this implies that $\HH_k$ is a closed subspace of $\HH$ such that 
\[\HH=\HH_k\oplus \spn\{\mu_\kk:  |\kk|_1\le k-1\},\]
and the two components are closed and disjoint subspaces of $\HH$.\\

\underline{Step 4  ($\LL$-invariance, $\sigma(\LL|_{\HH_k})$ ):} 
The $\LL$-invariance of the finite dimensional combination of eigenfunctions $\spn\{\mu_\kk:  |\kk|_1\le k-1\}$ is evident. For every $\lambda\in\sigma(\LL)$ also the corresponding kernel $\ker \Pi_{\LL,\lambda}$ is $\LL$-invariant. Therefore the expression \eqref{ker_Pi} has to be $\LL$-invariant, since it is just a (finite) direct sum of $\LL$-invariant spaces.

Concerning the spectrum of $\LL$ in $\HH_k$ we recall that $\sigma(\LL|_{\ker \Pi_{\LL,\lambda}})= \sigma(\LL)\setminus\{\lambda\}$. Thus, we obtain from \eqref{ker_Pi} that $\sigma(\LL|_{\HH_k})=\{-\dd\cdot\kk:|\kk|_1\ge k\}$.
\end{proof}

After having established the subspaces $\HH_k$ we now turn to the semigroup which is generated by $\LL$.

\begin{lemma}\label{lem:L_generates}
The Fokker-Planck operator $\LL$ generates a $C_0$-semigroup of bounded operators in $\HH$, which is denoted by $(\e^{t\LL})_{t\ge 0}$.
\end{lemma}

\begin{proof}
From \eqref{comp_estimate} in the proof of Lemma~\ref{lem_comp_est} we find that for $\zeta= \frac 12(1+\beta^2+\tr \maD)$ the operator $(\LL-\zeta)|_{C_0^\infty(\R^n)}$ and thus $\LL-\zeta$ is dissipative. So we may apply the Lumer-Phillips Theorem (cf.~Theorem~1.4.3 in \cite{pazy}) which proves that $\LL-\zeta$ generates a $C_0$-semigroup of contractions, thus $\LL$ generates a $C_0$-semigroup of bounded operators in $\HH$.
\end{proof}

According to equation (1.2) in \cite{Metafune2001} the semigroup operators $\e^{t\LL}$ for $t>0$ are given by
\begin{equation}\label{semigr_1}
(\e^{t\LL}f)(\x)=\frac{\e^{t\tr \maD}}{(4\pi)^{n/2}\det \mathbf{Q}_t^{1/2}}\int_{\R^n}\exp\Big(-\frac 14\mathbf y^T \mathbf{Q}_t^{-1}\mathbf y\Big)f(\e^{t\maD}\x-\mathbf y)\md\mathbf y,
\end{equation}
where $\mathbf{Q}_t=(2\maD)^{-1}(\e^{2t\maD}-\mathbf{I})$. We can equivalently use the following representation in Fourier space, which is useful for the subsequent analysis.

\begin{lemma}
For $f\in\HH$ and $t\ge0$ there holds 
\begin{equation}\label{Fe_tL}
\F[\e^{t\LL}f](\bfxi)=\exp\big(-\bfxi^T [(2\maD)^{-1}(\mathbf{I}-\e^{-2t\maD})]\bfxi\big)\cdot\hat f(\e^{-t\maD}\bfxi).
\end{equation}
\end{lemma}

\begin{proof}
If $t=0$ the identity \eqref{Fe_tL} is obviously fulfilled, so we assume $t>0$ in the following. For $f\in\HH$, \eqref{semigr_1} is well defined, and we can write it as
\[(\e^{t\LL}f)(\x)=(4\pi)^{-n/2}(\det \mathbf{Q}_t) ^{-1/2}\e^{t\tr \maD}(\phi*f)(\e^{t\maD}\x),\]
where $\phi(\x)=\exp(-\frac 14\x^T \mathbf{Q}_t^{-1}\x)$. Using the fact that $\mathbf{Q}_t$ is diagonal we immediately obtain that $\hat\phi(\bfxi)=(\det 4\pi \mathbf{Q}_t)^{1/2}\exp(-\bfxi^T \mathbf{Q}_t \bfxi)$. With this we can write the Fourier transform of \eqref{semigr_1} as
\begin{align*}
\F[\e^{t\LL}f](\bfxi)&=(4\pi)^{-n/2}\det \mathbf{Q}_t^{-1/2}\e^{t\tr \maD} \int_{\R^n}(\phi*f)(\e^{t\maD}\x)\exp(-\ii\x\cdot\bfxi)\md\x\\
&= (4\pi)^{-n/2}\det \mathbf{Q}_t^{-1/2} \F[\phi*f](\e^{-t\maD}\bfxi)\\
&= \exp\big(-\bfxi^T [(2\maD)^{-1}(\mathbf{I}-\e^{-2t\maD})]\bfxi\big)\hat f(\e^{-t\maD}\bfxi).
\end{align*}
So \eqref{semigr_1} and \eqref{Fe_tL} are equivalent for all $f\in\HH$.
\end{proof}

In the next step we investigate the long-time behavior of $(\e^{t\LL})_{t\ge 0}$ on the subspaces $\HH_k$. In the subspaces $H_k$, the analogue of this analysis was presented in Theorem \ref{LinH}(vi). Its proof was elementary since the eigenfunctions $\{\mu_\kk:\kk\in\N_0^n\}$ form an orthogonal basis of $H$. But in $\HH$ the orthogonality of the eigenfunctions is lost, which hence requires more technical estimates of the semigroup. For the rest of this chapter they will be mostly based on the representation \eqref{Fe_tL} of $(\e^{t\LL})_{t\ge 0}$. 

\begin{proposition}
For every $k\in{\N_0}$ there exists a constant $C_k>0$ such that there holds
\begin{equation}\label{decay_etl}
\|\e^{t\LL}|_{\HH_k}\|_{\mathscr B(\HH)}\le C_k\e^{-tkc_1},\quad\forall t\ge 0,
\end{equation}
where $c_1$ is the smallest entry of $\dd$.
\end{proposition}
%%%%%%%%%%%%%%%%%%
\begin{proof}
We fix $k\in{\N_0}$ and take any $f\in\HH_k$. Our aim is to estimate $\nn\e^{t\LL}f\nn_\omega$. \\

\underline{Step 1  (pointwise estimates of $\hat f$):} 
$\hat f$ is analytic on $\Omega_{\beta/2}$,
 and since $f\in\HH_k$ we get due to \eqref{HH_kk_fourier} that $\hat f(\bfxi)=\mathcal O(|\bfxi|_2^k)$ as $|\bfxi|_2\to 0$. 
More precisely, its Taylor expansion with remainder in Lagrange form reads for all $\bfxi\in\Omega_{\beta/2}$:
 \[
  \hat f (\bfxi) = \sum_{|\kk|=k} \tfrac1{\kk !} \bfxi^\kk (\nabla^\kk_{\bfxi} \hat f)(\kappa\bfxi) ,\Xx{for some} \kappa\in[0,1] \,.
 \]
Lemma~\ref{uniB} provides a uniform bound of $|\nabla^\kk_{\bfxi} \hat f|$ on $\Omega_{\beta'/2}$, for $0<\beta'<\beta$.
Hence 
\begin{equation} \label{f-estimate}
 |\hat f (\mathbf z)| \leq C\ |\mathbf z|^k_2 \ \nn f\nn_\omega, \qquad \forall \mathbf z \in \Omega_{\beta'/2} \,.
\end{equation}

For estimating the semigroup \eqref{Fe_tL} in the norm $\nn\cdot\nn_\omega$ we shall need the following estimate for each $\ell\in\{1,\ldots,n\}$: For $t>1$ we have
$$
  \mathbf z:=\e^{-t\maD}\Big(\bfxi\pm\ii\frac\beta 2{\mathbf e}_\ell\Big)\in \Omega_{\beta'/2},\quad \forall\,\bfxi\in\R^n\,,
$$
with $\beta'=\e^{-c_1}\beta<\beta$. Hence, \eqref{f-estimate} yields for all $\bfxi\in\R^n$:
\begin{align}\label{f-estimate2}
\Big|\hat f\Big(\e^{-t\maD}\Big(\bfxi\pm\ii\frac\beta 2{\mathbf e}_\ell\Big)\Big)\Big|
& \le C \,\Big|\e^{-t\maD}\Big(\bfxi\pm\ii\frac\beta 2{\mathbf e}_\ell\Big)\Big|^{k}_2 \,\nn f\nn_\omega\\
& \le C \,\e^{-kc_1t}\,\Big|\bfxi\pm\ii\frac\beta 2{\mathbf e}_\ell\Big|^{k}_2 \,\nn f\nn_\omega. \nonumber
\end{align}
\smallskip

\underline{Step 2  (semigroup estimate):} 
For estimating \eqref{Fe_tL} we compute with \eqref{f-estimate2} for any $\ell\in\{1,\ldots,n\}$ and for $t>1$:
\begin{align}
&\Big\|\F[\e^{t\LL}f]\Big(\bfxi\pm\ii\frac\beta 2{\mathbf e}_\ell\Big)\Big\|_{L^2(\R^n_{\bfxi})}^2 = 
\int_{\R^n}\Big|\exp\Big[-\Big(\bfxi\pm\ii\frac\beta 2{\mathbf e}_\ell\Big)^T [(2\maD)^{-1}(\mathbf{I}-\e^{-2t\maD})]\nonumber\\
& \qquad\qquad\qquad\qquad\qquad\qquad\qquad\cdot\Big(\bfxi\pm\ii\frac\beta 2{\mathbf e}_\ell\Big)\Big]\Big|^2\,
\Big|\hat f\Big(\e^{-t\maD}\Big(\bfxi\pm\ii\frac\beta 2{\mathbf e}_\ell\Big)\Big)\Big|^2\md\bfxi\nonumber\\
&\qquad\qquad\qquad \le C\int_{\R^n}\exp\big(-\bfxi^T [\maD^{-1}(\mathbf{I}-\e^{-2t\maD})]\bfxi\big)\:
\Big|\hat f\Big(\e^{-t\maD}\Big(\bfxi\pm\ii\frac\beta 2{\mathbf e}_\ell\Big)\Big)\Big|^2\md\bfxi\nonumber\\
&\qquad\qquad\qquad \le C\int_{\R^n}\e^{-|\bfxi|_2^2\gamma_{\maD}}
\Big|\hat f\Big(\e^{-t\maD}\Big(\bfxi\pm\ii\frac\beta 2{\mathbf e}_\ell\Big)\Big)\Big|^2\md\bfxi\label{interm_sg_est}\\
&\qquad\qquad\qquad \le C\Big(\frac2\beta\Big)^{2k} \,\e^{-2kc_1t}\nn f\nn_\omega^2
\int_{\R^n}\e^{-|\bfxi|_2^2\gamma_{\maD}}
\Big|\bfxi\pm\ii\frac\beta 2{\mathbf e}_\ell\Big|^{2k}_2\md\bfxi
\nonumber\\
&\qquad\qquad\qquad = C'\Big(\frac2\beta\Big)^{2k} \,\e^{-2kc_1t}\, \nn f\nn_\omega^2,\nonumber
\end{align}
where $\gamma_{\mathbf C}:=(1-\e^{-2c_1})/c_1$. 

Summing \eqref{interm_sg_est} over all $\ell\in\{1,\ldots,n\}$ we conclude: There exists some $C>0$ such that for all $t>1$ there holds
\begin{equation}\label{semigroup-est}
\nn\e^{t\LL} f\nn_\omega \le C\e^{-kc_1 t}\, \nn f\nn_\omega,\qquad \forall f\in\HH_k\,.
\end{equation}
But since $(\e^{t\LL})_{t\ge 0}$ are bounded operators on $\HH$,
 uniformly for $0\le t\le1$ (cf.\ Lemma~\ref{lem:L_generates}) the above estimate \eqref{semigroup-est} holds true for all $t\ge 0$ with an appropriately large constant $C>0$. 
\end{proof}

With this proposition we conclude the proof of Theorem \ref{trm_sec_1}.

\subsection{The perturbed Fokker-Planck operator}\label{sec1.4}

Having defined the extension of the Fokker-Planck operator $\LL$ in $\HH$ we now turn to the investigation of the properties of the perturbed operator $\LL+\Theta$. Note that our $\x$-coordinates are such that $\D=\I$, and $\CC$ is diagonal, see the discussion in the beginning of Section \ref{sec2}. We make the following assumptions on $\Theta$:

\medskip
\begin{samepage}\label{pagge}
\noindent\textbf{(C) Conditions on $\boldsymbol\Theta$:} We assume that $\Theta f:=\vartheta * f$ for all $f\in\HH$, for some function $\vartheta:\R^n\to\R$. Thereby the convolution kernel $\vartheta$ has the following properties:
\begin{enumerate}
\renewcommand{\theenumi}{\roman{enumi}}\renewcommand{\labelenumi}{(\theenumi)}
\item The Fourier transform $\hat \vartheta$ can be extended to an analytic function in $\Omega_{\beta/2}$ (also denoted by $\hat\vartheta$), and $\hat\vartheta\in L^\infty(\Omega_{\beta/2})$.
\item There holds $\hat\vartheta(\mathbf 0)=0$, i.e.~$\vartheta$ is massless.
\item The function
\[
\bfxi\mapsto\int_0^1\frac1 s\hat\vartheta(\bfxi^T s^{\mathbf C})\md s
\]
is analytic in $\Omega_{\beta/2}$, and its real part lies in $L^\infty(\Omega_{\beta/2})$.
\end{enumerate}
\end{samepage}

\begin{lemma}\label{Theta_bdd}
Under the assumptions \textbf{(C)} the operator $\Theta$  has the following properties in $\HH$:
\begin{enumerate}\renewcommand{\theenumi}{\roman{enumi}}\renewcommand{\labelenumi}{(\theenumi)}
\item\label{bdd_i} $\Theta\in\mathscr B(\HH)$.
\item\label{Lempi} For every $k\in{\N_0}$ there holds $\Theta\colon \HH_k\to\HH_{k+1}$.
\end{enumerate}
\end{lemma}

\begin{proof} We start by proving \eqref{bdd_i}. Due to \textbf{(C)(i)} we have for every $f\in\HH$ that $\F[\Theta f]=\hat\vartheta\hat f $ is analytic in $\Omega_{\beta/2}$, and since $f$ satisfies \eqref{unif_l2} we find
\[
 \sup_{\substack{|\mathbf b|_1<\beta/2\\ \mathbf b\in\R^n}}\|\hat\vartheta\hat f(\cdot+\ii\mathbf b)\|_{L^2(\R^n)}<\infty.
\]
So, according to Proposition~\ref{prop3.4}, $\Theta$ maps $\HH$ into $\HH$. It remains to show it is bounded. To this end we use the norm $\nn\cdot\nn_\omega$, see \eqref{f_norm:d}. We start with the following computation, where $\ell\in \{1,\ldots,n\}$: 
\begin{align*}
\int_{\R^n}\Big|(\hat\vartheta\hat f)\Big(\bfxi\pm\ii\frac\beta 2{\mathbf e}_\ell\Big)\Big|^2\md\bfxi &= \lim_{b\nearrow \beta/2}\int_{\R^n}\big|(\hat\vartheta\hat f)(\bfxi\pm\ii b{\mathbf e}_\ell)\big|^2\md\bfxi \\
&\le \|\hat\vartheta\|^2_{L^\infty(\Omega_{\beta/2})}\lim_{b\nearrow \beta/2}\int_{\R^n}\big|\hat f(\bfxi\pm\ii b{\mathbf e}_\ell)\big|^2\md\bfxi \\
&=  \|\hat\vartheta\|^2_{L^\infty(\Omega_{\beta/2})}\int_{\R^n}\Big|\hat f\Big(\bfxi\pm\ii\frac\beta 2{\mathbf e}_\ell\Big)\Big|^2\md\bfxi.
\end{align*}
Thereby we have used \eqref{obeta_iii} in Proposition~\ref{prop3.4}. Note that $\hat\vartheta\hat f$ is the Fourier-transform of an element of $\HH$, and thus we may evaluate it at the boundary of $\Omega_{\beta/2}$ in the sense of $L^2$-functions.  We can repeat this estimate for every $\ell\in\{1,\ldots, n\}$ and conclude from \eqref{f_norm:d} that $\Theta$ is bounded in $\HH$ with a norm proportional to $\|\hat\vartheta\|_{L^\infty(\Omega_{\beta/2})}$.

Next we show \eqref{Lempi}. According to \eqref{HH_kk_fourier} $f$ lies in $\HH_k$ iff $\hat f$ has a zero of order greater or equal to $k$ at the origin. Now due to \textbf{(C)(ii)} $\hat\vartheta\hat f$ has a zero of order greater or equal to $k+1$ at the origin. Since $\Theta$ maps $\HH$ into $\HH$ (due to Result \eqref{bdd_i}) this shows that $\Theta \colon \HH_k\to\HH_{k+1}$.
\end{proof}

\begin{corollary}\label{cor43}
If $\Theta$ satisfies \textbf{(C)} then for every $k\in{\N_0}$ the space $\HH_k$ is invariant under $\LL+\Theta$.
\end{corollary}

\begin{proof}
This is a direct consequence of Proposition~\ref{prop_pi} and Lemma~\ref{Theta_bdd}~\eqref{Lempi} above.
\end{proof}

Throughout the rest of this section we always assume that $\Theta$ is such that the conditions \textbf{(C)} are satisfied in $\HH$ for some $\beta>0$. Now we fix this $\beta$ and consider $\HH$ with the corresponding weight function $\omega(\x)=\sum_{i=1}^n\cosh \beta x_i$. In the following we discuss properties of $\LL+\Theta$ in $\HH$, which then lead to the final theorem.

\begin{lemma}\label{sigma_sigma_p}
The spectrum $\sigma(\LL+\Theta)$ consists entirely of isolated eigenvalues.
\end{lemma}

\begin{proof}
According to Theorem~\ref{trm_sec_1}, $\LL$ generates a $C_0$-semigroup of bounded operators in $\HH$ and has a compact resolvent. Due to Lemma~\ref{Theta_bdd}~\eqref{bdd_i}, $\Theta$ is a bounded operator. Thus we can apply Proposition III.1.12 in \cite{engel}, which proves that $R_{\LL+\Theta}(\zeta)$ is compact for every $\zeta\in\rho(\LL+\Theta)$. 

It now remains to apply Theorem~III.6.29 in \cite{kato}, which proves that $\sigma(\LL+\Theta)$ consists entirely of isolated eigenvalues.
\end{proof}

In order to characterize the spectrum of $\LL+\Theta$ and the corresponding semigroup we introduce the operator $\Psi\colon \HH\to\HH\colon f\mapsto f*\psi$. Thereby $\psi$ is defined by
\[\hat\psi(\bfxi):=\exp\Big(\int_0^1\frac1 s\hat\vartheta(\bfxi^T s^{\mathbf C})\md s\Big).\] 
As we shall see below, $\Psi$ provides a similarity transformation between the resolvents of $\LL$ and $\LL+\Theta$.

\begin{lemma}\label{lem45}
$\Psi$ satisfies the following properties in $\HH$:
\begin{enumerate}
\renewcommand{\theenumi}{\roman{enumi}}
\renewcommand{\labelenumi}{(\theenumi)}
\item For every $k\in{\N_0}$ the operator $\Psi$ is a bijection from $\HH_k$ to $\HH_k$.
\item Both $\Psi$ and its inverse $\Psi^{-1}$ are bounded. Thereby $\Psi^{-1}f=\F^{-1}[\hat f/\hat \psi]$ for all $f\in\HH$. 
\end{enumerate}
\end{lemma}

\begin{proof}
For the moment we define the operator $\bar\Psi f:=\F^{-1}[\hat f/\hat \psi]$ for all $f\in\HH$, and show in the following that it is the inverse of $\Psi$.
To begin with we note that, due to the condition \textbf{(C)(iii)}, both $\hat\psi$ and $1/\hat\psi$ are analytic and uniformly bounded in $\Omega_{\beta/2}$. Thus it follows analogously to the proof of Lemma~\ref{Theta_bdd}~\eqref{bdd_i} that both $\Psi$ and $\bar\Psi$ are bounded operators in $\HH$.

Since $\hat \psi$ and $1/\hat\psi$ both do not have any zeros in $\Omega_{\beta/2}$, it follows from the characterization \eqref{HH_kk_fourier} of the space $\HH_k$ that $\Psi$ and $\bar\Psi$ map $\HH_k$ into itself for every $k\in{\N_0}$.

Finally we observe that for every $f\in\HH$ there holds $\Psi\bar\Psi f=\bar\Psi\Psi f=f$, which finally proves that $\bar \Psi=\Psi^{-1}$.
\end{proof}

\begin{proposition}\label{prop_46}
There holds
\begin{enumerate}
\renewcommand{\theenumi}{\roman{enumi}}
\renewcommand{\labelenumi}{(\theenumi)}
\item $\sigma(\LL+\Theta)=\sigma(\LL)$.
\item For every $\kk\in\N_0^n$ the function $f_\kk:=\Psi\mu_\kk$ is an eigenfunction of $\LL+\Theta$ to the eigenvalue $-\dd\cdot\kk$. Furthermore, for every $\zeta\in\sigma(\LL+\Theta)$
\[\ker(\zeta-(\LL+\Theta))=\spn\{f_\kk:-\dd\cdot\kk=\zeta\}.\]
\item The eigenfunctions $f_\kk$ satisfy $f_\kk=\nabla^\kk f_{\mathbf 0}$ for all $\kk\in\N_0^n$.
\end{enumerate}
\end{proposition}

\begin{proof}
Due to Lemma~\ref{sigma_sigma_p} we know that the spectrum of $\LL+\Theta$ consists entirely of eigenvalues. So, in order to determine the spectrum we look for $\zeta\in\C$ and non-trivial solutions $f\in\HH$ of $(\zeta-\LL-\Theta)f=0$. After applying the Fourier transform this equation reads
\[(\zeta+|\bfxi|_2^2)\hat f+\bfxi^T\maD\nabla_{\bfxi}\hat f=\hat\vartheta\hat f.\]
We now make the (non-restrictive) ansatz $\hat f=\hat p\hat\psi$. Note that due to \textbf{(C)(iii)} and $\hat\psi\neq 0$ in $\Omega_{\beta/2}$, the requirement $f\in\HH$ implies that $\hat p$ is analytic in $\Omega_{\beta/2}$. A short calculation shows that $\hat\psi\hat\vartheta=\bfxi^T\maD\nabla_{\bfxi}\hat \psi$. Using this, we obtain the following equation for $\hat p$:
\[(\zeta+|\bfxi|_2^2)\hat p+\bfxi^T\maD\nabla\hat p=0.\]
We find that this is exactly equation \eqref{f_traf:eval}. In the proof of Lemma~\ref{distr_spec_L} we have shown that 
%all analytic solutions of this equation are spanned by the family $\{\hat\mu_\kk\}_{\kk\in\N_0^n}$. So we can use the results from Lemma~\ref{distr_spec_L} and conclude that 
$0\not\equiv p\in\HH$ is a solution iff $\zeta\in\{-\dd\cdot\kk:\kk\in\N_0^n\}$. And for a fixed $\zeta\in\C$, $p \in\spn\{\mu_\kk:-\dd\cdot\kk=\zeta\}$. 
\end{proof}

Note that $\hat f_{\mathbf 0}{(\mathbf 0)}=\hat\psi(\mathbf 0)\hat  \mu_{\mathbf 0}{(\mathbf 0)}=1$, hence $f_\mathbf{0}$ has mass one.

\begin{proposition}
$\LL+\Theta$ generates a $C_0$-semigroup of bounded operators, $(\e^{t(\LL+\Theta)})_{t\ge0 }$. For every $k\in{\N_0}$ the space $\HH_k$ is invariant under the semigroup, and there exists some $\tilde C_k>0$ such that 
\[\|\e^{t(\LL+\Theta)}|_{\HH_k}\|_{\mathscr B(\HH_k)}\le \tilde C_k\e^{-tkc_1},\quad\forall t\ge 0.\] 
\end{proposition}

\begin{proof}
According to Proposition~\ref{prop_46} the eigenfunctions of $\LL$ and $\LL+\Theta$ are related by $f_\kk=\Psi\mu_\kk$, for every $\kk\in\N_0^n$. So we find for every $\zeta\notin\sigma(\LL)$ and $\kk\in\N_0^n$ that the resolvents satisfy
\[R_\LL(\zeta)\mu_\kk=\frac{1}{\zeta+\dd\cdot\kk}\mu_\kk=\Psi^{-1}\frac{1}{\zeta+\dd\cdot\kk}f_\kk=\Psi^{-1}R_{\LL+\Theta}(\zeta)\Psi\mu_\kk.\]
Since $\spn\{\mu_\kk:\kk\in\N_0^n\}\subset\HH$ is dense and all operators in the above formula are bounded, we conclude the following operator equality in $\HH$:
\begin{equation}\label{4.1}
\Psi R_\LL(\zeta)\Psi^{-1}=R_{\LL+\Theta}(\zeta).\end{equation}

Take any $k\in{\N_0}$. According to Corollary \ref{cor43} and Lemma~\ref{lem45} the identity \eqref{4.1} holds also in $\HH_k$, and $R_{\LL+\Theta}(\zeta)$ is a bounded operator in $\HH_k$. Now we apply the Hille-Yosida Theorem to the decay estimate for $(\e^{t\LL})_{t\ge0}$ stated in Theorem~\ref{trm_sec_1} \eqref{sec_1_v}. It shows that for all $m\in{\N_0}$ and $\Real\zeta>-kc_1$ there holds
\[\|R_\LL(\zeta)^m|_{\HH_k}\|_{\mathscr B(\HH_k)} \le \frac{C_k}{(\Real\zeta+kc_1)^m}, \]
where $C_k>0$ is the same constant as in \eqref{CK}. Applying this resolvent estimate to \eqref{4.1} yields for all $m\in{\N_0}$ and $\Real\zeta>-kc_1$:
\[\|R_{\LL+\Theta}(\zeta)^m|_{\HH_k}\|_{\mathscr B(\HH_k)} \le \frac{C_k\|\Psi\|_{\mathscr B(\HH_k)}\|\Psi^{-1}\|_{\mathscr B(\HH_k)}}{(\Real\zeta+kc_1)^m}.\]
Applying the Hille-Yosida Theorem again implies that $\LL+\Theta$ generates a $C_0$-semigroup of bounded operators, which satisfies the following estimate:
\[\|\e^{t(\LL+\Theta)}|_{\HH_k}\|_{\mathscr B(\HH_k)}\le \tilde C_k\e^{-tkc_1},\]
where $0<\tilde C_k\le C_k\|\Psi\|_{\mathscr B(\HH_k)}\|\Psi^{-1}\|_{\mathscr B(\HH_k)}$.
\end{proof}

We conclude this section by summarizing the main results.

\begin{theorem}\label{finalt}
Under the conditions \textbf{(C)} on $\Theta$, the perturbed Fokker-Planck operator $\LL+\Theta$ has the following properties in $\HH$:
\begin{enumerate}\renewcommand{\theenumi}{\roman{enumi}}\renewcommand{\labelenumi}{(\theenumi)}
\item $\sigma(\LL+\Theta)=\sigma(\LL)=\{-\dd\cdot\kk:\kk\in\N_0^n\}$, i.e.~$\LL+\Theta$ is an \emph{isospectral} deformation of $\LL$.
\item The functions $f_\kk:=\Psi\mu_\kk$ are eigenfunctions of $\LL+\Theta$ for all $\kk\in\N_0^n$. For every $\lambda\in\sigma(\LL+\Theta)$ the corresponding eigenspace is given by
\[\ker(\lambda-(\LL+\Theta))=\spn\{f_\kk:-\dd\cdot\kk=\lambda\}.\]
\item\label{finaliii} For every $k\in{\N_0}$, the operator $\LL+\Theta$ generates a $C_0$-semigroup $(\e^{t(\LL+\Theta)})_{t\ge 0}$ on $\HH_k$, and there exists some constant $\tilde C_k>0$ such that
\[\big\|\e^{t(\LL+\Theta)}|_{\HH_k}\big\|_{\mathscr B(\HH_k)}\le \tilde C_k\e^{-tkc_1},\quad\forall t\ge 0.\]
\end{enumerate}
\end{theorem}

In particular, this theorem implies exponential convergence of the solutions of the perturbed Fokker-Planck equation towards the stationary solution:

\begin{corollary}
Let $\varphi\in\HH$ be given, and let $f(t):=\e^{t(\LL+\Theta)}\varphi$ be the corresponding solution of \eqref{orig_equat}. Set $m:=\int_{\R^n} \varphi(\x)\md \x\in\C$. Then there exists a constant $C>0$ such that
\[\|f(t)-m f_{\mathbf 0}\|_\omega \le C\|\varphi-m f_{\mathbf 0}\|_\omega\e^{-tc_1},\quad\forall t\ge 0,\]
i.e.~$f(t)$ converges exponentially to $m f_{\mathbf 0}$.
\end{corollary}

\begin{proof}
Since $f_{\mathbf 0}$ is the unique normalized zero eigenfunction of $\LL+\Theta$ we obtain:
\[f(t)-m f_{\mathbf 0}= \e^{t(\LL+\Theta)}(\varphi-m f_{\mathbf 0}).\]
Since $\varphi-m f_{\mathbf 0}$ has zero mean, it follows from  Lemma~\ref{lemma48} that it lies in $\HH_1$. But $(\e^{t(\LL+\Theta)})_{t\ge 0}$ decays exponentially on $\HH_1$ with rate $-c_1$, see Theorem~\ref{finalt} \eqref{finaliii}. So we get for all $t\ge 0$:
\[\|f(t)-m f_{\mathbf 0}\|_\omega= \|\e^{t(\LL+\Theta)}(\varphi-m f_{\mathbf 0})\|_\omega \le \tilde C_1\|(\varphi-m f_{\mathbf 0})\|_\omega\,\e^{-tc_1}.\]
\end{proof}

\begin{remark}
Note that $\LL+\Theta$ is neither self-adjoint in $H$ nor in $\HH$. But the fact that $\sigma(\LL+\Theta)\subset\R$ and that $\LL+\Theta$ is only a ``deformation'' of $\LL$, see \eqref{4.1}, suggests that $\LL+\Theta$ is self-adjoint in an appropriate space. To verify this we introduce the inner product
\[\la f,g\ra_{\mathfrak H}:=\int_{\R^n}\frac 1\mu \Psi^{-1} f\cdot\overline{\Psi^{-1} g}\md\x,\]
and the corresponding norm $\|\cdot\|_{\mathfrak H}$. The associated space $\mathfrak H$ is the set of all functions such that $\|\cdot\|_{\mathfrak H}$ is finite. This is indeed a Hilbert space, and $\Psi$ is an isometry between $H$ and $\mathfrak H$. Using \eqref{4.1} we see the self-adjointness of $L+\Theta$ in $\mathfrak H$:
\begin{align*}
\la (L+\Theta)f, g\ra_{\mathfrak H} &=\la\Psi\circ L\circ\Psi^{-1}f, g\ra_{\mathfrak H}\\
    &= \la L(\Psi^{-1} f),\Psi^{-1}g\ra_H = \la\Psi^{-1}f, L(\Psi^{-1} g)\ra_H\\
    &=\la f, (L+\Theta) g\ra_{\mathfrak H}, 
\end{align*}
where we have used the self-adjointness of $L$ in $H$. In $\mathfrak H$ the eigenfunctions $f_\kk$ of $L+\Theta$ are orthogonal again (like the functions $\mu_\kk$ in $H$). Altogether, we conclude that $L$ in $H$ and $L+\Theta$ in $\mathfrak H$ are isometrically equivalent via the map $\Psi$. Hence, $L+\Theta$ inherits most properties of $L$. However, we point out that discovering the map $\Psi$, without the preceding analysis, is a non-trivial issue.

Furthermore, the Hilbert space $\mathfrak H$ is difficult to be characterized explicitly. In particular, it is usually not possible to describe $\mathfrak H$ as a weighted $L^2$-space. A simple calculation shows that $\mathfrak H=L^2(\nu)$ for some weight function $\nu$ only if for all $f\in C_0^\infty(\R^n)$ there holds
\[\nu=\frac 1f\cdot \Big(\frac 1\psi\Big)*\Big(\frac{(1/\psi)*f}{\mu}\Big).\]
But in general this function $\nu$ will not be independent of $f$.
\end{remark}

\appendix

\section{Results in functional analysis and deferred proofs}

On $\Omega=\R^n$ it is possible to find compact embeddings of weighted Sobolev spaces into weighted $L^2$-spaces if certain conditions on the weight functions are satisfied. Here we need the following corollary from Theorem~2.4 in \cite{Opic1989}:
\begin{lemma}\label{comp_emb}
Let $v,w$ be two weight functions on $\Omega=\R^n$. Assume further that 
\begin{equation}\label{opic_ass}
\lim_{r\to\infty}\esssup_{\x\in B_r^2(\mathbf 0)^c}\frac{w(\x)}{v(\x)}=0.
\end{equation}
Then there holds the compact embedding $H^1(v,w)\hookrightarrow\hookrightarrow L^2(w)$.
\end{lemma}

%\begin{proof}
%Since the weights are uniformly bounded from below we have $H^1(B_r^2(\mathbf 0); v,w)=H^1(B_r^2(\mathbf 0))$ and $L^2(B_r^2(\mathbf 0); w)=L^2(B_r^2(\mathbf 0))$. Due to the Rellich-Kondrachov theorem (cf.~Theorem 6.2 in \cite{adams1975}) we know that the compact embedding $H^1(B_r^2(\mathbf 0))\hookrightarrow\hookrightarrow L^2(B_r^2(\mathbf 0))$ holds for all $r>0$. In order to verify the assumptions of Theorem 2.4 in \cite{Opic1989} we evaluate for $\|u\|_{v,w}\le 1$:
%\begin{align*}
%\|u\|_{B_r^2(\mathbf 0)^c;w}^2 & \le \int_{B_r^2(\mathbf 0)^c}|u(\x)|^2 v(\x)\esssup_{\x\in B_r^2(\mathbf 0)^c}\frac{w(\x)}{v(\x)}\md \x\\
%&\le \esssup_{\x\in B_r^2(\mathbf 0)^c}\frac{w(\x)}{v(\x)}.
%\end{align*}
%According to the assumption \eqref{opic_ass} the right hand side tends to zero as $r\to\infty$. Thus the condition (2.13) in \cite{Opic1989} is satisfied, and this proves the desired embedding.
%\end{proof}

\begin{lemma}\label{dens_test_fun_2}
Let $\nu$ be a weight function on $\R^n$. Then $C_0^\infty(\R^n)$ is dense in $L^2(\nu)$. 
\end{lemma}

The proof of the above lemma is straightforward, see \cite{thesis} for more details.

\begin{lemma}\label{lem:A4}
There exists a constant $C>0$ such that for every $f\in\HH$ we have
\[|\nabla^\kk\hat f(\mathbf 0) |\le C\|f\|_\omega,\quad \kk\in\N_0^n.\]
\end{lemma}

\begin{proof}
We have
\begin{align*}
|\nabla^\kk\hat f(\mathbf 0) |& \le \|\nabla^\kk\hat f\|_{L^\infty(\R^n)} = \|\F[\x^\kk f(\x)]\|_{L^\infty(\R^n)}\le \|\x^\kk f(\x)\|_{L^1(\R^n)}\\
&= \int_{\R^n}|f(\x)|\omega(\x)^{1/2}\cdot \big|\x^\kk\omega(\x)^{-1/2}\big|\md\x \le \|f\|_\omega\Big(\int_{\R^n}\x^{2\kk}\omega(\x)^{-1}\md\x\Big)^{1/2}.
\end{align*}
Since $\omega(\x)^{-1}$ decays exponentially as $|\x|_2\to \infty$ the last integral on the right hand side is finite.
\end{proof}

\begin{comment}
\begin{lemma}\label{uniB}
Let $f\in\HH$. Then for every $0<\beta'<\beta$ there holds
\[
  \sup_{\mathbf z\in\Omega_{\beta'/2}}|\hat f(\mathbf z)|\le C(\beta')\,\|f\|_\omega\,.
%  \sup_{\mathbf z\in\Omega_{\beta'/2}}|\hat f(\mathbf z)|<\infty.
\]
\end{lemma}

\begin{proof}
Using the result \eqref{obeta_ii} in Proposition~\ref{prop3.4} and the continuity of the Fourier transform from $L^1(\R^n)$ into $L^\infty(\R^n)$ we shall show that $\|f(\x)\e^{\mathbf b\cdot\x}\|_{L^1(\R^n_\x)}$ is uniformly bounded for $|\mathbf b|_1<\beta'/2$. So we compute
\begin{align*}
\big\|f(\x)\e^{\mathbf b\cdot\x} \big\|_{L^1(\R^n_\x)} & =\int_{\R^n}|f(\x)|\e^{\mathbf b\cdot\x}\md \x
\le \|f\|_\omega\Big(\int_{\R^n}\frac{\e^{2\mathbf b\cdot\x}}{\omega(\x)}\md\x\Big)^{\frac12}.
\end{align*}
For the last integral we apply the (discrete) H\"older inequality and obtain:
\begin{align*}
\int_{\R^n}\frac{\e^{2\mathbf b\cdot\x}}{\omega(\x)}\md\x &\le \int_{\R^n}\frac{\e^{\beta'|\x|_\infty}}{\omega(\x)}\md\x \le 2\int_{\R^n}\frac{\omega\big(\frac{\beta'}{\beta}\x\big)}{\omega(\x)}\md\x =:C(\beta')^2 <\infty,
\end{align*}
which now is a bound independent of $\mathbf b$.
\end{proof}
\end{comment}

\begin{lemma}\label{uniB} % generalization of preceding lemma; edited by Franz on Friday 5th of June 2015
For every $0<\beta'<\beta$ and $\kk\in \N_0^n$,
 there exists a positive constant $C$ such that
 \[
  \sup_{\mathbf z\in\Omega_{\beta'/2}}|\nabla^\kk \hat f(\mathbf z)|\le C\,\|f\|_\omega, \quad \forall f\in\HH \,.
 \]
\end{lemma}

\begin{proof}
Due to Proposition~\ref{prop3.4},
 all functions $f\in\HH$ satisfy $\hat f (\bfxi + \ii \mathbf b) = \F [f(\x)\e^{\mathbf b\cdot\x}](\bfxi)$
 for $|\mathbf b|_1<\beta/2$.
Hence,
 \[
 ( \nabla^\kk \hat f )(\bfxi + \ii \mathbf b) =\F[(-\ii\x)^\kk f(\x)](\bfxi + \ii \mathbf b) = \F [(-\ii \x)^\kk f(\x)\e^{\mathbf b\cdot\x}](\bfxi)
 \]
 follows for $|\mathbf b|_1<\beta'/2$ and $\kk\in\N_0^n$.
Then,
\begin{align*}
 \sup_{\mathbf z\in\Omega_{\beta'/2}}|\nabla^\kk \hat f(\mathbf z)|
%  &\leq \sup_{\mathbf b : |\mathbf b|_1<\beta'/2} \|\nabla^\kk_{\bfxi} \, \hat f (\bfxi + \ii \mathbf b)\|_{L^\infty(\R^n_{\bfxi})}
  &\leq  \sup_{ |\mathbf b|_1<\beta'/2} \| \F [(-\ii \x)^\kk f(\x)\e^{\mathbf b\cdot\x}](\bfxi)\|_{L^\infty(\R^n_{\bfxi})} \\
  &\leq \sup_{|\mathbf b|_1<\beta'/2} \|\x^\kk f(\x)\e^{\mathbf b\cdot\x}\|_{L^1(\R^n_{\x})} \\
  &\leq \sup_{|\mathbf b|_1<\beta'/2} \big\|\tfrac{\x^{\kk} \e^{\mathbf b\cdot\x}}{\sqrt{\omega(\x)}}\big\|_{L^2(\R^n_{\x})} \|f\|_{\omega} \,.
\end{align*}
The norm $\|\tfrac{\x^{\kk} \e^{\mathbf b\cdot\x}}{\sqrt{\omega(\x)}}\|^2_{L^2(\R^n_{\x})}$ can be estimated as 
 %we apply the (discrete) H\"older inequality and obtain:
 \begin{align*}
  \int_{\R^n}\frac{\x^{2\kk} \e^{2\mathbf b\cdot\x}}{\omega(\x)}\md\x
    &\le \int_{\R^n}\frac{\x^{2\kk} \e^{\beta'|\x|_\infty}}{\omega(\x)}\md\x
     \le 2\int_{\R^n}\frac{\x^{2\kk} \omega\big(\frac{\beta'}{\beta}\x\big)}{\omega(\x)}\md\x =:C^2 <\infty,
 \end{align*}
 where $C$ is finite due to $0<\beta'<\beta$. %, and uniform in $\mathbf b$.
Thus, the estimate $\sup_{\mathbf z\in\Omega_{\beta'/2}}|\nabla^\kk \hat f(\mathbf z)|\le C\,\|f\|_\omega$ for all $f\in\HH$ follows.
\end{proof}

\begin{lemma}\label{lem:funct}
Let $X\hookrightarrow \mathcal X$ be Hilbert spaces, and $\psi_0,\ldots,\psi_{k-1}\in\mathscr B(\mathcal X,\C)$ be linearly independent functionals. Then $\tilde \psi_j:=\psi_j|_X\in\mathscr B(X,\C)$ for all $0\le j\le k-1$, and \[\bigcap_{j=0}^{k-1}\ker\psi_j=\cl_\mathcal X\bigcap_{j=0}^{k-1}\ker\tilde\psi_j.\]
\end{lemma}

This result coincides with Lemma~C.2 in \cite{StAr14}. The proof can be found in therein.

\begin{lemma}\label{lem:proj}
Consider two Hilbert spaces $X\hookrightarrow \mathcal X$ and a projection $\Rho_\mathcal X\in\mathscr B(\mathcal X)$, such that $\Rho_X:=\Rho_\mathcal X|_X\in\mathscr B(X)$. Then $\ran\Rho_{\mathcal X}=\cl_{\mathcal X}\ran\Rho_X$ and $\ker\Rho_{\mathcal X}=\cl_{\mathcal X}\ker\Rho_X$.
\end{lemma}

This result coincides with Lemma~C.1 in \cite{StAr14}.

%--------------------------------------------acknowledgments:
%\vspace{0.5cm} \indent {\it
%A\,c\,k\,n\,o\,w\,l\,e\,d\,g\,m\,e\,n\,t\,s.\;} The authors
%acknowledge support by the \ldots.

\end{document}